%% file: online-eit.tex
\pdfoutput=1
\documentclass[a4paper,english]{jnsao}
\usepackage[utf8]{inputenc}
\usepackage[T1]{fontenc}
\usepackage{csquotes}
\usepackage{babel}
\usepackage{tikz}
\usepackage{float}
\usepackage{algpseudocode}
\usepackage{enumitem}
\usepackage{comment}
\usepackage{mathtools}
\usepackage[nameinlink,capitalize]{cleveref}
\usepackage[textsize=footnotesize,color=DarkOrange!40]{todonotes}
\usepackage{pgfplots}
\usepackage{subcaption}
\usepackage{multirow,graphicx}
\usepackage{changes-simple}

\usepgfplotslibrary{colorbrewer}
\def\ignorelegendentry#1{}
\pgfplotsset{
    every axis label/.append style = {font = \scriptsize},
    every tick label/.append style = {font = \scriptsize},
    ignore legend/.style={
        every axis legend/.code={\let\addlegendentry\ignorelegendentry}
    },
    log x ticks with fixed point/.style={
        xticklabel={
            \pgfkeys{/pgf/fpu=true}
            \pgfmathparse{exp(\tick)}%
            \pgfmathprintnumber[fixed relative, precision=3]{\pgfmathresult}
            \pgfkeys{/pgf/fpu=false}
        },
    },
    log y ticks with fixed point/.style={
        yticklabel={
            \pgfkeys{/pgf/fpu=true}
            \pgfmathparse{exp(\tick)}%
            \pgfmathprintnumber[fixed relative, precision=3]{\pgfmathresult}
            \pgfkeys{/pgf/fpu=false}
        },
    },
}

\input{symbols}

\input{symbols-texonly}

\input{online-eit-macros}

\manuscriptcopyright{}
\manuscriptlicense{}
\manuscriptstatus{Manuscript}
\manuscripteprint{2412.12944}
\manuscripteprinttype{arxiv}

\theoremstyle{definition}
\newtheorem{assumption}[theorem]{Assumption}
\crefname{assumption}{Assumption}{Assumptions}
\floatstyle{ruled}
\floatname{algorithm}{Algorithm}
\newfloat{algorithm}{tbp!}{loa}
\crefname{algorithm}{Algorithm}{Algorithms}

\algrenewcommand{\algorithmiccomment}[1]{\hfill$\rightsquigarrow$\ {\itshape #1}}

\date{2024-12-17 (revised 2025-03-17)}

\author{
    Neil Dizon\thanks{%
        School of Mathematics and Statistics, University of New South Wales, Sydney, Australia.
        \email{n.dizon@unsw.edu.au}, \orcid{0000-0001-8664-2255}
    }
    \and
    Jyrki Jauhiainen\thanks{%
        Department of Technical Physics, University of Eastern Finland, Finland,
        \email{jyrki.jauhiainen@uef.fi}, \orcid{0000-0001-6711-6997}
    }
    \and
    Tuomo Valkonen\thanks{%
        Research Center in Mathematical Modeling and Optimization (MODEMAT), Quito, Ecuador
        \emph{and}
        Department of Mathematics and Statistics, University of Helsinki, Finland,
        \email{tuomo.valkonen@iki.fi}, \orcid{0000-0001-6683-3572}
    }
}
\shortauthor{Dizon, Jauhiainen, Valkonen}
\title{Online optimisation for dynamic electrical impedance tomography}
\shorttitle{Online optimisation for dynamic EIT}

\acknowledgements{%
    This research has been supported by the Research Council of Finland grants 338614, 314701, 353088, and 358944.%
}

\makeatletter
 \@ifundefined{@the@H@page}{%
    \let\@the@H@page\relax
}{}
\makeatother

\begin{document}

\maketitle

\begin{abstract}
    Online optimisation studies the convergence of optimisation methods as the data embedded in the problem changes. Based on this idea, we propose a primal dual online method for nonlinear time-discrete inverse problems. We analyse the method through regret theory and demonstrate its performance in real-time monitoring of moving bodies in a fluid with Electrical Impedance Tomography (EIT). To do so, we also prove the second-order differentiability of the Complete Electrode Model (CEM) solution operator on $L^\infty$.
\end{abstract}

\section{Introduction}
\label{sec:intro}
Electrical impedance tomography (EIT) is an imaging technique for inferring the electrical conductivity distribution within a body through boundary currents and potentials. While measurements in EIT can be performed in real time, reconstructing images from the data is computationally intensive. This challenge is critical in applications such as real-time monitoring of industrial processes, where immediate feedback is essential -- for instance, in detecting blockages or leaks in pipelines.

Traditionally, inverse problems, including those arising in EIT, have been studied in a static context, where robust theoretical foundations and solution methods are available. However, the need for real-time reconstructions in dynamic settings has grown significantly \cite{lechleiter2018dynamic,holland2010reducing,hunt2014weighing,tuomov-phaserec,lipponen2011nonstationary}. Addressing this demand requires novel approaches capable of efficiently processing large data sets and capturing time-dependent changes in the imaged domain.

To this end, we introduce online optimization methods tailored to time-discrete nonlinear inverse problems, formulated as the conceptual problem
\begin{equation}
    \label{eq:intro:problem}
    \min_{(x^0,x^1,x^2,\ldots) \in \PpredictConstr}~ \sum_{k=0}^\infty J_k(\thisx)
\end{equation}
for a set $\PpredictConstr \subset \prod_{k=0}^\infty X_k$ that encodes the temporal couplings between the frame-wise variables $x^k$ over time index $k$.
With Tikhonov-type regularization, the frame-wise objective can be expressed as
\begin{equation}
    \label{eq:intro:j}
    J_k(x) \defeq E(A_k(x) - b_k) + R(x),
\end{equation}
where $R$ is a regularization term, $A_k$ a nonlinear forward operator, $b_k$ the measurement data, and the data fidelity $E$ models noise.

In this work, we focus on EIT process monitoring using isotropic total variation regularization, $R(x)=\alpha\norm{Dx}_{2,1}$.
Here $A_k$ represents the solution operator for the complete electrode model (CEM) \cite{cheng1989electrode}. For given electrode potentials, it maps the electrical conductivity $x$ on a domain $\Omega \subset \R^n$ to boundary current measurements.
The temporal couplings encoded by $\PpredictConstr$ arise from a partial differential equation (PDE) that models matter movement within the body.
For simplicity in our numerical demonstrations, we employ the incompressible transport equation, though more sophisticated models, such as the Navier–Stokes equations, could be applied.

The problem \eqref{eq:intro:problem} is formal: we cannot in practise solve an optimisation problem for an unbounded time segment. At each instant $N$, we can access at most the initial segment $\{(A_k, b_k)\}_{k=0}^N$ of forward operators and data.
It is conceivable to solve the corresponding finite horizon problems
$
    \min_{x^{0:N} \in \PpredictConstr_{0:N}}~ \sum_{k=0}^N J_k(\thisx)
$
where we use the slicing notation
\begin{equation}
    \label{eq:intro:slicing}
    x^{0:N}=(x^0,\dots,x^N)
    \quad\text{and}\quad
    \PpredictConstr_{0:N} \defeq \{ x^{0:N} \mid (x^0,x^1,\ldots) \in \PpredictConstr\}.
\end{equation}
However, even this is numerically unwieldy for large $N$, and hardly real-time with standard optimisation methods. In practise, in a long monitoring process, both CPU and memory requirements would also force us to work with a short time window of data.

\term{Online optimisation} \cite{zinkevich2003online} attempts to solve \eqref{eq:intro:problem} in real time. The idea in most methods is to take a \term{single} step of a standard optimisation method for $\min J_k$ at each index $k$. For introductions we refer to \cite{hazan2016introduction,belmega2018online,orabona2020modern,simonetto2020time}.
Basic online optimisation methods take $\PpredictConstr$ in \eqref{eq:intro:problem} to constrain $x^0=x^1=x^2=\ldots$, i.e., do not consider problems that evolve over time, only data that arrives gradually.
\term{Dynamic online optimisation} methods \cite{hall13dynamical,tuomov-predict,zhang2019distributed,chang2020unconstrained,zhang2022regrets,nonhoff2020online,tang2022running, bastianello2020primal,zhang2021online,bernstein2019online,chen2022online} typically intersperse optimisation steps with \term{prediction} steps for some prediction operators $W_k$ that attempt to model more general temporal constraint sets $\PpredictConstr$ subject to available, possibly noisy and corrupted, information.
In particular, a dynamic online forward-backward (a.k.a.~proximal gradient) method for \eqref{eq:intro:problem}\&\eqref{eq:intro:j} iterates
\begin{equation}
    \label{eq:intro:fb}
    \thisx \defeq \prox_{\tau R}(\this{\primalpredict} - \tau \grad A_k(\this{\primalpredict})^*(A_k(\this{\primalpredict})-b_k))
    \quad\text{with the predictions}\quad
    \nexxt{\primalpredict} \defeq W_k(\thisx).
\end{equation}

Most earlier works on dynamic online optimisation concentrate on such forward-backward methods.
They have limited applicability to inverse problems with total variation regularisation $R = \norm{\grad \freevar}_{2,1}$, as the proximal operator $\prox_{\tau R}(x) \defeq \argmin_z \frac{1}{2}\norm{z-x}^2 + \tau R(z)$ is expensive to calculate: it corresponds to total variation denoising.
Therefore, we developed in \cite{tuomov-predict,better-predict} primal-dual dynamic online optimisation methods for \eqref{eq:intro:problem} with a linear $A_k$ in \eqref{eq:intro:j}.
The initial theory in \cite{tuomov-predict} imposed severe restrictions on the dual component of the predictor. These were relaxed in \cite{better-predict}, where several improved dual predictors were developed, for the primal predictor based on optical flow.

The main results of \cite{tuomov-predict} could also only be interpreted through regularisation theory, as the \term{regret} that we were able to prove was non-symmetric.
Indeed, \term{convergence} results are rarely available for online optimisation methods. Instead, one attempts to bound the \term{regret} of past updates with respect to all information available up to an instant $N$. For \eqref{eq:intro:fb}, one can bound the \term{dynamic regret} \cite{hall13dynamical}
\[
    \dynregret(x^{0:N}) = \sup_{\optx^{0:N} \in \PpredictConstr_{0:N}} \sum_{k=0}^N \left( J_k(x^k) -  J_k(\optx^k)\right).
\]
This may be negative, if the \term{comparison set} $\PpredictConstr_{0:N}$ does not include all the possible “paths” $x^{k:N}$ that the iterates generated by the algorithm may take.
For the primal-dual method of \cite{tuomov-predict}, the result was even weaker: we could only bound
\[
    \sup_{\optx^{0:N} \in \PpredictConstr_{0:N}} \left(\breve J_{0:N}(x^{0:N}) -  \sum_{k=0}^N J_k(\optx^k)\right)
\]
for a temporal infimal convolution $\breve J_{0:N}$ between the comparison set and the framewise objectives.
In \cite{better-predict}, at the cost of having to reduce the size of the comparison set, we managed to symmetrise these results to a bound on
\[
    \mathring J_{0:N}(x^{0:N}) - \sup_{\optx^{0:N} \in \PpredictConstr_{0:N}} \mathring J_{0:N}(\optx^{0:N})
\]
for a different temporal (sub-)infimal convolution $\mathring J_{0:N}$.

The theory in \cite{better-predict} was still for convex functions.
In \cref{sec:pd}, we will
\begin{enumerate}[label=(\alph*),nosep]\itshape
    \item extend the primal-dual method of \cite{better-predict} to non-convex objectives, including \eqref{eq:intro:problem}\&\eqref{eq:intro:j} with non-linear forward operators $A_k$, and
    \item allow for inexact gradients, in particular, inexact $A_k$ and its Jacobian.
\end{enumerate}
The non-convexity adds significant new technical challenges to the proof.
The inexact computations help to reduce the high expense of solving PDEs and can, fortunately, be incorporated with much less effort.
While we do not yet treat interweaved PDE solvers as \cite{jensen2022nonsmooth} did for static problems, that is our ultimate goal.
Instead, we do delayed, intermittent solves of the PDEs “in the background” to obtain real-time performance.
To prove that the EIT problem satisfies the conditions of the online method, in \cref{sec:eit},
\begin{enumerate}[resume*]\itshape
    \item  we will prove the second-order differentiability of the CEM solution operator for $L^\infty$ conductivies.
\end{enumerate}

In the final \cref{sec:numerical}, we evaluate the proposed method numerically on (nearly) real-time EIT reconstruction.
Our approach to EIT process monitoring differs from the recent work in \cite{alsaker2023multithreaded}, where the authors achieved real-time performance using extensive multithreading and code optimisation with the D-bar method \cite{nachman1996global, siltanen2000implementation}. In contrast, our online method attains real-time performance thanks to the inherently light computational cost of iterations of our algorithm, albeit dependent on meticulous predictor design.
Finally, our optimisation-based approach is not the only possibility treat dynamic inverse problems in an online fashion. Especially if uncertainty quantification is required, full Bayesian modelling can be used with numerical techniques such as non-linear Kalman filters and Markov-Chain Monte Carlo \cite{lan2023bayesian,gerber2021online,albers2019ensemble}.
Except in the simplest fully Gaussian cases (which would exclude a total variation prior), this can, however, come with much higher computational resource demands and complexity.

\paragraph{Additional notation}

In the slicing notation \eqref{eq:intro:slicing}, we allow $N=\infty$ with $x^{k:\infty}=(x^{k}, x^{k+1}, \ldots)$ and also set $\PpredictConstr_{k} \defeq \PpredictConstr_{k:k}$.
We write $\linear(X; Y)$ for the space of bounded linear operators between normed spaces $X$ and $Y$, and $\Id \in \linear(X; X)$ for the identity operator.
The block-diagonal operator consisting of $M$ and $\Gamma$ reads $\diag\bigl(M,\; \Gamma\bigr)$.
When $X$ is Hilbert, we abbreviate $\iprod{x}{y}_M = \iprod{Mx}{y}$ for $M \in \linear(X; X)$.
The notation $\fakenorm{x}_{M}^2$ mimics norm notation when $M$ may not be positive semi-definite. If it is, written $M \ge 0$, we set $\norm{x}_M \defeq \sqrt{\iprod{x}{x}_M}$.
For a Borel $g: \R^n \supset \Omega \to \R^n$, we set $\norm{g}_{2,1} \defeq \int_\Omega \norm{g(\xi)}_2 \d \xi$. Applied to the gradient of a differentiable function, this produces the isotropic total variation.

For $A \subset X$ and $x \in X$, we write $\iprod{A}{x} = \{\iprod{z}{x} \mid z \in A\}$.
For any $B \subset \R$ (especially $B=\iprod{A}{x}$), $B \ge 0$ means $t \ge 0$ for all $t \in B$.
For a convex function $F: X \to \extR$, $\subdiff F$ denotes the subdifferential, and $F^*$ the Fenchel conjugate.
The $\{0,\infty\}$-valued indicator function of $A$ is written $\delta_A$.
We refer to \cite{clasonvalkonen2020nonsmooth} for more details on convex analysis.

\section{Online primal-dual proximal splitting for nonconvex problems}
\label{sec:pd}

We consider the conceptual problem
\begin{equation}
    \label{eq:pd:problem}
    \min_{(x^0,x^1,x^2,\ldots) \in \PpredictConstr} \sum_{k=0}^\infty F_k(x^k) + E_k(x^k) + G_k(K_k x^k),
\end{equation}
where $G_k$ and $F_k$ are convex and possibly nonsmooth, but $E_k$ is smooth but possibly nonconvex.
The operator $K_k$ is linear and bounded. The \term{primal comparison set} $\PpredictConstr$ encodes temporal couplings between the framewise variables $x^k$.
This clearly includes \eqref{eq:intro:problem}\&\eqref{eq:intro:j}, with non-linear forward-operators $A_k$ accommodated by setting $E_k(x) = E(A_k(x) - b_k)$.

We present in \cref{alg:pd:alg} our proposed primal-dual online method for \eqref{eq:pd:problem}.
It involves the Fenchel conjugate $G_k^*$ of $G_k$ and a dual variable $y^k$ introduced through writing $G_k(Kx) = \sup_{y} \iprod{K_kx}{y} - G_k^*(y)$, as well as the predictor $P_k$ and the predictions
\[
    (\nexxt{\breve x}, \nexxt{\breve y}) \defeq P_k(\thisx, \thisy).
\]
We have also replaced $\grad E_k(\this{\primalpredict})$ by an estimate $\estgrad_k(\this{\primalpredict})$. This allows for the inexact computation of $A_k$ and $\grad A_k$, as discussed in the Introduction.

To derive the algorithm, note that a single term in \eqref{eq:pd:problem} is minimised when we solve
\begin{equation}
    \label{eq:pd:problem:static-pd}
    \min_{x}\max_y~ F_k(x) + E_k(x) + \iprod{K_kx}{y} - G_k^*(y).
\end{equation}
Taking forward-backward steps alternatingly with respect to $x$ and $y$, we obtain the primal step of the algorithm, as well as the dual step subject to a small modification. The predictor transfers iterates generated this way between time frames and spaces.

In \cref{sec:pd:assumptions}, we provide essential definitions and outline the formal assumptions needed to prove regret bounds for the method in the subsequent \cref{sec:pd:regret}.

\begin{algorithm}
    \caption{Nonconvex predictive online primal-dual proximal splitting (POPD-N).}
    \label{alg:pd:alg}
    \begin{algorithmic}[1]
        \Require For all $k \in \N$, on Hilbert spaces $X_k$ and $Y_k$, convex, proper, lower semi-continuous  $E_{k+1}: X_{k+1} \to \R$, $F_{k+1}: X_{k+1} \to \extR$ and $G_{k+1}^* : Y_{k+1} \to \extR$, primal-dual predictor $P_k: X_k \times Y_k \to X_{k+1} \times Y_{k+1}$, and $K_{k+1} \in \linear(X_{k+1}; Y_{k+1})$.
        Estimates $\estgrad_{k+1}(\nexxt{\primalpredict})$ of the gradients $\grad E_{k+1}(\nexxt{\primalpredict})$.
        Step length parameters $\tau_{k+1},\sigma_{k+1}>0$.
        \State Pick initial iterates $x^0 \in X_0$ and $y^0 \in Y_0$.
        \For{$k \in \N$}
        \State\label{item:alg:pd:predict} $(\nexxt\primalpredict, \nexxt\dualpredict) \defeq P_k(\thisx, \thisy)$
        \Comment{prediction step}
        \State\label{item:alg:pd:primal}$\nextx \defeq \prox_{\tau_{k+1} F_{k+1}}(\nexxt{\primalpredict} - \tau_{k+1}\estgrad_{k+1}(\nexxt{\primalpredict}) - \tau_{k+1} K_{k+1}^*\nexxt{\dualpredict})$
        \Comment{primal step}
        \State\label{item:alg:pd:dual}$\nexty \defeq \prox_{\sigma_{k+1} G_{k+1}^*}(\nexxt{\dualpredict} + \sigma_{k+1} K_{k+1}(2\nextx-\nexxt{\primalpredict}))$
        \Comment{dual step}
        \EndFor
    \end{algorithmic}
\end{algorithm}

\subsection{Assumptions}
\label{sec:pd:assumptions}

Write $u^k=(x^k, y^k)$. Then $u^{0:\infty}$ generated by \cref{alg:pd:alg} from an initial $u^0$, lies in
\[
    \ProductSpace_{0:\infty} \defeq (X_0 \times Y_0) \times (X_1 \times Y_1) \times (X_2 \times Y_2) \cdots.
\]
Since \cref{alg:pd:alg} involves the dual variable $y^k$, we need to expand the set $\PpredictConstr \subset X_{0:\infty}$ of primal comparison sequences in \eqref{eq:pd:problem} to $\PDpredictConstr \subset \ProductSpace_{0:\infty}$. We then define the sets of primal and dual comparison sequences by projection as
\[
    \PpredictConstr \defeq \left\lbrace \optx^{0:\infty} \in X_{0:\infty} \mid (\bar x^{0:\infty},\bar y^{0:\infty}) \in \PDpredictConstr\right\rbrace
    \quad\text{and}\quad
    \DpredictConstr \defeq \left\lbrace \bar y^{0:\infty} \in Y_{0:\infty}  \mid  (\bar x^{0:\infty},\bar y^{0:\infty}) \in \PDpredictConstr\right\rbrace.
\]
To algorithmically track sequences in these sets, we need the primal-dual predictor $P_k$.
We next formalise this predictor, the comparison sequences, and the functions of \eqref{eq:pd:problem}.

\begin{assumption}[Basic structural assumptions]
    \label{ass:pd:main}
    Given $N \in \N$, on Hilbert spaces $X_k$ and $Y_k$, ($0 \le k \le N$), we are provided with:
    \begin{enumerate}[label=(\roman*),nosep]
        \item For all $0 \le k \le N$, convex, proper, and lower semi-continuous $F_k: X_k \to \extR$, and $G_k^*: Y_k \to \extR$ (with respective strong convexity factors $\gamma_k \ge 0$ and $\rho_k \ge 0$), as well as $K_k  \in \mathbb{L}(X_k; Y_k)$, and a possibly non-convex but finite-valued $E_k : X_k \to \R$.
        \item A bounded set $\PDpredictConstr_{0:N} \subset \ProductSpace_{0:N}$ of primal-dual comparison sequences.
        \item
        Primal-dual predictors $P_k: X_k \times Y_k\to X_{k+1}\times Y_{k+1}$.

    \end{enumerate}
\end{assumption}

\begin{example}[Primal-dual predictors]
    \label{ex:pd:predictors}
    In the numerical experiments of \cref{sec:numerical}, for the primal variable $x$, we will use the incompressible flow predictor $W_k: X_k \to X_{k+1}$, $W_k(\thisx): \xi \mapsto \thisx (\xi + h^k(\xi))$, where $h^k(\xi)$ is the predicted displacement at $\xi \in \Omega$ between the frames $k$ and $k+1$. For the dual variable $y$, we use two different predictions: The first one, $T^1_k: Y_k \to Y_{k+1}$, seeks to maintain $\iprod{\snabla \nexxt\primalpredict}{\nexxt\dualpredict} = \iprod{\snabla \thisx}{\thisy}$.
    This is beneficial for preserving total variation \cite{better-predict}: With $G_k(z)=\alpha\norm{z}_{2,1}$, at a solution $(\hat x^k, \hat y^k)$ to \eqref{eq:pd:problem:static-pd}, we have $\iprod{\snabla \hat x^k}{\hat y^k}=\alpha\norm{\snabla \hat x^k}_{2,1}$.
    The second one, $T^2_k: Y_k \to Y_{k+1}$, is defined by the affine update $\nexxt\dualpredict = \thisy + c \snabla \nexxt\primalpredict$. The spatial parameter $c$ is designed to promote sparsity in desired areas. We take $c$ to be inversely proportional to the magnitude of the flow $h^k$, meaning that sparsity is promoted in areas with less inter-frame movement.
\end{example}

Necessary first-order optimality conditions for the static problem $\min F_k+E_k+G_k \circ K_k$, equivalently \eqref{eq:pd:problem:static-pd}, can be expressed with the general notation $u=(x,y)$ as \cite{clasonvalkonen2020nonsmooth,tuomov-nlpdhgm-redo}
\begin{equation}
    \label{eq:pd:h0}
    0 \in H_k(\this{\realoptu})
    \quad\text{for}\quad
    H_k(u) \defeq \begin{pmatrix}
        \subdiff F_k(x) + \grad E_k(x) + K_k^* y\\
        \subdiff G_k^*(y) - K_k x
    \end{pmatrix},
\end{equation}
Likewise, \cref{alg:pd:alg} reads in implicit as iteratively solving for $\thisu$ the inclusion
\begin{gather}
    \label{eq:ppext-pdps}
    0 \in \widetilde H_k(\thisu) + M_k(\thisu-\this\pdpredict),
\shortintertext{where}
    \label{eq:pd:m}
    \widetilde H_k(u) \defeq \begin{pmatrix}
        \subdiff F_k(x) + \estgrad_k(\this{\primalpredict}) + K_k^* y\\
        \subdiff G_k^*(y) - K_k x
    \end{pmatrix}
    \quad\text{and}\quad
    \Precond_k = \begin{pmatrix}
        \inv\tau_k \Id & - K_k^* \\
        -K_k & \inv\sigma_k \Id
    \end{pmatrix}.
\end{gather}
By avoiding explicit proximal maps, this formulation facilitates convergence analysis \cite{tuomov-nlpdhgm} and, by extension, regret analysis.

Since we will be taking forward steps with respect to $E_k$, we will require it to be smooth in an appropriate sense.
The next assumption introduces a global online version of the \term{three-point smoothness} inequality; for the corresponding static inequality, see \cite[Appendix B]{valkonen2020testing} and \cite{clasonvalkonen2020nonsmooth} for exact gradients, and \cite{dizonvalkonen2024tracking} for inexact gradients based on single-loop splitting methods.
For EIT with exact forward operator computations, we reformulate the condition in terms of the PDE itself in \cref{sec:eit}, with some technical results relegated to \cref{app:numerics:pde}.
The assumption also introduces compatibility conditions on the \term{step length parameters} $\tau_k$ and $\sigma_k$ from \cref{alg:pd:alg}, and so-called \term{testing parameters} that encode both primal and dual convergence rates; see \cite{valkonen2020testing,clasonvalkonen2020nonsmooth}.
In this paper, while providing a general theory, we will only apply unaccelerated methods with constant step length and testing parameters.

\begin{assumption}[Global three-point growth and smoothness inequality]
    \label{ass:pd:main-global}
    Given $N \in \N$, \cref{ass:pd:main} holds, and for some $\EkGrowth, \EkLoss \ge 0$, and an error $\Err_k \ge 0$, for all $1 \le k \le N$:
    \begin{enumerate}[label=(\roman*),nosep]
        \item\label{item:pd:main-forwardstep}
        $E_k$ satisfies for all $\this\optx \in \PpredictConstr_k$ and $x \in X_k$ the “erroneous” three-point smoothness
        \[
            \iprod{\estgrad_k(\this\primalpredict)}{x-\this{\optx}}
            \ge
            E_k(x) - E_k(\this{\optx})
            + \frac{\EkGrowth}{2}\norm{x - \this\optx}^2
            - \frac{\EkLoss}{2}\norm{x - \this\primalpredict}^2
            - \Err_k.
        \]
        \item We are given step length and testing parameters $\tau_k,\sigma_k>0$ and $\eta_k,\tauTest_k,\sigmaTest_k > 0$ with
        \begin{subequations}%
            \label{eq:pd:stepconds1}%
            \begin{align}%
                \label{eq:pd:symcond}
                \eta_k & = \tauTest_k\tau_k = \sigmaTest_k \sigma_k,
                && \text{(primal-dual coupling)}
                \\
                \label{eq:pd:primaltestcond-positivity}
                1    & > {
                    \EkLoss\tau_k + \tau_k\sigma_k\norm{K_k}^2}
                && \text{(metric positivity)}.
            \end{align}%
        \end{subequations}
    \end{enumerate}
\end{assumption}

\begin{example}
    \label{ex:pd:noaccel}
    For an unaccelerated method, we take $\tau_k \equiv \tau$ and $\sigma_k \equiv \sigma$ for some $\tau, \sigma > 0$, along with $\eta_k \equiv \tau$, $\tauTest_k \equiv 1$ and $\sigmaTest_k \equiv \frac{\tau}{\sigma}$.
    For accelerated choices for standard (non-online) primal-dual proximal splitting under strong convexity, see \cite{valkonen2020testing,clasonvalkonen2020nonsmooth}.
\end{example}

The next assumption presents a local version of the previous global assumption. In what follows, we will assume \emph{either} of these assumptions to hold.

\begin{assumption}[Local three-point growth and smoothness inequality]
    \label{ass:pd:main-local}
    Given $N \in \N$, \cref{ass:pd:main} holds, and for all $1 \le k \le N$:
    \begin{enumerate}[label=(\roman*),nosep]
        \item\label{item:pd:main-forwardstep-local0}
        $E_k$ satisfies for some $\EkGrowthGlobal \in \R$, $\EkLossGlobal \ge 0$, and $\ErrGlobal_k\ge0$, for any  $\this\realoptu = (\this{\realoptx}, \this{\realopty}) \in H_k^{-1}(0)$, and for all $x \in X_k$ the “erroneous” three-point monotonicity-like property
        \[
            \iprod{\estgrad_k(\this\primalpredict)- \grad E_k(\this\realoptx)}{x - \this\realoptx} \ge {\EkGrowthGlobal}{}\norm{x - \this\realoptx}^2 - {\EkLossGlobal}{}\norm{x - \this\primalpredict}^2 - \ErrGlobal_k.
        \]
        \item\label{item:pd:main-forwardstep-local}
        $E_k$ satisfies for some factors $\EkLoss \ge 0$ and $\EkGrowthMono \ge \EkGrowth \ge 0$, errors $\Err_k, \ErrMono_k \ge 0$, and a radius $\delta>0$, for all  $\this\optx \in \PpredictConstr_{k} \isect B(\this\primalpredict, \delta)$ and $x \in B(\this\optx,{\delta})$, the inequality
        \begin{gather*}
            \iprod{\estgrad_k(\this\primalpredict)}{x-\this{\optx}}
            \ge
            E_k(x)-E_k(\this{\optx})
            + \frac{\EkGrowth}{2}\norm{x - \this\optx}^2
            - \frac{\EkLoss}{2}\norm{x - \this\primalpredict}^2 - \Err_k,
        \intertext{as well as, for any  $\this\realoptu = (\this{\realoptx}, \this{\realopty}) \in H_k^{-1}(0) \isect B(\this\primalpredict, \delta) \times Y_k$ and $x \in B(\realoptx^k,{\delta})$ that}
            \iprod{\estgrad_k(\this\primalpredict)- \grad E_k(\this\realoptx)}{x - \this\realoptx} \ge {\EkGrowthMono}{}\norm{x - \this\realoptx}^2 - {\EkLossMono}{}\norm{x - \this\primalpredict}^2 - \ErrMono_k.
        \end{gather*}
        \item\label{item:pd:main-step-params-local} We are given step length and testing parameters $\tau_k,\sigma_k>0$ and $\eta_k,\tauTest_k,\sigmaTest_k > 0$ as well as a $\EkLipCoeffCoeff \in (0,1)$ with
        \begin{subequations}%
            \label{eq:pd:stepconds1-local}%
            \begin{align}%
                \label{eq:pd:symcond-local}
                \eta_k & = \tauTest_k\tau_k = \sigmaTest_k \sigma_k,
                \quad\text{and}
                \\[-0.5ex]
                \label{eq:pd:primaltestcond-positivity-local}
                1 & \ge
                \max \left\{
                    \EkLoss,
                    -2(\gamma_k + \EkGrowthGlobal),
                    \frac{2}{1-\EkLipCoeffCoeff}\EkLossGlobal,
                    \frac{2}{1-\EkLipCoeffCoeff}\EkLossMono
                \right\} \tau_k
                + \tau_k\sigma_k\norm{K_k}^2.
            \end{align}%
        \end{subequations}
    \end{enumerate}
\end{assumption}

\Cref{ass:pd:main-local}\,\cref{item:pd:main-forwardstep-local0} is global, while \cref{item:pd:main-forwardstep-local} is local. The former is completely analogous to the second part of the latter, however, the factor $\EkGrowthGlobal$ is allowed to be negative.

To use the local version of the three-point growth and smoothness inequality, we need additional assumptions on the comparison sequences. Before stating these assumptions, we formalise the predictors $P_k$ of \cref{alg:pd:alg}.
To do so, we introduce
\begin{equation*}
    \GrowthOp_k \defeq
    \diag\bigl(
        (\gamma_k+\EkGrowth)  \Id,\;
        \rho_k \Id
    \bigr)
    \quad\text{and}\quad
    \LossOp_{k} \defeq
    \diag\bigl(
        \EkLoss\Id,\; 0
    \bigr).
\end{equation*}
For use with \cref{ass:pd:main-local}\,\ref{item:pd:main-forwardstep-local0} and the first part of \ref{item:pd:main-forwardstep-local}, we also define
\begin{align*}
    \GrowthOpGlobal_k & \defeq
    2
    \diag\bigl(
        (\gamma_k + \EkGrowthGlobal) \Id,\;
        \rho_k \Id
    \bigr),
    &
    \LossOpGlobal_k & \defeq
    2\diag\bigl(
        \EkLossGlobal \Id,\;
        0
    \bigr),
    \\
    \GrowthOpMono_k & \defeq
    2
    \diag\bigl(
        (\gamma_k + \EkGrowthMono)  \Id,\;
        \rho_k \Id
    \bigr),
    \quad\text{and}
    &
    \LossOpMono_k & \defeq
    2\diag\bigl(
        \EkLossMono \Id,\;
        0
    \bigr).
\end{align*}

The first part of the next assumption ensures that the set of comparison sequences is large enough to have primal-dual critical points of the static objectives in its proximity.
The second part is a more technical condition on the uniformity of bounds.
To state the assumption, recalling that $\nexxt\pdpredict \defeq P_k(\thisu)$, we define the \term{prediction errors}
\begin{equation}
    \label{eq:pd:prediction-error}
    \epsilon_{k+1}^{\dagger}(u^k, \optu^{k:{k+1}})
    \defeq
    \frac{1}{2}\fakenorm{\nexxt\pdpredict-\overnextu}_{\eta_{k+1}\Precond_{k+1}}^2 - \frac{1}{2}\fakenorm{\this u-\this{\opt u}}_{\eta_k(\Precond_k+ \GrowthOp_k)}^2
    \quad\text{for all}\quad \optu^{k:{k+1}} \in \PDpredictConstr_{k:k+1}.
\end{equation}
They measure the difference of deviation from a chosen comparison sequence between the current iterate and its prediction.
Since $u^k$ is always the algorithm-generated iterate, when $\optu^{k:{k+1}}$ is clear from the context, we write for brevity $\epsilon_{k+1}^{\dagger} \defeq \epsilon_{k+1}^{\dagger}(u^k, \optu^{k:{k+1}})$.
For brevity, we also write
\[
    H_{0:N}(u^{0:N}) \defeq  H_0(u^0)\times\dots\times H_N(u^N)\subset \ProductSpace_{0:N}.
\]

\begin{assumption}[Critical point proximity]
    \label{ass:pd:add-local}
    Let $N \in \N$ be given, and \cref{ass:pd:main-local} to be assumed (instead of \cref{ass:pd:main-global}). Then:
    \begin{enumerate}[label=(\roman*),nosep]
        \item\label{ass:pd:add-local-i}
        For some $\criticalprox_k>0$, ($1 \le k \le N$), we have
        \[
            \PDpredictConstr_{1:N}
            \subset
            \bigl\{
                \bar u^{1:N} \in \ProductSpace_{1:N}
            \bigm|
                \inf\nolimits_{\hat u^{1:N} \in H^{-1}_{1:N}(0)}\norm{\bar u^{k} - \hat u^{k}}_{M_k + \GrowthOpMono_k} \le \criticalprox_k
                \text{ for } 1\le k \le N
            \bigr\}.
        \]
        \item\label{ass:pd:add-local-init-iterate}
        For some $\YoungCoeff_k, \Delta > 0$ as well as $\tilde\delta \in (0, \delta)$,  ($1 \le k \le N$), for
        \[%
            \deltaCoeff_k
            =
            \phi_k(1+2\tau_k\min\{\gamma_k + \EkGrowthGlobal,\EkLossGlobal\}-\tau_k\sigma_k\norm{K_k}^2)
            >
            0
        \]%
        (where the positivity is a consequence of \eqref{eq:pd:primaltestcond-positivity-local})
        for all $\optu^{0:N} \in \PDpredictConstr_{0:N}$
        \begin{subequations}
        \label{eq:pd:d-bound-delta-xi-ineq}
        \begin{gather}
            \label{eq:pd:d-bound}
            \begin{aligned}[t]
            0 < d_N(\optu^{0:N})
            \defeq
            \inf_{0 \le n \le N}
            \biggl(
                \frac{\theta_{n+1}(\delta-\tilde\delta)^2}{\xi_{n+1}}
                &
                - \frac{1+ \Perturbr}{\YoungCoeff_{n+1}}\eta_{n+1}\criticalprox_{n+1}^2
                \\
                &
                - 2\epsilon_{n+1}^{\dagger}
                - \sum_{k=1}^n \left(
                    \frac{1+\Perturbr}{2\EkLipCoeffCoeff}\eta_k\criticalprox_k^2
                    +
                    \epsilon_{k}^{\dagger}
                    +
                    \eta_k\ErrMono_k
                \right)
            \biggr)
            \end{aligned}
            \\[-1ex]
        \shortintertext{as well as}
            \label{eq:pd:delta-xi-ineq}
            (2-\EkLipCoeffCoeff)(\inv\xi_k+1)(\delta-\tilde\delta)^2 + 2\inv\theta_k\eta_k\ErrGlobal_k
            \le \tilde\delta^2.
        \end{gather}
        \end{subequations}
    \end{enumerate}
\end{assumption}

\begin{remark}
    We must have $\theta_{n+1}(\delta-\tilde\delta)^2 \ge (1+ \Perturbr)\eta_{n+1}\criticalprox_{n+1}^2$ and $2\inv\theta_k\eta_k\ErrGlobal_k \le \tilde\delta^2$ for \eqref{eq:pd:d-bound-delta-xi-ineq} to hold.
    Therefore, given these bounds, the optimal choice of $\xi_k$ is to solve it from \eqref{eq:pd:delta-xi-ineq} as an equality, and insert the result into \eqref{eq:pd:d-bound}.
\end{remark}

\begin{remark}
    Similarly to $\theta_k$, also $\criticalprox_k$ depends on the testing parameter $\phi_k$, so its mere increase is not sufficient to satisfy $d_N(\optu^{0:N})>0$.
    For constant step lengths, $\phi_k \equiv 1$, $\eta_k \equiv \tau$, and $\xi_k \equiv \kappa_k \equiv \kappa$, provided that $\gamma_{n+1} + \EkGrowthGlobal[n+1] \ge 0$ for all $1 \le n \le N$,
    taking and ensuring
    \begin{gather*}
        \bar\theta_N \defeq 1-\tau\sigma\sup\limits_{1 \le k \le N}\norm{K_k}^2>0,
        \\[-1.5ex]
    \shortintertext{\eqref{eq:pd:d-bound-delta-xi-ineq} hold when}
        \bar\theta_N(\delta-\tilde\delta)^2
        >
        \sum_{k=1}^{N+1}
        \left(
            (1+\Delta)\tau\criticalprox_k^2
            +
            2\kappa\epsilon_{k}^{\dagger}
            +
            \tau\kappa\ErrMono_k
        \right)
        \quad\text{and}
        \\
        (2-\kappa)(\inv\kappa+1)(\delta-\tilde\delta)^2 + 2\inv{\bar\theta}_N\tau\ErrGlobal_k
            \le \tilde\delta^2.
    \end{gather*}
    That is, we can satisfy the conditions by taking $\tau$ and $\tau\sigma$ small enough, $\tilde\delta< \delta$ large enough, given that the prediction errors $\epsilon_{k}^{\dagger}$ have small sums compared to the radius $\delta>0$. However, large critical point proximities $\criticalprox_k$ and gradient errors $\ErrGlobal_k$ and $\ErrMono_k$ can be compensated for by small step lengths, as long as their sums are bounded.
\end{remark}

\subsection{Regret analysis}
\label{sec:pd:regret}

We now analyse the regret of \cref{alg:pd:alg}.
The main work is with the local \cref{ass:pd:main-local} for the nonconvex function $E_k$. Readers only interested in the global \cref{ass:pd:main-global}, or the general ideas, may skip \cref{lemma:pd:apriori,lemma:pd:aposteriori}.

We first verify the positive semi-definiteness of operators central to our analysis.

\begin{lemma}
    \label{lemma:pd:positivity}
    Let $N \ge 1$ and suppose \cref{ass:pd:main} holds for any $0 \le k \le N$ and $u^{1:N}$ generated by \cref{alg:pd:alg} for an initial $u^0 \in X_0 \times Y_0$. Then the following statements hold:
    \begin{enumerate}[label=(\roman*)]
    \item\label{item:pd:positivity:i}
    The operators $\GrowthOp_k$, $\GrowthOpMono_k,$  $\LossOp_k$, and $\LossOpMono_k$ are positive semi-definite.
    \item\label{item:pd:positivity:ii}
    If \cref{ass:pd:main-global} holds, then $\Precond_k$ and $\Precond_k - \LossOp_k$ are positive semi-definite.
    \item\label{item:pd:positivity:iii}
    If \cref{ass:pd:main-local} holds,  then $\Precond_k$ and $\Precond_k - \LossOp_k$ are positive semi-definite, as well as
    \begin{align*}
        (1-\EkLipCoeffCoeff)M_k
        & \ge \LossOpGlobal_k,
        &
        \eta_k(M_k + \LossOpGlobal_k)
        & \ge \diag\bigl( \deltaCoeff_k \Id,\; 0 \bigr),
        &
        \GrowthOpMono_k
        & \ge 2\GrowthOp_k,
        \\
        (1-\EkLipCoeffCoeff)M_k
        & \ge \LossOpMono_k,
        &
        \eta_k(M_k + \GrowthOpGlobal_k)
        & \ge \diag\bigl( \deltaCoeff_k \Id,\; 0 \bigr)
    \end{align*}
    for
    $
        \deltaCoeff_k
        =
        \phi_k(1+2\tau_k\min\{\gamma_k + \EkGrowthGlobal,\EkLossGlobal\}-\tau_k\sigma_k\norm{K_k}^2) > 0
    $
    as in \cref{ass:pd:add-local}.
    \end{enumerate}
\end{lemma}

\begin{proof}
    The positive semi-definiteness of $\GrowthOp_k$, $\GrowthOpMono_k$, $\LossOp_k$, and $\LossOpMono_k$ follows from $\gamma_k,\rho_k,\EkLoss,\EkGrowth,\EkGrowthMono,\EkLossMono \ge 0$.
    This establishes \cref{item:pd:positivity:i}.

    For \cref{item:pd:positivity:ii}, using \cref{ass:pd:main-global} and Young's inequality, we have
    \[
        \eta_k(M_k - \LossOp_k)
        =
        {\eta_k}{}
        \begin{pmatrix}
            (\inv\tau_k - \EkLoss)  \Id & - K_k^* \\
            - K_k & \inv\sigma_k \Id
        \end{pmatrix}
        \ge
        \phi_k
        \begin{pmatrix}
            \Id - \tau_k \EkLoss - \tau_k\sigma_k K_k^*K_k & 0 \\
            0 & 0
        \end{pmatrix}.
    \]
    Thus, \eqref{eq:pd:primaltestcond-positivity} establishes the positive semi-definiteness of $\Precond_k$ and $\Precond_k - \LossOp_k$ .

    Finally, for \cref{item:pd:positivity:iii}, using  \eqref{eq:pd:primaltestcond-positivity-local} in \cref{ass:pd:main-local}, and Young's inequality, we estimate%
    \[
        \begin{aligned}[t]
            (1-\EkLipCoeffCoeff)
            \Precond_k
            -
            \LossOpGlobal_k
            &
            =
            (1-\EkLipCoeffCoeff)
            \begin{pmatrix}
                \bigl(\inv\tau_k - \frac{2\EkLossGlobal}{1-\EkLipCoeffCoeff}\bigr)\Id & - K_k^* \\
                -K_k & \bigl(\inv\sigma_k + \frac{2\rho_k}{1-\EkLipCoeffCoeff}\bigr)\Id
            \end{pmatrix}
            \\
            &
            \ge
            (1-\EkLipCoeffCoeff)
            \begin{pmatrix}
                \bigl(\inv\tau_k - \frac{2\EkLossGlobal}{1-\EkLipCoeffCoeff}\bigr) \Id - \sigma_k K_k^*K_k & 0 \\
                0 & 0
            \end{pmatrix}
            \ge
            0,
        \end{aligned}
    \]
    which also shows that $\Precond_k \ge 0$.
    Likewise, we show $(1-\EkLipCoeffCoeff)\Precond_k \ge \LossOpMono_k$.
    Moreover,
    \[
        \begin{aligned}[t]
            \eta_k(
                \Precond_k
                +
                \GrowthOpGlobal_k
            )
            &
            =
            \eta_k
            \begin{pmatrix}
                \left(\inv\tau_k + 2(\gamma_k + \EkGrowthGlobal)\right)\Id & - K_k^* \\
                -K_k & \left(\inv\sigma_k + {2\rho_k}{}\right)\Id
            \end{pmatrix}
            \\
            &
            \ge
            \phi_k
            \begin{pmatrix}
                \left(1 + 2\tau_k(\gamma_k + \EkGrowthGlobal)\right) \Id - \tau_k\sigma_k K_k^*K_k & 0 \\
                0 & 0
            \end{pmatrix}
            \ge
            \begin{pmatrix} \theta_k \Id & 0 \\ 0 & 0 \end{pmatrix}.
        \end{aligned}
    \]
    The claim regarding $\eta_k(M_k + \LossOpGlobal_k)$ follows by a similar argument, as does the positive semi-definiteness of $\Precond_k - \LossOp_k$.
    Finally,
    $
        \GrowthOpMono_k - 2\GrowthOp_k
        =
        \diag\bigl(
            (\EkGrowthMono-\EkGrowth)  \Id,\,
            0
        \bigr)
    $
    by the definitions of $\GrowthOp_k$ and $\GrowthOpMono_k$.
    Since $\EkGrowthMono \ge \EkGrowth$ by \cref{ass:pd:main-local}\,\cref{item:pd:main-forwardstep-local}, this proves $\GrowthOpMono_k \ge 2\GrowthOp_k$.
\end{proof}

\Cref{lemma:pd:positivity} justifies the norm notation $\norm{\freevar}_M$ for $M=\Precond_k,\Precond_k+ \GrowthOp_k,\Precond_k - \LossOp_k$. Now, using the critical point proximity \cref{ass:pd:add-local} and the global Lipschitz-like bound of \cref{ass:pd:main-local}\,\cref{item:pd:main-forwardstep-local0}, and assuming sufficient proximity of the previous iterate $\prev u$ to the corresponding time index of a comparison sequence, the next “a priori” lemma shows that the current primal iterate $\thisx$ is in the ball where the local three-point inequalities of \cref{ass:pd:main-local}\,\cref{item:pd:main-forwardstep-local} hold. The “a posteriori” lemma that follows, will then show that, in fact, $\thisx$ is also in the proximity of a comparison sequence. An inductive argument will then easily establish regret estimates as in the convex case \cite{better-predict}.

\begin{lemma}[A priori estimate]
    \label{lemma:pd:apriori}
    Let $N \ge 1$ and suppose \cref{ass:pd:main,ass:pd:main-local,ass:pd:add-local} hold, and that $u^{1:N}$ and $\pdpredict^{1:N}$ are generated by \cref{alg:pd:alg} for an initial $u^0 \in X_0 \times Y_0$.
    Further, let $\optu^{0:N} \in \PDpredictConstr_{0:N}$ and $\Perturbr > 0$, and for a $1 \le k \le N$, suppose that
    \begin{equation}
        \label{eq:pd:apriori:cond}
        \norm{\prevu-\prev{\optu}}_{\eta_{k-1}(M_{k-1}+ \GrowthOp_{k-1})}^2
        \le
        \frac{\theta_k(\delta-\tilde\delta)^2}{\xi_k}
        - \frac{1+ \Perturbr}{\YoungCoeff_k}\eta_k\criticalprox_k^2 - 2\epsilon_{k}^{\dagger}
    \end{equation}
    for the prediction error $\epsilon_{k}^{\dagger}=\epsilon_{k}^{\dagger}(u^{k-1},\optu^{k-1:k})$ of \cref{eq:pd:prediction-error}, the constants $c_k, \kappa_k, \theta_k,\YoungCoeff_k,\criticalprox_k,\tilde\delta$ and radius $\delta>0$ of \cref{ass:pd:add-local,ass:pd:main-local}.
    Then there exists $ \this\realoptu = (\this\realoptx,\this\realopty) \in H_{k}^{-1}(0)$ with
    \begin{gather}
        \label{eq:pd:apriori:existence-bound}
        \norm{\bar u^{k} - \hat u^{k}}_{M_k + \GrowthOpMono_k}^2 \le (1+\Perturbr)\criticalprox^2_k,
    \shortintertext{as well as}
        \label{eq:pd:apriori:claim}
        \norm{\thisx-\this\realoptx}
        <
        \xrealoptxbound,
        \quad
        \norm{\this\primalpredict-\this\realoptx}
        <
        \pxrealoptxbound,
        \quad
        \norm{\this\primalpredict-\this\optx}
        <
        \xoptxbound,
        \quad
        \text{and}
        \quad
        \norm{\thisx-\this\optx}
        <
        \pxoptxbound.
    \end{gather}
\end{lemma}

\begin{proof}
    By \cref{ass:pd:add-local}\,\ref{ass:pd:add-local-i}, \eqref{eq:pd:apriori:existence-bound} holds for some $\this\realoptu = (\this{\realoptx}, \this{\realopty}) \in H_k^{-1}(0)$.
    By the definition of $H_k$, we have
    $\this{\realopt q} \defeq -\grad E_k(\this{\realopt x})-K_k^* \this{\realopty} \in \subdiff F_k(\this{\realoptx})$ and $\this{\realopt p} \defeq K_k\this{\realoptx} \in \subdiff G_k^*(\this{\realopty})$.
    Now Cauchy-Schwartz inequality, Young's inequality and the (strong) monotonicity of $\subdiff F_k$ and $\subdiff G_k^*$ together with the erroneous three-point monotonicity-like property of $E_k$ and $\estgrad_k$ in \cref{ass:pd:main-local}\,\cref{item:pd:main-forwardstep-local0} yield
    \[
        \begin{aligned}[t]
            \iprod{\widetilde H_k(\thisx,\thisy)}{\thisu-\this{\realoptu}}
            &
            =
            \iprod{\subdiff{F_k(\thisx)} - \this{\realopt q}}{\thisx - \this\realoptx}
            + \iprod{\subdiff G_k^*(\thisy) - \this{\realopt p}}{\thisy-\this\realopty}
            \\
            \MoveEqLeft[-1]
            +  \iprod{\estgrad_k(\this\primalpredict)- \grad E_k(\this\realoptx)}{\thisx - \this\realoptx}
            \\
            \MoveEqLeft[-1]
            + \iprod{K_k^*(\thisy-\this{\realopty})}{\thisx-\this\realoptx}
            - \iprod{K_k(\thisx-\this{\realoptx})}{\thisy - \this\realopty}
            \\
            &
            \ge
            (\gamma_k+\EkGrowthGlobal)\norm{\thisx-\this\realoptx}^2
            +\rho_k\norm{\thisy-\this\realopty}^2
            -\EkLossGlobal \norm{\this\primalpredict-\this\realoptx}^2
            -\ErrGlobal_k
            \\
            &
            =
            \frac{1}{2}\fakenorm{\thisu-\this\realoptu}_{\GrowthOpGlobal_k}^2
            -\frac{1}{2}\norm{\this\pdpredict-\this\realoptu}_{\LossOpGlobal_k}^2
            -\ErrGlobal_k.
        \end{aligned}
    \]
    Applying here \eqref{eq:ppext-pdps} and the Pythagoras' identity
    \begin{gather}
        \nonumber
        \iprod{\thisu-\this\pdpredict}{\thisu-\this{\realoptu}}_{\Precond_k }
        =
        \frac{1}{2}\norm{\thisu-\this\pdpredict}^2_{\Precond_k }
        +
        \frac{1}{2}\norm{\thisu-\this{\realoptu}}^2_{\Precond_k }
        -
        \frac{1}{2}\norm{\this{\realoptu} - \this\pdpredict }^2_{\Precond_k },
        \\[-0.5ex]
    \shortintertext{we get}
        \label{eq:pd:apriori:pythagoras}
        \frac{1}2\norm{\this\pdpredict-\this{\realoptu}}_{M_k + \LossOpGlobal_k}^2
        \ge
        \frac{1}{2}\norm{\thisu-\this\realoptu}_{M_k+\GrowthOpGlobal_k}^2
        + \frac{1}{2}\norm{\thisu-\this\pdpredict}_{\Precond_k}^2
        -\ErrGlobal_k.
    \end{gather}
    The operators $M_k$, $M_k+\GrowthOpGlobal_k$, and $M_k + \LossOpGlobal_k$ are positive semi-definite by \cref{lemma:pd:positivity}.%

    Since $\YoungCoeff_k > 0$ by \cref{ass:pd:add-local}\,\ref{ass:pd:add-local-init-iterate}, Young's inequality gives
    \[
        \begin{aligned}
            \norm{\this\pdpredict-\this{\realoptu}}_{M_k + \LossOpGlobal_k}^2
            &
            \le
            (1+\YoungCoeff_k)\norm{\this\pdpredict - \this{\optu}}_{M_k + \LossOpGlobal_k}^2
            +
            (1+\inv\YoungCoeff_k)\norm{\this{\optu}-\this{\realoptu}}_{M_k + \LossOpGlobal_k}^2 .
        \end{aligned}
    \]
    By \cref{lemma:pd:positivity}, $M_k + \LossOpGlobal_k \le (2 - \EkLipCoeffCoeff)M_k$.
    By the definition of the prediction errors in \eqref{eq:pd:prediction-error}
    \[
        \frac{1}{2}\norm{\this\pdpredict-\this\optu}_{\eta_{k}\Precond_{k}}^2 \le \frac{1}{2}\norm{\prev u-\prev{\opt u}}_{\eta_{k-1}(\Precond_{k-1}+ \GrowthOp_{k-1})}^2
        + \epsilon_{k}^{\dagger}
    \]
    Using the above inequalities and $\GrowthOp_k \ge 0$ from \cref{lemma:pd:positivity}, we thus obtain
    \[
        \begin{aligned}
            \frac{1}{2}\norm{\this\pdpredict-\this{\realoptu}}_{\eta_k(M_k + \LossOpGlobal_k)}^2
            &
            \le
            \frac{(2-\EkLipCoeffCoeff)(1+\YoungCoeff_k)}{2}
            \left(
            \norm{\prevu-\prev{\optu}}_{\eta_{k-1}(M_{k-1}+ \GrowthOp_{k-1})}^2
            +
            2\epsilon_{k}^{\dagger}
            +
            \inv\YoungCoeff_k\norm{\this\optu-\this{\realoptu}}_{\eta_k M_k}^2
            \right).
        \end{aligned}
    \]
    We multiply \eqref{eq:pd:apriori:pythagoras} by $\eta_k$ and use it on the left hand side.
    Then we use \cref{eq:pd:apriori:cond,eq:pd:apriori:existence-bound} on the right hand side. This yields
    \begin{gather}
        \label{eq:pd:apriori:2}
        \begin{aligned}
        \frac{1}{2}\norm{\thisu-\this\realoptu}_{\eta_k(M_k+\GrowthOpGlobal_k)}^2
        + \frac{1}{2}\norm{\thisu-\this\pdpredict}_{\eta_k\Precond_k}^2
        &
        \le
        \frac{1}{2}\norm{\this\pdpredict-\this{\realoptu}}_{\eta_k(M_k + \LossOpGlobal_k)}^2
        \le
        \frac{1}{2}\alpha_k\theta_k
        \end{aligned}
    \end{gather}
    for
    $
        \alpha_k \defeq (2-\EkLipCoeffCoeff)(\inv\xi_k+1)(\delta-\tilde\delta)^2 + 2\inv\theta_k\eta_k\ErrGlobal_k.
    $
    By \cref{ass:pd:add-local}\,\cref{ass:pd:add-local-init-iterate}, we have $\alpha_k \le \tilde\delta^2 < \delta^2$.
    Thus \cref{lemma:pd:positivity} and \eqref{eq:pd:apriori:2} show
    \begin{align}
        \label{eq:pd:apriori:3}
        \delta^2\theta_k
        &
        \ge
        \alpha_k\theta_k
        \ge
        \norm{\this\pdpredict-\this{\realoptu}}_{\eta_k(M_k + \LossOpGlobal_k)}^2
        \ge
        \pNormCoeff\norm{\this\primalpredict-\this\realoptx}^2
    \quad\text{as well as}\\
        \nonumber%
        \delta^2\theta_k
        &
        \ge
        \alpha_k\theta_k
        \ge
        \norm{\thisu-\this\realoptu}_{\eta_k(M_k+\GrowthOpGlobal_k)}^2
        \ge
        \pNormCoeff\norm{\thisx-\this\realoptx}^2.
    \end{align}
    This shows the first and second parts of of \eqref{eq:pd:apriori:claim}.

    By \cref{lemma:pd:positivity} and \eqref{eq:pd:apriori:cond}, we have $\norm{u}_{\eta_k M_k}^2 \ge \pNormCoeff\norm{x}^2$, $\GrowthOpMono_k \ge 0$, and
    $\eta_k\criticalprox_k^2<(1+\Perturbr)\eta_k\criticalprox_k^2 < \theta_k(\delta-\tilde\delta)^2$.
    Thus using \cref{ass:pd:main-local}\,\cref{ass:pd:add-local-i}, we establish
    \[
        \pNormCoeff\norm{\this\optx-\this{\realoptx}}^2
        \le
        \norm{\this\optu-\this{\realoptu}}_{\eta_k M_k}^2
        \le
        \norm{\this\optu-\this{\realoptu}}_{\eta_k(M_k + \GrowthOpMono_k)}^2
        \le
        \eta_k\criticalprox_k^2
        <
        \pNormCoeff(\delta-\tilde\delta)^2.
    \]
    Hence $\norm{\this\optx-\this{\realoptx}} \le \delta-\tilde\delta$.
    By \cref{ass:pd:add-local}\,\cref{ass:pd:add-local-init-iterate}, $\alpha_k < \tilde\delta^2$.
    Minding \eqref{eq:pd:apriori:3}, it follows that
    \[
        \begin{aligned}[t]
        \norm{\this\primalpredict-\this\optx}
        &
        \le
        \norm{\this\primalpredict-\this\realoptx}
        +
        \norm{\this\optx-\this\realoptx}
        \le
        \sqrt{\alpha_k} + \delta-\tilde\delta
        <
        \delta,
        \end{aligned}
    \]
    and completely analogously $\norm{\thisx-\this\optx} < \delta$, proving last two inequalities of \eqref{eq:pd:apriori:claim}.
\end{proof}

Now that we know that the curent primal iterate belongs to a ball where the estimates of \cref{ass:pd:main-local}\,\cref{item:pd:main-forwardstep-local} hold, we can continue with the aforementioned a posteriori estimate.

\begin{lemma}[A posteriori estimate]
    \label{lemma:pd:aposteriori}
    Let $N \ge 1$ and suppose \cref{ass:pd:main,ass:pd:main-local,ass:pd:add-local} hold, $\optu^{0:N} \in \PDpredictConstr_{0:N}$.
    Let $u^{1:N}$ and $\pdpredict^{1:N}$ be generated by \cref{alg:pd:alg} for an initial $u^0 \in X_0 \times Y_0$ satisfying for $d_N$ given in \cref{ass:pd:add-local} the local initialisation condition
    \begin{equation}
        \label{eq:pd:aposteriori:init-cond}
        \frac{1}{2}\norm{ u^0-{\opt u^0}}_{\eta_{0}(\Precond_{0}+ \GrowthOp_{0})}^2
        \le
        d_N(\bar u^{0:N}).
    \end{equation}
    Then for all $1 \le n \le N$, the three-point smoothness and growth inequality
    \begin{equation}
        \label{eq:pd:aposteriori:claim:growth}
        E_n(x^n)-E_n({\optx^n})
        + \frac{\EkGrowth}{2}\norm{\optx^n -x^n}^2
        - \frac{\EkLoss}{2}\norm{x^n-\primalpredict^n}^2
        \le
        \iprod{\estgrad_n(\primalpredict^n)}{x^n-{\optx^n}}
        +\Err_k
    \end{equation}
    holds, as does the tracking inequality
    \begin{equation}
        \label{eq:pd:aposteriori:claim:tracking}
        \frac{1}{2}\norm{u^n-\optu^n}_{ \eta_{n}(\Precond_{n}+ \GrowthOp_{n})}^2
        <
        \frac{1}{2}\norm{ u^0-{\opt u^0}}_{\eta_{0}(\Precond_{0}+ \GrowthOp_{0})}^2
        +
        \sum_{k=1}^n
        \left(
            \epsilon_{k}^{\dagger}
            +
            \frac{1}{\EkLipCoeffCoeff}\criticalprox_k^2
        \right).
    \end{equation}
\end{lemma}

\begin{proof}
    The estimate \eqref{eq:pd:aposteriori:claim:growth} follows directly from \cref{ass:pd:main-local} if we show that the iterates and predictions $x^k, \breve x^k \in B(\optx^k,\delta)$.
    This follows from \cref{lemma:pd:apriori}, if we prove its assumptions, i.e., \eqref{eq:pd:apriori:cond} for all $1 \le k \le N$.
    We do this by induction, and in the course of it, also prove \eqref{eq:pd:aposteriori:claim:tracking}.
    That is, we prove by induction, for all $1 \le n \le N$ that
    \begin{equation}
        \label{eq:pd:aposteriori:0}
        \frac{1}{2}\norm{ u^{k-1}-{\opt u^{k-1}}}_{\eta_{k-1}(\Precond_{k-1}+ \GrowthOp_{k-1})}^2
        \le
        \frac{\theta_k(\delta-\tilde\delta)^2}{\xi_k}
        - \frac{1+ \Perturbr}{\YoungCoeff_k}\eta_k\criticalprox_k^2
        - 2\epsilon_{k}^{\dagger}
        \quad\text{for all } 1 \le k \le n.
    \end{equation}

    The inductive basis, i.e., \eqref{eq:pd:aposteriori:0} for for $n=1$, follows directly from \eqref{eq:pd:aposteriori:init-cond} and the definition of $d_N$ in \cref{ass:pd:add-local}.
    For the inductive step, suppose that \eqref{eq:pd:aposteriori:0} holds for some $1 \le n < N$.
    We will prove that the same holds for $n+1$.
    Towards this end, let $k=1,\ldots,n$ be arbitrary.
    The inductive assumption \eqref{eq:pd:aposteriori:0} directly proves the assumption \eqref{eq:pd:apriori:cond} of \cref{lemma:pd:apriori} up to $k=n$.
    By the lemma, $x^k,\primalpredict^k \in B(\realoptx^k,{\delta})$ for some $\realoptu^k=(\realoptx^k,\realopty^k) \in \inv H_k(0)$, so the latter inequality of \cref{ass:pd:main-local}\,\ref{item:pd:main-forwardstep-local} gives
    \[
        \iprod{\estgrad_k(\this\primalpredict)- \grad E_k(\this\realoptx)}{\thisx - \this\realoptx} \ge {\EkGrowthMono}{}\norm{\thisx - \this\realoptx}^2 - {\EkLossMono}{}\norm{\thisx-\this\primalpredict}^2
        -\ErrMono_k.
    \]
    Recall the expression for $\widetilde H_k$ from \eqref{eq:pd:m} and $H_k$ from \eqref{eq:pd:h0}.
    Since $0 \in H_k(\this{\realoptx}, \this{\realopty})$, we have $\this{\realopt q} \defeq -\grad E_k(\realopt x^k)-K_k^* \this{\realopty} \in \subdiff F_k(\this{\realoptx})$ and $\this{\realopt p} \defeq K_k\this{\realoptx} \in \subdiff G_k^*(\this{\realopty})$. These together with the above three-point inequality and the (strong) monotonicity of $\subdiff F_k$ and $\subdiff G_k^*$ yield
    \begin{equation}
        \label{eq:pd:aposteriori:1}
        \begin{aligned}[t]
            \iprod{\widetilde H_k(\thisx,\thisy)}{\thisu-\this{\realoptu}}
            &
            =
            \iprod{\subdiff{F_k(\thisx)} - \this{\realopt q}}{\thisx - \this\realoptx}
            +  \iprod{\estgrad_k(\this\primalpredict)- \grad E_k(\this\realoptx)}{\thisx - \this\realoptx}
            \\
            \MoveEqLeft[-1]
            + \iprod{\subdiff G_k^*(\thisy) - \this{\realopt p}}{\thisy-\this\realopty}
            \\
            \MoveEqLeft[-1]
            + \iprod{K_k^*(\thisy-\this{\realopty})}{\thisx-\this\realoptx}
            - \iprod{K_k(\thisx-\this{\realoptx})}{\thisy - \this\realopty}
            \\
            &
            \ge
            \left(\gamma_k+{\EkGrowthMono}{}\right)\norm{\thisx-\this\realoptx}^2
            +\rho_k\norm{\thisy-\this\realopty}^2
            -{\EkLossMono}{}  \norm{\thisx-\this\primalpredict}^2
            -\ErrMono_k
            \\
            &
            = \frac{1}{2}\norm{\thisu - \this{\hat u}}_{\GrowthOpMono_k}^2 - \frac{1}{2}\norm{\thisu-\this\pdpredict}_{  \LossOpMono_k}^2
            -\ErrMono_k.
            \end{aligned}
    \end{equation}
    Applying the linear “testing operator” $\iprod{\freevar}{\thisu-\this{\realoptu}}$ to both sides of \eqref{eq:ppext-pdps} then yields
    \[
        0 \ge
        \iprod{\thisu-\this\pdpredict}{\thisu-\this{\realoptu}}_{\Precond_k}
        +\frac{1}{2}\norm{\thisu-\this\realoptu}_{\GrowthOpMono_k}^2
        - \frac{1}{2}\norm{\thisu-\this\pdpredict}_{ \LossOpMono_k}^2
        -\ErrMono_k.
    \]
    Since $M_k$ is positive semi-definite by \cref{lemma:pd:positivity}, as are $M_k+ \GrowthOpMono_k$ and $\Precond_k - \LossOpMono_k$, they all induce semi-norms.
    Similarly to the proof of \cref{lemma:pd:apriori}, the Pythagoras' identity yields
    \[
        \frac{1}{2}\norm{\this\pdpredict-\this{\realoptu}}_{M_k}^2
        \ge
        \frac{1 }{2}\norm{\thisu-\this\realoptu}_{M_k+ \GrowthOpMono_k}^2
        + \frac{1}{2}\norm{\thisu-\this\pdpredict}_{\Precond_k - \LossOpMono_k}^2
        -\ErrMono_k
    \]
    Using here the Pythagoras' identitities
    \begin{align*}
        \norm{\thisu-\this\realoptu}_{M_k+ \GrowthOpMono_k}^2
        &
        = \norm{\thisu - \this{\optu}}_{M_k+ \GrowthOpMono_k}^2 + \norm{ \this{\optu} -\this\realoptu}_{M_k+ \GrowthOpMono_k}^2
        + 2\iprod{\thisu - \this{\optu}}{\this{\optu}-\this\realoptu}_{M_k+ \GrowthOpMono_k}
        \quad\text{and}
        \\
        \norm{\this\pdpredict-\this{\realoptu}}_{ M_k}^2
        &
        =
        \norm{\this\pdpredict-\this\optu}_{ M_k}^2
        +
        \norm{\this\optu-\this{\realoptu}}_{ M_k}^2
        +
        2\iprod{\this\pdpredict-\this\optu}{\this\optu-\this{\realoptu}}_{ M_k},
    \end{align*}
    we obtain
    \[
        \begin{aligned}
            &
            \frac{1}{2}\norm{\this\pdpredict-\this\optu}_{M_k}^2
            +
            \iprod{\this\pdpredict-\this\optu}{\this\optu-\this{\realoptu}}_{M_k}
            +
            \ErrMono_k
            \\[-0.5ex]
            &
            \ge
            \frac{1}{2}\norm{\thisu-\this\optu}_{  M_k+ \GrowthOpMono_k }^2
            + \frac{1}{2}\norm{\this\optu-\this\realoptu}_{ \GrowthOpMono_k }^2
            + \frac{1}{2}\norm{\thisu-\this\pdpredict}_{\Precond_k - \LossOpMono_k}^2
            + \iprod{\thisu - \this{\optu}}{\this{\optu}-\this\realoptu}_{M_k+ \GrowthOpMono_k}.
        \end{aligned}
    \]
    Splitting $\iprod{\this\pdpredict-\this\optu}{\this\optu-\this{\realoptu}}_{M_k} - \iprod{\thisu - \this{\optu}}{\this{\optu}-\this\realoptu}_{M_k} =  \iprod{\this\pdpredict-\thisu}{\this\optu-\this{\realoptu}}_{M_k}$, multiplying by $\eta_k$, using the definition \cref{eq:pd:prediction-error} of the prediction errors, and rearranging, we obtain
    \begin{multline}
        \label{eq:pd:aposteriori:induction}
        \frac{1}{2}\norm{\prev u-\prev{\opt u}}_{\eta_{k-1}(\Precond_{k-1}+ \GrowthOp_{k-1})}^2
        + (\epsilon_{k}^{\dagger} + \eta_k\ErrMono_k)
        +
        \iprod{\this\pdpredict-\thisu}{\this\optu-\this{\realoptu}}_{ \eta_k M_k}
        \\
        \ge
        \frac{1}{2}\norm{\thisu-\this\optu}_{\eta_k(M_k+ \GrowthOpMono_k)}^2
        +
        \frac{1}{2}\norm{\this\optu-\this\realoptu}_{ \eta_k\GrowthOpMono_k }^2
        +
        \frac{1}{2}\norm{\thisu-\this\pdpredict}_{\eta_k(\Precond_k - \LossOpMono_k)}^2
        +
        \iprod{\thisu - \this{\optu}}{\this{\optu}-\this\realoptu}_{\eta_k\GrowthOpMono_k}.
    \end{multline}
    Summing both sides of \eqref{eq:pd:aposteriori:induction} over $k=1,\dots,n$, telescoping, and rearranging gives
    \begin{gather}
        \label{eq:pd:aposteriori:5}
        \frac{1}{2}\norm{ u^0-{\opt u^0}}_{\eta_{0}(\Precond_{0}+ \GrowthOp_{0})}^2
        +
        \sum_{k=1}^n
        \bigl(
            (\epsilon_{k}^{\dagger}+\eta_k\ErrMono_k)
            + A_k
        \bigr)
        \ge
        \frac{1}{2}\norm{u^n-\optu^n}_{ \eta_{n}(\Precond_{n}+\GrowthOpMono_{n})}^2,
    \shortintertext{for}
        \nonumber
        \begin{aligned}[t]
            A_k
            &
            \defeq
            \iprod{\this\pdpredict-\thisu}{\this\optu-\this{\realoptu}}_{ \eta_k M_k}
            -
            \iprod{\thisu - \this{\optu}}{\this{\optu}-\this\realoptu}_{\eta_k\GrowthOpMono_k}
            \\
            \MoveEqLeft[-1]
            -
            \frac{1}{2}\norm{\this\optu-\this\realoptu}_{\eta_k\GrowthOpMono_k }^2
            -
            \frac{1}{2}\norm{\thisu-\this\optu}_{\eta_k(\GrowthOpMono_k -  \GrowthOp_k)}^2
            -
            \frac{1}{2}\norm{\thisu-\this\pdpredict}_{\eta_k(\Precond_k - \LossOpMono_k)}^2,
        \end{aligned}
    \end{gather}
    where we used the notation $\norm{\cdot}_{\GrowthOpMono_k -  \GrowthOp_k}$ since \cref{lemma:pd:positivity} guarantees $\GrowthOpMono_{k} \ge  \GrowthOp_{k} + \GrowthOpMono_{k}/2  \ge \GrowthOp_{k}$ and both operators are by definition self-adjoint. Using the first of these inequalities, $M_k - \LossOpMono_k \ge \EkLipCoeffCoeff M_k$ (which also holds by  \cref{lemma:pd:positivity}), and $0<\YoungCoeffIII<1$, applying of Young's inequality twice, and rearranging, gives
    \[
        \begin{aligned}
            A_k
            &
            \le
            \frac{\YoungCoeffIII}{2}\norm{\thisu-\this\pdpredict}_{\eta_k M_k}^2
            +
            \frac{1}{2\YoungCoeffIII}\norm{\this\optu-\this{\realoptu}}_{\eta_k M_k}^2
            +
            \frac{1}{4}\norm{\thisu - \this{\optu}}_{\eta_k\GrowthOpMono_k}^2
            +
            \frac{1}{2}\norm{\this\optu-\this\realoptu}_{\eta_k\GrowthOpMono_k}^2
            \\
            \MoveEqLeft[-1]
            -
            \frac{1}{2}\norm{\thisu-\this\optu}_{\eta_k(\GrowthOpMono_k -  \GrowthOp_k)}^2
            -
            \frac{1}{2}\norm{\thisu-\this\pdpredict}_{\eta_k(\Precond_k - \LossOpMono_k)}^2
            \\
            &
            \le
            \frac{1}{2\YoungCoeffIII}\norm{\this\optu-\this{\realoptu}}_{\eta_k(M_k + \GrowthOpMono_k)}^2
            \le
            \frac{1+\Perturbr}{2\YoungCoeffIII}\eta_k\criticalprox_k^2.
        \end{aligned}
    \]
    In the final step we have used \eqref{eq:pd:apriori:existence-bound} from \cref{lemma:pd:apriori}.
    Inserting this result in \eqref{eq:pd:aposteriori:5} and using $\GrowthOp_N \le \GrowthOpMono_N$ from \cref{lemma:pd:positivity}, we obtain
    \[
        \frac{1}{2}\norm{u^n-\optu^n}_{ \eta_{n}(\Precond_{n}+\GrowthOp_{n})}^2
        \le
        \frac{1}{2}\norm{ u^0-{\opt u^0}}_{\eta_{0}(\Precond_{0}+ \GrowthOp_{0})}^2
        +
        \sum_{k=1}^n
        \left(
            \epsilon_{k}^{\dagger} + \eta_k\ErrMono_k
            +
            \frac{1+\Perturbr}{2\EkLipCoeffCoeff}\eta_k\criticalprox_k^2
        \right).
    \]
    Since $\Perturbr < 1$, this directly proves the tracking estimate \eqref{eq:pd:aposteriori:claim:tracking}.
    Using the assumed  \eqref{eq:pd:aposteriori:init-cond}, with the definition of $d_N$ from \cref{ass:pd:add-local}, we also obtain \eqref{eq:pd:aposteriori:0} for $k=n+1$, finishing the induction step.
    As mentioned in the beginning of the proof, it now follows from \cref{lemma:pd:apriori} that $x^k, \breve x^k \in B(\optx^k,\delta)$ for all $1 \le k \le N$, so that \eqref{eq:pd:aposteriori:claim:growth} follows from \cref{ass:pd:main-local}\,\cref{item:pd:main-forwardstep-local}.
\end{proof}

With this, we are now ready to show the main theorems. For brevity, we define the initialisation and cumulative prediction and gradient error
\begin{gather}
    \label{eq:pd:cumulative-errors}
    e_N^\Sigma(u^{0:N-1},\opt u^{0:N}) \defeq
    \sum_{k=0}^{N-1}\left(
        \epsilon_{k+1}^{\dagger}(u^k,\optu^{k:k+1}) + \eta_{k+1}\Err_{k+1}
    \right)
\shortintertext{and the Lagrangian duality gap}
    \nonumber
    \GenGap_{k}^H(u,\optu)
    \defeq
    \bigl(
        [F_{k}+E_{k}](x) + \iprod{K_{k}x}{\opty} - G_{k}^*(\opty)
    \bigr)
    -\bigl(
        [F_{k}+E_{k}](\optx) + \iprod{K_{k}^*y}{\optx} - G_{k}^*(y)
    \bigr).
\end{gather}
The next theorem shows that the sum of the Lagrangian duality gaps is bounded by the initialisation and the cumulative prediction error.
\begin{theorem}
    \label{thm:pd:main-gap}
    Let $N \ge 1$ and suppose \cref{ass:pd:main} holds for $u^{1:N}$ generated by \cref{alg:pd:alg} for an initial $u^0 \in X_0 \times Y_0$.
    Let $\optu^{0:N} \in \PDpredictConstr^{0:N}$, and suppose that either
    \begin{enumerate}[label=(\roman*),nosep]
        \item\label{item:pd:main-gap:global}
        The global \cref{ass:pd:main-global} holds, or
        \item\label{item:pd:main-gap:local}
        The local \cref{ass:pd:main-local,ass:pd:add-local} hold, and, for $d_N$ given in \cref{ass:pd:add-local}, we have the local initialisation bound
        \begin{equation}
            \label{eq:pd:main-gap:local:init}
            \frac{1}{2}\norm{ u^0-{\opt u^0}}_{\eta_{0}(\Precond_{0}+ \GrowthOp_{0})}^2
            \le
            d_N(\optu^{0:N}).
        \end{equation}
    \end{enumerate}
    Then $\Precond_k+ \GrowthOp_k$ and $M_k -\LossOp_k$ are positive semi-definite and with $\this \pdpredict \defeq P_k(\prevu)$,
    \begin{equation*}
        \sum_{k=1}^{N}\left( \GenGap^H_{k}(\thisu,\this\optu) + \frac{1}{2}\norm{\thisu-\this \pdpredict}_{\eta_{k}(\Precond_{k} - \LossOp_{k})}^2\right)
        \le
        \frac{1}{2}\norm{u^0-\optu^0}^2_{\eta_0(\Precond_0+ \GrowthOp_0)}
        + e_N^\Sigma(u^{0:N-1},\optu^{0:N}).
    \end{equation*}
\end{theorem}

\begin{proof}
    \Cref{lemma:pd:positivity} proves the positive semi-definiteness of $\eta_k\Precond_k$ and $\eta_k (M_k -\LossOp_k)$. Consequently, also $\eta_k\Precond_k+ \GrowthOp_k$ is positive semi-definite.

    Fix, for now, $k \in \lbrace 1, \dots, N-1\rbrace$. Recall the expression for $\widetilde H_k$ from \eqref{eq:pd:m}. In the case \ref{item:pd:main-gap:local}, conditions of \cref{lemma:pd:aposteriori} hold, so we have that
    \begin{equation}
        \label{eq:pd:main-gap:three-point}
        \iprod{\estgrad_k(\this\primalpredict)}{\thisx-\this{\optx}}
        \ge
        E_k(\thisx)-E_k(\this{\optx})  - \frac{\EkLoss}{2}\norm{\thisx-\this\primalpredict}^2 + \frac{\EkGrowth}{2}\norm{\this\optx -\thisx}^2
        - \Err_k,
    \end{equation}
    whenever $\this\optx$ is a component of $\optu^{1:N} = ((\optx^1, \opty^1), \ldots, (\optx^N, \opty^N)) \in \PDpredictConstr_{1:N}$.
    The (strong) convexity of $F_k$ and $G_k^*$ with \eqref{eq:pd:main-gap:three-point} yield
    \begin{gather}
        \label{eq:pd:main-gap:0}
        \begin{aligned}[t]
            \iprod{\widetilde H_k(\thisx,\thisy)}{\thisu-\this{\optu}}
            &
            = \iprod{\subdiff{F_k(\thisx)}}{\thisx - \this\optx} +  \iprod{\estgrad_k(\this\primalpredict)}{\thisx - \this\optx}
            \\
            \MoveEqLeft[-1]
            + \iprod{\subdiff G_k^*(\thisy)}{\thisy-\this\opty}
            + \iprod{K_k^*\thisy}{\thisx-\this\optx}
            - \iprod{K_k\thisx}{\thisy - \this\opty}
            \\
            \MoveEqLeft[9]
            \ge \left(F_k(\thisx) - F_k(\this\optx) + \frac{\gamma_k +\EkGrowth}{2}\|\thisx - \this\optx\|^2\right)
            + \left(E_k(\thisx) - E_k(\this\optx) - \frac{\EkLoss}{2}\|\this\primalpredict - \thisx\|^2 \right)
            \\[-0.5ex]
            \MoveEqLeft[8]
            + \left(G_k^* (\thisy) - G_k^*(\this\opty) + \frac{\rho_k}{2}\|\thisy-\this\opty\|^2\right)
            - \iprod{K_k^*\thisy}{\this\optx} +  \iprod{K_k\thisx}{\this\opty}
            -\Err_k
            \\
            \MoveEqLeft[9]
            = \frac{1}{2}\norm{\thisu-\this\optu}_{\GrowthOp_k}^2 + \GenGap^H_k(\thisx,\this\optu)  - \frac{1}{2}\norm{\thisu-\this\pdpredict}_{\LossOp_k}^2
            -\Err_k.
        \end{aligned}
    \end{gather}
    The case \cref{item:pd:main-gap:global} assumes \eqref{eq:pd:main-gap:three-point} directly, thus \eqref{eq:pd:main-gap:0} holds globally.

    In both cases, for all $k = 1,\ldots, N$, we apply the linear “testing operator” $\iprod{\freevar}{\thisu-\this{\optu}}$ to both sides of \eqref{eq:ppext-pdps}. This followed by \eqref{eq:pd:main-gap:0} yields
    \[
        \Err_k
        \ge
        \iprod{\thisu-\this\pdpredict}{\thisu-\this{\optu}}_{\Precond_k}
        + \frac{1}{2}\norm{\thisu-\this\optu}_{ \GrowthOp_k}^2
        + \GenGap^H_k(\thisu,\this\optu)
        - \frac{1}{2}\|\thisu-\this\pdpredict\|_{\LossOp_k}^2.
    \]
    Multiplying this by $\eta_k$, using the Pythagoras' identity to convert the inner product into norms, then continuing with the definition \eqref{eq:pd:prediction-error} of the prediction errors, we obtain
    \begin{multline*}
        \frac{1}{2}\norm{\prevu-\prev{\optu}}_{\eta_{k-1}(\Precond_{k-1}+ \GrowthOp_{k-1})}^2
        + \epsilon_{k}^{\dagger}(\optu^{k-1:k}) + \eta_k\Err_k
        \ge
        \frac{1}{2}\norm{\this\pdpredict-\this{\optu}}_{\eta_k \Precond_k}^2
        + \eta_k \Err_k
        \\
        \ge
        \frac{1}{2}\norm{\thisu-\this{\optu}}_{\eta_{k}(\Precond_{k}+ \GrowthOp_{k})}^2
        + \eta_k\GenGap^H_{k}(\thisu,\this\optu)
        + \frac{1}{2}\norm{\thisu-\this\pdpredict}_{\eta_k(\Precond_k-\LossOp_k)}^2.
    \end{multline*}
    Summing over $k=1,\dots, N$ gives
    \[
        \begin{aligned}
            \frac{1}{2}\norm{u^0-\optu^0}^2_{\eta_0(\Precond_0+ \GrowthOp_0)}
            +
            e_N^\Sigma(u^{0:N-1}, \optu^{0:N})
            &
            \ge
            \frac{1}{2}\norm{u^N-{\optu^N}}_{\eta_{N}(\Precond_{N}+ \GrowthOp_{N})}^2
            \\[-1ex]
            \MoveEqLeft[-1]
            +
            \sum_{k=1}^{N}
            \Bigl(
                \eta_k\GenGap^H_{k}(\thisu,\this\optu)
                +
                \frac{1}{2}\norm{\thisu-\this \pdpredict}_{\eta_{k}(\Precond_{k} - \LossOp_{k})}^2
            \Bigr).
        \end{aligned}
    \]
    Since $\frac{1}{2}\norm{u^N-{\optu^N}}_{\eta_{N} \Precond_{N}+ \GrowthOp_{N}}^2 \ge 0$, the claim follows.
\end{proof}

The next corollary derives function value estimates from the preceding gap estimates.
Its proof is exactly as the proof of \cite[Theorem 2.6]{better-predict}.
The estimates are with respect to
\begin{gather}
    \label{eq:pd:tildeg}
    \mathring G_{1:N}(z^{1:N})
    \defeq
    \sum_{k=1}^N
    \sup_{\tilde y^{k} \in  \DpredictConstr_{k}}
    \bigl[
        \iprod{z^{k}}{\tilde y^{1:N}} - \eta_k G_{k}^*(\tilde y^{k})
    \bigr]
    \quad\text{in place of}\quad
    \nonumber
    G_{1:N}(y^{1:N}) \defeq \sum_{k=1}^{N} \eta_k G_{k}(\thisy),
\shortintertext{We also denote}
    \nonumber
    Q_{1:N}(x^{1:N}) \defeq \sum_{k=1}^{N} \eta_k [F_{k}  + E_{k}](\thisx),
    \qquad
    K_{1:N}x^{1:N} \defeq (\eta_1 K_{1} x^1, \ldots, \eta_N K_N x^N).
\end{gather}
If the dual comparison sets $\DpredictConstr_{1:N}$ were convex, then, recalling the formula $(f_1 + f_2)^* = f_1^* \infconv f_2^*$ for infimal convolutions (denoted $\infconv$) of convex functions $f_1$ and $f_2$, we would have
$
    \mathring G_{1:N} = G_{1:N} \infconv \delta_{\DpredictConstr_{1:N}}^*.
$
In general,
$\mathring G_{1:N} \le G_{1:N} \infconv \delta_{\DpredictConstr_{1:N}}^*$.
That is, $\mathring G_{1:N}$ is a “sub-infimal” convolution of $G_{1:N}$ and the temporal coupling. As $G_k \circ K_k$ is typically a total variation type regularisation functional, $\mathring G_{1:N} \circ K_{1:N}$  becomes a spatiotemporal regulariser with aspects of spatial total variation, and the temporal properties of the problem at hand.

\begin{corollary}
    \label{cor:pd:main}
    Suppose that the assumptions of \cref{thm:pd:main-gap} hold, with the initialisation bound \eqref{eq:pd:main-gap:local:init} for all $\optu^{0:N} \in \PDpredictConstr^{0:N}$ in the local convergence case \cref{item:pd:main-gap:local}. Then
    \begin{multline*}
        [Q_{1:N}(x^{1:N}) + \mathring G_{1:N}(K_{1:N}x^{1:N})]
        - \sup_{\optx^{1:N} \in  \PpredictConstr_{1:N}}
            [Q_{1:N}(\optx^{1:N}) + \mathring G_{1:N}(K_{1:N}\optx^{1:N})]
        \\
        \le
        \sup_{\optu^{0:N}\in  \PDpredictConstr_{0:N}} \bigg( \frac{1}{2}\norm{u^0-\optu^0}^2_{\eta_0\Precond_0+\Gamma_0} + c_N(\optx^{1:N}, y^{1:N}) + e_N^\Sigma(u^{0:N-1}, \optu^{0:N}) \bigg),
    \end{multline*}
    where $e_N^\Sigma(u^{0:N-1}, \optu^{0:N})$ is given by \eqref{eq:pd:cumulative-errors}, and the \term{comparison set solution discrepancy}
    \[
        c_N(\optx^{1:N}, y^{1:N}) \defeq  \inf_{\tilde y^{1:N} \in \DpredictConstr_{1:N}} \iprod{K_{1:N}\optx^{1:N}}{y^{1:N}-\tilde y^{1:N}} + G_{1:N}^*(\tilde y^{1:N}) - G_{1:N}^*(y^{1:N}).
    \]
\end{corollary}

\begin{proof}
    The claim follows from \cref{thm:pd:main-gap}, by the exact same proof as that of \cite[Theorem 2.6]{better-predict}, with $e_N^\Sigma(u^{0:N-1},\optu^{0:N})$ in place of $e_N(u^{0:N-1},\optu^{0:N})$ of the latter.
\end{proof}

\begin{remark}[Comparison set solution discrepancy]
    According to \cite[Remark 2.7 and Section 3]{better-predict}, $c_N \le 0$ when $G_k \circ K_k=\alpha \norm{\grad\freevar}_{2,1}$ is the total variation, the dual initialisation achieves the total variation (i.e., $\iprod{y^0}{x^0}=\alpha \norm{\grad x^0}_{2,1}$ with $\norm{y^0}_{2,\infty} \le \alpha$), and the dual predictor is total variation preserving.
    Examples of total variation preserving predictors are provided in \cite{better-predict}.
\end{remark}

\section{The online EIT problem}
\label{sec:eit}

We now treat the online EIT problem, which we describe in \cref{sec:eit-descr}.
To apply \cref{cor:bt:Ebounds}, we need to prove the smoothness inequalities of \cref{ass:pd:main-global,ass:pd:main-local}. Based on auxiliary results from \cref{app:dataterm}, we do this in \cref{ssec:tpi} after first proving the necessary second-order differentiability of the CEM solution operator in \cref{sec:eit-diff}.

\subsection{Problem description}
\label{sec:eit-descr}

To model the dynamic EIT problem as the optimisation problem \eqref{eq:pd:problem}, we now take
\begin{equation}
    \label{eq:eit:functionals}
    \begin{aligned}[t]
        E_k(\conductivity) & \defeq \frac{1}{2}\sum_{j=1}^{N_2}\norm{\WOp (I(\conductivity,U^{j,k}) - \EITmeas^{j,k})}_2^2,
        \\[-1ex]
        F_k(\conductivity) & \defeq \delta_{[\conductivity_m,\conductivity_M]}(\conductivity),
        \quad\text{and}\quad
        G_k(K_k\conductivity) \defeq \alpha \norm{\conductivity}_{2,1},
    \end{aligned}
\end{equation}
where $\WOp \in \R^{(N_1-1) \times (N_1-1)}$ is a data precision matrix, modelling noise characteristics, and $\alpha > 0$ is the regularisation parameter.
The conductivity $x$ is bounded between $0 < \conductivity_m < \conductivity_M$. The term $\EITmeas^{j,k} \in \R^{N_1}$ is a vector of measurements corresponding to the electrode potentials $U^{j,k} \in \R^{N_1}$ at time instance $k$. The currents $I_i^{j,k} \defeq I_i(\conductivity,U^{j,k})$ for $i=1,\ldots,N_1$ are obtained by solving $(u^{j,k}, I_1, \dots, I_{N_1})$ from weak versions of the Complete Eletrode Model (CEM) equations
\begin{subequations}\label{eq:eit:cem}
    \begin{alignat}{4}
        \label{eq:eit:cem1}
        \snabla \cdot (\conductivity^k(\xi) \snabla u^{j,k}(\xi)) &=0  &&
        \quad\text{for } \xi\in \Omega, \\
        \label{eq:eit:cem2}
        u^{j,k}(\xi) + \zeta_{i} \conductivity^k(\xi){\snabla u^{j,k}(\xi)} \cdot{\nu(\xi)} &= U^{j,k}_{i}   &&
        \quad\text{for } \xi\in \partial \Omega_{e_{i}}\text{,}\;\; i=1,\dots,N_1, \\
        \label{eq:eit:cem4}
        \conductivity^k(\xi){\snabla u^{j,k}(\xi)}\cdot{\nu(\xi)} &=0   &&
        \quad\text{for } \xi\in \partial \Omega \setminus (\partial \Omega_{e_1} \union \ldots \union \partial \Omega_{e_{N_1}}), \\[-1ex]
        \label{eq:eit:cem3}
        \int_{\partial \Omega_{e_i}} \conductivity^k(\xi) {\snabla u^{j,k}(\xi)}\cdot{\nu(\xi)}\: \dif\BoundMeas &= -I^{j,k}_{i} &&
        \quad\text{for } i=1,\dots,N_1,
    \end{alignat}
\end{subequations}
where $u^{j,k}$ is the inner potential and $\zeta_{i}$ is the contact impedance between the electrode $i$ and the medium inside the domain $\Omega \subset \R^d$. We focus on this “potential-to-current” model, as it aligns with state-of-the-art EIT measurement devices \cite{jauhiainen2020ripgn}. This model was originally introduced in \cite{voss2018imaging}, and subsequently used in e.g. \cite{voss2019three,jauhiainen2020non,jauhiainen2020ripgn}.
The conventional “current-to-potential” model solves the potentials given the electric currents.

We assume that $x \in L^\infty(\Omega)$ for a bounded Lipschitz domain\footnote{The domain $\Omega$ is unrelated to the operator $\LossOpGlobal_k$ of \cref{sec:pd}.} $\Omega \subset \R^n$.
Given the discretisation of $x$ in the numerical experiments of \cref{sec:numerical}, the results of this section remain compatible with the theory of \cref{sec:pd}.
For a given $x$, we write
\begin{gather}
    \label{eq:eit:tuples}
    \CS{\conductivity} = (\CSf{\conductivity},\CSb{\conductivity})
    \in \Hs\defeq H^1(\Omega) \oplus \R^{N_1}
\intertext{for the electrical potential and electrode currents $(u,I)$ that solve \eqref{eq:eit:cem} weakly, i.e.,}
    \label{eq:eit:bilin}
    B_\conductivity(\CS{\conductivity},w) = L(w)
    \quad\text{for all}\quad
    w=(v,V) \in \Hs.
\shortintertext{where the bilinear form $B_x$ reads}
    \nonumber
    B_\conductivity(\CS{\conductivity},w)
    = \int_{\Omega} \conductivity \snabla \CSf{\conductivity} \cdot \snabla v \dif\Lmeas + \sum_{i=1}^{N_1} \frac{1}{\zeta_i}\int_{\partial \Omega_{e_i}} \CSf{\conductivity} (v - V_i)\dif\BoundMeas
        + \sum_{i=1}^{N_1} (\CSb{\conductivity})_i V_i,
\shortintertext{and the linear form $L$ is}
    \nonumber
    L(w) = \sum_{i=1}^{N_1} \frac{1}{\zeta_i} \int_{\partial \Omega_{e_i}} U_i(v-V_i) \dif\BoundMeas.
\shortintertext{We equip $\Hs$ with the norm}
    \label{eq:eit:H-norm}
    \norm{w}^2_{\Hs}\defeq  \norm{v }_{H^1}^2 + \norm{V}_{2}^2
    \quad\text{for}\quad w \in \Hs.
\end{gather}

\subsection{Differentiability of the EIT solution operator}
\label{sec:eit-diff}

The problem \eqref{eq:eit:bilin} is well-posed \cite{somersalo1992existence}: see the “potential-to-current” model in \cite{jauhiainen2021nonplanar}.
The first order differentiability of the reversed “current-to-potential” model has, moreover, been extensively discussed in earlier works, e.g., \cite{jin2010analysis,darde2022contact}.
We will prove the first order differentiability of the solution operator of the potential-to-current model. This, to our knowledge, has not been previously proven. We then show the Lipschitz boundedness of this derivative and use it to prove second order differentiability in $L^\infty$. This has only been shown in the finite dimensional case for the current-to-potential model.

For notational clarity, we write $\dCS{\conductivity}{}$, i.e., $h \mapsto \dCS{\conductivity}{h}$ for the (Fréchet) derivative of $x \mapsto w_{\conductivity}$ at $\conductivity \in L^\infty(\Omega)$, when it exists. The notation for $v_{\conductivity}$ is analogous.
Thus, $h \mapsto \dCS{\conductivity}{h} \in \linear(L^\infty(\Omega); \Hs)$.
Likewise, $\ddCS{\conductivity}{}{}$, i.e., $(h_1, h_2) \mapsto \ddCS{\conductivity}{h_1}{h_2}$ is the second (Fréchet) derivative of $x \mapsto w_{\conductivity}$ at $\conductivity$, when it exists.
For clarity, we write $\snabla w_{\conductivity}$ for the weak gradient of $w_{\conductivity}$, and $\snabla \dCS{\conductivity}{h}$ for $\snabla[\xi \mapsto \dCS{\conductivity}{h}(\xi)]$, i.e., $\snabla$ is always a (spatial) gradient with respect to $\xi \in \Omega$.
Given $\conductivity \in [\conductivity_m,\conductivity_M]$, a.e., and that the domain $\Omega$ is Lipschitz, the solutions $\CS{\conductivity}: X \to \Hs$ are continuous in $L^p$ for any $p \ge 1$ \cite[Remark 2.8]{jauhiainen2022mumford}. Moreover, under the domain scaling condition $\zeta^{-1}_k\abs{\partial \Omega_{e_k}} \le 1$, we have
\begin{align}
    \label{eq:eit:coercivity}
    C_1B_\conductivity(w,w) &\ge \norm{w}^2_{\Hs}
    \quad\text{for all}\quad w \in \Hs\quad\text{and}
    \\
    \label{eq:eit:boundedness}
    \norm{\CS{\conductivity}}_{\Hs} &\le C_2\norm{U}_{2}\quad\text{for any } \conductivity \in [\conductivity_m,\conductivity_M].
\end{align}
For solutions $\CS{\conductivity} \in \Hs$ to \eqref{eq:eit:bilin}, the scaling condition $\zeta^{-1}_k\abs{\partial \Omega_{e_k}} \le 1$ can be removed (see \cite[Lemma 2.6 and Theorem 2.7]{jauhiainen2022mumford}), and the coefficients then have values
\[
    C_1 = 2{C_\Omega}(\Lambda \min \lbrace 1, \conductivity_m \rbrace)^{-1}
    \quad\text{and}\quad
    C_2 =\sqrt{2}{C_1}
    \quad\text{for}\quad
    C_\Omega =(\zeta_m/e_M)^{-\frac{N-2}{N-1}}
\]
with $\zeta_m = \min_k \zeta_k$ and $e_M = \max_k \abs{\partial \Omega_{e_k}}$.\footnote{We know that $\Lambda$ is a finite positive constant but the exact value is unknown. For more details, see \cite[Lemma 3.2]{somersalo1992existence} for current-to-potential model and potential-to-current model \cite[Lemma 2]{jauhiainen2020non}.}
Throughout, we work with:

\begin{assumption}
    \label{ass:eit}
    $0 < x_m < x_M < \infty$ and $\Omega \subset \R^d$ is a Lipschitz domain.
\end{assumption}

The next corollary is a simple application of this and the earlier results of \cite{jauhiainen2022mumford}.

\begin{corollary}
    \label{cor:eit:solbound}
    Let \cref{ass:eit} hold.
    Then for $C_3 \defeq C_1C_2\norm{U}_2$ we have
    \[
        \norm{\CS{\conductivity_2}- \CS{\conductivity_1}}_\Hs \le C_3\norm{\conductivity_2-\conductivity_1}_{\infty}
        \quad\text{for all}\quad x_1, x_2 \in [x_m, x_M].
    \]
\end{corollary}

\begin{proof}
    Using Hölder inequality, \cref{eq:eit:coercivity,eq:eit:boundedness}, recalling that we write $w_{x_1}=(v_{x_1}, V_{x_1})$ and $w_{x_2}=(v_{x_2}, V_{x_2})$, we estimate
    \begin{multline*}
        \norm{\CS{\conductivity_2} - \CS{\conductivity_1}}^2_{\Hs}
        \le C_1 B_{\conductivity_1}(\CS{\conductivity_2} - \CS{\conductivity_1}, \CS{\conductivity_2} - \CS{\conductivity_1})
        \le C_1\int_\Omega (\conductivity_1-\conductivity_2)(\snabla \CSf{\conductivity_2} \cdot \snabla (\CSf{\conductivity_2} - \CSf{\conductivity_1})) \dif \Lmeas
        \\
        \le C_1\norm{\conductivity_1-\conductivity_2}_\infty \norm{\CS{\conductivity_2}}_\Hs \norm{\CS{\conductivity_2} - \CS{\conductivity_1}}_{\Hs}
        \le C_1C_2 \norm{U}_{2}\norm{\conductivity_2-\conductivity_1}_\infty \norm{\CS{\conductivity_2} - \CS{\conductivity_1}}_{\Hs}.
    \end{multline*}
    For a detailed derivation of the second inequality, see \cite[equation (20)]{jauhiainen2022mumford})
\end{proof}

As promised, we can now show first-order differentiability in $L^\infty(\Omega)$.

\begin{lemma}
    \label{lemma:eit:frechet-first}
    Let \cref{ass:eit} hold.
    Then the solution map $x \mapsto \CS{\conductivity}: L^\infty(\Omega)\to \Hs$ of \eqref{eq:eit:bilin} is Fréchet differentiable at any $\conductivity \in [x_m, x_M]$ with the Fréchet derivative at $x$, i.e., the map $(h \mapsto \dCS{\conductivity}{h}) \in \linear(L^\infty(\Omega); \Hs)$ norm-bounded by $C_3=C_1C_2\norm{U}_2$ and satisfying
    \begin{gather}
        \label{eq:eit:frechet-first:weak-eq}
        B_\conductivity( \dCS{\conductivity}{h},w) = - \int_\Omega h \snabla \CSf{\conductivity} \cdot \snabla v \dif \Lmeas
    \shortintertext{and}
        \nonumber
        \norm{\CS{\conductivity+h} -\CS{\conductivity} - \dCS{\conductivity}{h}}_\Hs \le C_1C_3 \norm{h}_{\infty}^2
        \quad\text{for all}\quad
        h\in L^\infty(\Omega) \ \text{and}\ w=(v,V) \in \Hs.
    \end{gather}
\end{lemma}

\begin{proof}
    Suppose that $\CS{\conductivity}$ is a solution to \eqref{eq:eit:bilin} with $\conductivity$. Now, with respect to $w=(v,V)$, the left-hand side of \eqref{eq:eit:frechet-first:weak-eq} is coercive by \eqref{eq:eit:coercivity} and the right-hand side is clearly linear and bounded.
    Thus the Lax-Milgram theorem establishes the existence of $\dCS{x}{h}$.
    We will show that the mapping $h \mapsto \dCS{x}{h}$ is bounded and linear with respect to $h$, and then proceed to confirm that it is, indeed, the Fréchet derivative, i.e., that
    \begin{equation}
        \label{eq:eit:frechet-first:lim}
        \lim_{\norm{h}_\infty \to 0}\frac{\norm{\CS{\conductivity_0+h} - \CS{\conductivity_0} - \dCS{\conductivity_1}{h}}_{\Hs}}{\norm{h}_{L^\infty(\Omega)}} = 0.
    \end{equation}

    To see linearity, take $a,b \in \R$ and notice from \eqref{eq:eit:frechet-first:weak-eq} that, due to the linearity of the right-hand side with respect to $h$, both $\dCS{\conductivity}{ah_1 + bh_2}$ and $a\dCS{\conductivity}{h_1}+ b\dCS{\conductivity}{h_2}$ solve \eqref{eq:eit:frechet-first:weak-eq}. Hence, due the well-posedness \eqref{eq:eit:frechet-first:weak-eq}, $\dCS{\conductivity}{ah_1 + bh_2},w$ and $a\dCS{\conductivity}{h_1}+ b\dCS{\conductivity}{h_2}$ must be equal.

    To show boundedness, we observe that by the Hölder inequality and \cref{eq:eit:coercivity,eq:eit:frechet-first:weak-eq,eq:eit:boundedness},
    \begin{equation}
        \label{eq:eit:frechet-first:bound}
        \norm{\dCS{\conductivity}{h}}_\Hs^2
        \le
        C_1\abs{B_\conductivity(\dCS{\conductivity}{h},\dCS{\conductivity}{h})}
        =
        \adaptabs{\int_\Omega h \snabla \CSf{\conductivity} \cdot \snabla \dCS{\conductivity}{h}  \dif \Lmeas}
        \le
        C_1C_2\norm{U}_{2}\norm{h}_\infty\norm{\dCS{\conductivity}{h}}_\Hs.
    \end{equation}

    Finally, to confirm \eqref{eq:eit:frechet-first:lim}, suppose that $\norm{h}_\infty \le \conductivity_m/2$. This ensures that the solution $\CS{\conductivity+h}$ to \eqref{eq:eit:bilin} at $\conductivity+h$ exists since $0<\conductivity_m/2 \le \conductivity+h \le \conductivity_M + \conductivity_m/2$. By \eqref{eq:eit:coercivity}
    \begin{equation}
        \label{eq:eit:frechet-first:bilin}
        \norm{\CS{\conductivity+h} -\CS{\conductivity} - \dCS{\conductivity}{h}}^2_\Hs \le C_1 {B_\conductivity(\CS{\conductivity+h} -\CS{\conductivity} - \dCS{\conductivity}{h},\CS{\conductivity+h} -\CS{\conductivity} - \dCS{\conductivity}{h})}.
    \end{equation}
    Moreover, for any $w \in \Hs$,
    \begin{gather*}
        \begin{aligned}[t]
            B_\conductivity(\CS{\conductivity+h} -\CS{\conductivity},w)
            &
            = B_\conductivity(\CS{\conductivity+h},w) - B_\conductivity(\CS{\conductivity},w)
            = B_\conductivity(\CS{\conductivity+h},w) - L(w)
            \\
            &
            = B_\conductivity(\CS{\conductivity+h},w) - B_{\conductivity+h}(\CS{\conductivity+h},w)
            = -\int_\Omega h \snabla \CSf{\conductivity+h} \cdot \snabla v \dif \Lmeas,
        \end{aligned}
        \\[-1ex]
    \shortintertext{hence}
        B_\conductivity(\CS{\conductivity+h} -\CS{\conductivity} - \dCS{\conductivity}{h},w)
        =
        B_\conductivity(\CS{\conductivity+h} -\CS{\conductivity},w) - B_\conductivity(\dCS{\conductivity}{h},w)
        =
        -\int_\Omega h \snabla (\CSf{\conductivity+h} - \CSf{\conductivity}) \cdot \snabla v \dif \Lmeas.
    \end{gather*}
    Taking $w = \CS{\conductivity+h} -\CS{\conductivity} - \dCS{\conductivity}{h}$ and using \cref{cor:eit:solbound} we thus estimate
    \begin{equation}
        \label{eq:eit:frechet-first:lim0}
        \begin{aligned}[t]
            B_\conductivity(w, w)
            &
            = -{\int_\Omega h \snabla (\CSf{\conductivity+h} - \CSf{\conductivity}) \cdot \snabla (\CSf{\conductivity+h} -  \CSf{\conductivity} - \dCSf{\conductivity}{h})\dif \Lmeas}
            \\[-0.5ex]
            &
            \le \norm{h}_{\infty}\adaptabs{\int_\Omega \snabla (\CSf{\conductivity+h} - \CSf{\conductivity}) \cdot \snabla (\CSf{\conductivity+h} -  \CSf{\conductivity} - \dCSf{\conductivity}{h})\dif \Lmeas}
            \\
            &
            \le \norm{h}_{\infty}\norm{\CS{\conductivity+h} - \CS{\conductivity}}_\Hs \norm{\CS{\conductivity+h} - \CS{\conductivity} - \dCS{\conductivity}{h}}_\Hs
            \\
            &
            \le \norm{h}_{\infty}C_3\norm{h}_{\infty} \norm{\CS{\conductivity+h} - \CS{\conductivity} - \dCS{\conductivity}{h}}_\Hs.
        \end{aligned}
    \end{equation}
    Combining \eqref{eq:eit:frechet-first:bilin} and \eqref{eq:eit:frechet-first:lim0} yields
    $
        \norm{\CS{\conductivity+h} -\CS{\conductivity} - \dCS{\conductivity}{h}}_\Hs \le C_1C_3 \norm{h}_{\infty}^2
    $,
    which proves \eqref{eq:eit:frechet-first:lim}.
\end{proof}

The following lemma shows that the Fréchet derivative of the solution map is Lipschitz.
This will be needed to show the second-order differentiability.

\begin{lemma}
    \label{eq:eit:lipschitz-first}
    Let \cref{ass:eit} hold.
    Then, for any given $h \in L^\infty(\Omega)$, the map $\conductivity \mapsto \dCS{\conductivity}{h}: L^\infty(\Omega)\to \Hs$ is Lipschitz with constant $2C_1C_3\norm{h}_\infty$, that is
    \[
        \norm{\dCS{\conductivity_1}{h}-\dCS{\conductivity_2}{h}}_\Hs
        \le
        2C_1C_3\norm{h}_\infty \norm{\conductivity_2-\conductivity_1}_{\infty}
        \quad\text{for all}\quad \conductivity_1,\conductivity_2 \in [x_m, x_M].
    \]
\end{lemma}

\begin{proof}
    For any $w_1,w_2 \in \Hs$ we have
    \begin{equation}
        \label{eq:eit:lipschitz-first:0}
        B_{\conductivity_1}(w_1,w_2) = B_{\conductivity_2}(w_1,w_2)  + \int_{\Omega} (\conductivity_1 - \conductivity_2) \snabla v_1 \cdot \snabla v_2 \dif \Lmeas.
    \end{equation}
    Furthermore, by \cref{lemma:eit:frechet-first}, $\dCS{\conductivity_1}{h}$ and $\dCS{\conductivity_2}{h}$ satisfy
    \[
        B_{\conductivity_1}(\dCS{\conductivity_1}{h},w)
        =
        - \int_\Omega h\snabla \CSf{\conductivity_1} \cdot \snabla v \dif \Lmeas
    \quad\text{and}\quad
        B_{\conductivity_2}(\dCS{\conductivity_2}{h},w)
        = - \int_\Omega h\snabla \CSf{\conductivity_2} \cdot \snabla v \dif \Lmeas
    \]
    for any $w \in \Hs$. Using these and \eqref{eq:eit:lipschitz-first:0} yields
    \begin{equation*}
        \begin{aligned}[t]
            A & := B_{\conductivity_1}(\dCS{\conductivity_1}{h}-\dCS{\conductivity_2}{h}, \dCS{\conductivity_1}{h}-\dCS{\conductivity_2}{h})
            \\
            &
            = B_{\conductivity_1}(\dCS{\conductivity_1}{h}, \dCS{\conductivity_1}{h}-\dCS{\conductivity_2}{h}) - B_{\conductivity_2}(\dCS{\conductivity_2}{h}, \dCS{\conductivity_1}{h}-\dCS{\conductivity_2}{h})
            \\
            \MoveEqLeft[-1]
            - \int_{\Omega} (\conductivity_1 - \conductivity_2) \snabla \dCSf{\conductivity_2}{h} \cdot \snabla (\dCSf{\conductivity_1}{h}-\dCSf{\conductivity_2}{h}) \dif \Lmeas
            \\
            &
            = \int_\Omega [h\snabla (\CSf{\conductivity_2} - \CSf{\conductivity_1}) + (\conductivity_2 - \conductivity_1) \snabla \dCSf{\conductivity_2}{h} ]\cdot \snabla (\dCSf{\conductivity_1}{h}-\dCSf{\conductivity_2}{h}) \dif \Lmeas
            \\
            &
            \le
            \left(\norm{h \snabla (\CSf{\conductivity_2} - \CSf{\conductivity_1})}_2 + \norm{(\conductivity_2 - \conductivity_1) \snabla \dCSf{\conductivity_2}{h}}_2\right) \norm{\snabla (\dCSf{\conductivity_1}{h}-\dCSf{\conductivity_2}{h})}_2.
        \end{aligned}
    \end{equation*}
    Using the Cauchy-Schwartz inequality and $\norm{\snabla v}_2 \le \norm{w}_\Hs$ (see \cref{eq:eit:H-norm}) again followed by \cref{cor:eit:solbound} and $\norm{w'(\conductivity_2)}_\Hs \le C_3\norm{h}_{\infty}\norm{U}_{2}$ allows us to continue
    \begin{equation}
        \label{eq:eit:lipschitz-first:1}
        \begin{aligned}
            A
            &
            \le \norm{h}_\infty \norm{\CS{\conductivity_2}-\CS{\conductivity_1}}_\Hs \norm{\dCS{\conductivity_1}{h}-\dCS{\conductivity_1}{h}}_\Hs
            + \norm{\conductivity_2-\conductivity_1}_{\infty} \norm{\dCS{\conductivity_2}{h}} \norm{\dCS{\conductivity_2}{h}-\dCS{\conductivity_1}{h}}_\Hs
            \\
            &
            \le \norm{h}_\infty C_3\norm{\conductivity_2-\conductivity_1}_\infty \norm{\dCS{\conductivity_1}{h}-\dCS{\conductivity_1}{h}}_\Hs
            + \norm{\conductivity_2-\conductivity_1}_{\infty} C_3\norm{h}_{\infty} \norm{\dCS{\conductivity_1}{h}-\dCS{\conductivity_1}{h}}_\Hs.
        \end{aligned}
    \end{equation}
    Finally, \cref{eq:eit:coercivity,eq:eit:lipschitz-first:1} establish
    \[
        \begin{aligned}[t]
            \norm{\dCS{\conductivity_1}{h}-\dCS{\conductivity_2}{h}}_\Hs^2
            &
            \le
            C_1B_{\conductivity_1}(\dCS{\conductivity_1}{h}-\dCS{\conductivity_2}{h}, \dCS{\conductivity_1}{h}-\dCS{\conductivity_2}{h})
            \\
            &
            \le
            2C_1C_3\norm{h}_\infty \norm{\conductivity_2-\conductivity_1}_{\infty} \norm{\dCS{\conductivity_1}{h}-\dCS{\conductivity_2}{h}}_\Hs.
        \end{aligned}
    \]
    Dividing by $ \norm{\dCS{\conductivity_1}{h}-\dCS{\conductivity_2}{h}}_\Hs$ establishes the claimed Lipschitz constant.
\end{proof}

The next lemma confirms the second-order differentiability of the solution mapping.

\begin{lemma}
    \label{lemma:eit:frechet-second}
    Let \cref{ass:eit} hold and $\conductivity \in [x_m, x_M]$.
    Then $h_1 \mapsto \dCS{\conductivity}{h_1}: L^\infty(\Omega)\to \Hs$ is Fréchet differentiable at all $h_1 \in L^\infty(\Omega)$, the Fréchet derivative $h_2 \to \ddCSf{\conductivity}{h_1}{h_2}$ is norm-bounded in $\linear(L^\infty(\Omega) \times L^\infty(\Omega); \Hs)$  by $C_4=2C_1C_3\norm{U}_{2}$, and satisfies the linear equation
    \begin{equation}
        \label{eq:eit:frechet-second:weak-eqn}
        B_\conductivity(\ddCS{\conductivity}{h_1}{h_2},w) = -  \left( \int_\Omega h_1\snabla \dCSf{\conductivity}{h_2} \cdot \snabla v \dif \Lmeas + \int_\Omega h_2 \snabla \dCSf{\conductivity}{h_1} \cdot \snabla v \dif \Lmeas \right)
    \end{equation}
    for all $h_2\in L^\infty(\Omega)$ and $w=(v,V) \in \Hs$.
\end{lemma}

\begin{proof}
    The right-hand side of \eqref{eq:eit:frechet-second:weak-eqn} is linear and bounded with respect to $w$. Hence by the Lax-Milgram theorem there exists a solution $\ddCS{\conductivity}{h_1}{h_2}$ to \eqref{eq:eit:frechet-second:weak-eqn}.
    We will show that $h_2 \mapsto \ddCS{\conductivity}{h_1}{h_2}$ is the Fréchet derivative of $h_1 \mapsto \dCS{\conductivity}{h_1}$.
    To see that the former is a candidate for the Fréchet derivative, we first establish the boundedness and the bilinearity of $(h_1, h_2) \mapsto \ddCS{\conductivity}{h_1}{h_2}$.
    Indeed, due to \cref{eq:eit:coercivity,eq:eit:frechet-second:weak-eqn},
    \[
        \begin{aligned}
            \norm{\ddCS{\conductivity}{h_1}{h_2}}_\Hs^2
            &
            \le C_1 B_\conductivity(\ddCS{\conductivity}{h_1}{h_2},\ddCS{\conductivity}{h_1}{h_2})
            \\
            &
            \le C_1\left( \adaptabs{  \int_\Omega h_1\snabla \dCSf{\conductivity}{h_2} \cdot \snabla \ddCSf{\conductivity}{h_1}{h_2} \dif \Lmeas}\right.
            \left.
            + \adaptabs{\int_\Omega h_2 \snabla \dCSf{\conductivity}{h_1} \cdot \snabla \ddCSf{\conductivity}{h_1}{h_2} \dif \Lmeas } \right)
            \\
            &
            \le C_1\bigl( \norm{h_1}_\infty \norm{   \dCS{\conductivity}{h_2} }_\Hs + \norm{h_2}_\infty\norm{\dCS{\conductivity}{h_1} }_\Hs\bigr) \norm{  \ddCS{\conductivity}{h_1}{h_2} }_\Hs
            \\
            &
            \le 2C_1C_3\norm{U}_{2} \norm{h_1}_\infty \norm{h_2}_\infty \norm{  \ddCS{\conductivity}{h_1}{h_2} }_\Hs.
        \end{aligned}
    \]
    Clearly then $\norm{\ddCS{\conductivity}{h_1}{h_2}}_\Hs \le 2C_1C_3\norm{U}_{2} \norm{h_1}_\infty \norm{h_2}_\infty$.

    To see bilinearity, first observe that the first integral left-hand side of \eqref{eq:eit:frechet-second:weak-eqn} is linear with respect to $h_1$ and the second integral is linear with respect to $h_2$. Then take $h_2= af + bg$, for $a,b\in \R$ and $f,g\in L^\infty(\Omega)$. Since $h \mapsto \dCSf{\conductivity}{h}$ is linear, we have
    \[
        \begin{aligned}[t]
        \int_\Omega h_1\snabla\dCSf{\conductivity}{af + bg} \cdot \snabla v \dif \Lmeas
        &=
        \int_\Omega h_1\snabla(a\dCSf{\conductivity}{f} + b\dCSf{\conductivity}{g})\cdot \snabla v \dif \Lmeas
        \\
        &
        =
        \int_\Omega h_1(a\snabla \dCSf{\conductivity}{f} \cdot \snabla v) \dif \Lmeas+ \int_\Omega h_1(b\snabla \dCSf{\conductivity}{g} \cdot \snabla v) \dif \Lmeas,
        \end{aligned}
    \]
    This show sthe linearity of the first term with respect to $h_2$. The same reasoning shows the linearity of the second term with respect to $h_1$, confirming the bilinearity of the both terms. Arguing similarly to \cref{lemma:eit:frechet-first}, we conclude that $(h_1, h_2) \mapsto \ddCS{\conductivity}{h_1}{h_2}$ is bilinear.

    Finally, we confirm that $h_2 \mapsto \ddCS{\conductivity}{h_1}{h_2}$ is the Fréchet derivative of $h_1 \mapsto \dCS{\conductivity}{h_1}$.
     Let $\norm{h_2}_\infty \le \conductivity_m/2$ and $h_1 \in L^\infty(\Omega)$. Then
    \begin{equation}
        \label{eq:eit:frechet-second:0}
        \begin{aligned}[t]
            &
            B_\conductivity(\dCS{\conductivity+h_2}{h_1} - \dCS{\conductivity}{h_1} - \ddCS{\conductivity}{h_1}{h_2},w)
            \\
            &
            = B_\conductivity(\dCS{\conductivity+h_2}{h_1},w) - B_\conductivity(\dCS{\conductivity}{h_1},w) - B_\conductivity(\ddCS{\conductivity}{h_1}{h_2},w)
            \\
            &
            = B_{\conductivity+h_2}(\dCS{\conductivity+h_2}{h_1},w) - B_\conductivity(\dCS{\conductivity}{h_1},w)
            -\int_\Omega h_2\snabla \dCSf{\conductivity+h_2}{h_1} \cdot \snabla v \dif \Lmeas  - B_\conductivity(\ddCS{\conductivity}{h_1}{h_2},w)
            \\
            &
            = - \int_\Omega h_1\snabla (\CSf{\conductivity + h_2} -\CSf{\conductivity} )\cdot \snabla v \dif \Lmeas  -\int_\Omega h_2\snabla \dCSf{\conductivity+h_2}{h_1} \cdot \snabla v \dif \Lmeas  - B_\conductivity(\ddCS{\conductivity}{h_1}{h_2},w)
            \\
            &
            = -\int_\Omega h_1\snabla (\CSf{\conductivity + h_2} -\CSf{\conductivity} -  \dCSf{\conductivity}{h_2})\cdot \snabla v \dif \Lmeas  - \int_\Omega h_2\snabla (\dCSf{\conductivity+h_2}{h_1} - \dCSf{\conductivity}{h_1}) \cdot \snabla v \dif \Lmeas,
        \end{aligned}
    \end{equation}
    where on the fourth line we used \cref{lemma:eit:frechet-first} on the first two bilinear terms, and on the last line we used \eqref{eq:eit:frechet-second:weak-eqn}.
    Observe for any $w_1 = (v_1, V_1),w_2=(v_2, V_2) \in \Hs$ that
    \[
        \int_\Omega h\snabla  v_1 \cdot \snabla v_2 \dif \Lmeas \le \norm{h}_\infty\norm{\snabla v_1}_2 \norm{\snabla v_2}_2 \le \norm{h}_\infty\norm{ w_1}_\Hs \norm{ w_2}_\Hs.
    \]
    Using this, \cref{eq:eit:coercivity,eq:eit:frechet-second:0} with $w=\dCS{\conductivity+h_2}{h_1} - \dCS{\conductivity}{h_1} - \ddCS{\conductivity}{h_1}{h_2}$, as well as the triangle inequality, we obtain
    \begin{equation}
        \label{eq:eit:frechet-second:1}
        \begin{aligned}[t]
            &
            \norm{\dCS{\conductivity+h_2}{h_1} - \dCS{\conductivity}{h_1} - \ddCS{\conductivity}{h_1}{h_2}}_\Hs^2
            \\
            &
            \le C_1 \abs{B_\conductivity(\dCS{\conductivity+h_2}{h_1} - \dCS{\conductivity}{h_1} - \ddCS{\conductivity}{h_1}{h_2},\dCS{\conductivity+h_2}{h_1} - \dCS{\conductivity}{h_1} - \ddCS{\conductivity}{h_1}{h_2})}
            \\
            &
             \le C_1 \left( \adaptabs{\int_\Omega h_1\snabla (\CSf{\conductivity + h_2} -\CSf{\conductivity} -  \dCSf{\conductivity}{h_2})\cdot \snabla (\dCSf{\conductivity+h_2}{h_1} -\dCSf{\conductivity}{h_1} -  \ddCSf{\conductivity}{h_1}{h_2}) \dif \Lmeas}\phantom{\int_\Omega}\right.
            \\
            \MoveEqLeft[-2.5]
            +\left.\adaptabs{\int_\Omega h_2\snabla (\dCSf{\conductivity+h_2}{h_1} - \dCSf{\conductivity}{h_1}) \cdot \snabla (\dCSf{\conductivity+h_2}{h_1} -\dCSf{\conductivity}{h_1} -  \ddCSf{\conductivity}{h_1}{h_2}) \dif \Lmeas}\right)
            \\
            &
            \le C_1 \bigl(\norm{h_1}_\infty\norm{\CS{\conductivity + h_2} -\CS{\conductivity} -  \dCS{\conductivity}{h_2}}_\Hs
            +\norm{h_2}_\infty\norm{ \dCS{\conductivity+h_2}{h_1} - \dCS{\conductivity}{h_1}}_\Hs\bigr)
            \\
            \MoveEqLeft[-2.5]
            \cdot \norm{\dCS{\conductivity+h_2}{h_1} - \dCS{\conductivity}{h_1} - \ddCS{\conductivity}{h_1}{h_2}}_\Hs.
        \end{aligned}
    \end{equation}
    By \cref{lemma:eit:frechet-first,eq:eit:lipschitz-first}, we have, respectively
    \[
        \norm{\CS{\conductivity+h_2} -\CS{\conductivity} - \dCS{\conductivity}{h_2}}_\Hs
        \le
        C_1C_3 \norm{h_2}_\infty^2
        \quad\text{and}\quad
        \norm{\dCS{\conductivity}{h_1}-\dCS{\conductivity+h_2}{h_1}}_\Hs
        \le
        2C_1C_3 \norm{h_1}_\infty \norm{h_2}_{\infty}.
    \]
    Dividing  \eqref{eq:eit:frechet-second:1} by $\norm{\dCS{\conductivity+h_2}{h_1} - \dCS{\conductivity}{h_1} - \ddCS{\conductivity}{h_1}{h_2}}_\Hs$ and using these estimates establishes
    \[
        \begin{aligned}
            \norm{\dCS{\conductivity+h_2}{h_1} - \dCS{\conductivity}{h_1} - \ddCS{\conductivity}{h_1}{h_2}}_\Hs
            \le C_1^2C_3(1 + 2\norm{h_1}_\infty)\norm{h_1}_\infty \norm{h_2}_\infty^2 .
        \end{aligned}
    \]
    Dividing by $\norm{h_2}_\infty$ and letting $\norm{h_2}_\infty \downto 0$ shows second-order Fréchet differentiability.
\end{proof}

\subsection{Three-point inequalities}\label{ssec:tpi}

Next, we discuss how \cref{ass:pd:main-local} can be satisfied. We assume that $\PpredictConstr_{k}$ is a finite dimensional space, as will be the case in the numerical realisation of the next section, so that the results of \cref{sec:eit-diff} are compatible with the results of \cref{sec:pd}.
We write $w_x(U^{j,k})$ for the interior potential--electrode current tuples \eqref{eq:eit:tuples} corresponding to multiple electrode potentials $U^{j,k}$ and the conductivity $x$.
Define $S: L^\infty(\Omega) \to R^{N_1N_2}$, by
\begin{equation}
    \label{eq:SolOpk}
    S(\conductivity)=(S_1(\conductivity), \ldots, S_{N_1}(\conductivity))
    \quad\text{for}\quad
    \SolOp_k(\conductivity)
    = \bigl(\WOp Pw_\conductivity(U^{1,k}), \dots, \WOp Pw_\conductivity(U^{N_2,k}) \bigr),
\end{equation}
where $P$ is a projection from $\Hs$ to $\R^{N_1}$ that extracts the electrode currents, i.e., $Pw_\conductivity(U^{j,k}) = I^{j,k}$.
Then \cref{lemma:eit:frechet-first,lemma:eit:frechet-second} show that $S'$ and $S''$ are norm-bounded by
\begin{equation}
    \label{eq:dSkbounds}
    S'_{\max} \defeq N_2C_1C_2\max_{j,k}\norm{U^{j,k}}_2\norm{\WOp}_2
    \text{ and }
    S''_{\max}\defeq 2N_2C_1^2C_2^2\max_{j,k}\norm{U^{j,k}}_2^2\norm{\WOp}_2.
\end{equation}
Writing $\mathcal{I}^k$ for the measurement vector corresponding to \eqref{eq:SolOpk}\footnote{The entries of $\mathcal{I}^k$ are also multiples of $\WOp$.},  $E_k$ of \eqref{eq:eit:functionals} then reads
\begin{equation}
    \label{eq:eit:Ek}
    E_k(\conductivity) \defeq \frac{1}{2}\norm{\SolOp_k(\conductivity) - \mathcal{I}^k}^2.
\end{equation}

We define $\bar B_x w$ as the Riesz representation of $B_x(w, \freevar)$.
Then, minding \eqref{eq:eit:coercivity}, $\bar B_x \in \linear(\Hs; \Hs)$ is invertible with eigenvalues bounded from below by $\inv C_1$.
Also write $l_{x,h} \in \Hs$ for the Riesz representation of the right hand side of \eqref{eq:eit:frechet-first:weak-eq} as a functional of $w$:
$\iprod{l_{x, h}}{w} = - \int_\Omega h \snabla \CSf{\conductivity} \cdot \snabla v \dif \Lmeas$ for any $w=(v,V) \in \Hs$.
Note that here $v_x$ is a component of the solution $w_x=(v_x, W_x) \in \Hs$.
Then, by \cref{lemma:eit:frechet-first}, $\bar B_x w_{x,h}' = l_{x,h}$.
By the invertibility of $\bar B_x$, we can write $w_{x,h}' = \inv{\bar B_x} l_{x,h}$.
Also write $l_{x,h}=L_x h$, where $L \in \linear(L^p(\Omega); \Hs)$, with $1 < p < \infty$.
It follows that
\[
    \norm{\Sigma^{-1/2} P w_{x,h}'}^2
    = \iprod{\Sigma^{-1} P^* {\bar B_x}^{-1} l_{x,h}}{P {\bar B_x}^{-1} l_{x,h}}
    = \iprod{L_x^*{\bar B_x}^{-*} P^* \Sigma^{-1} P {\bar B_x}^{-1} L_x h}{ h}.
\]
Writing $w_{x,h}'(U^{j,k})$ for the differential of $w_x(U^{j,k})$ in the direction $h$ from $x$, and $L_x^{j,k}$ for $L_x$ corresponding to $w_x(U^{j,k})=(v_x(U^{j,k}), V_x(U^{j,k}))$, we thus get
\begin{gather*}
    \norm{S_k'(x,h)}_{\R^{N_1N_2}}^2=
    \sum_{j=1}^{N_2}
    \norm{\Sigma^{-1/2} Pw_{x,h}'(U^{j,k})}_{\R^{N_1}}^2
    =
    \iprod{A_k h}{ h}
    \\[-1ex]
\shortintertext{for}
    A_k \defeq \sum_{j=1}^{N_2} (L_x^{j,k})^*{\bar B_x}^{-*} P^* \Sigma^{-1}  P {\bar B_x}^{-1} L_x^{j,k}.
\end{gather*}

The next proof relies on a lower bound on $A_k$ to model the idea that for a potential measurement setup $U^{j,k}$, ($j=1,\ldots,N_2$), some electrode, indexed by $i=1,\ldots,N_1$, should react to a change $h$ in the conductivity $x$. If $x$ and $h$ are discretised to a finite grid of $n$ nodes, then it seems reasonable that this can be achieved as long as $N_1N_2 > n$. If the latter does not hold, the condition could still be achieved at specific $x$ or for specific directions $h$.
Practically, for the condition to hold, the comparison point $\this\optx$ (e.g., ground-truth) data fit $\norm{\SolOp_k(\this\optx ) - b_k}$ has to be good enough, i.e., for the noise level to be low enough, and the radius $\delta>$ where we seek to satisfy \cref{ass:pd:main-local}, has to be small enough.
This radius affects the closeness requirement of the initial iterate to $\optx^0$ through \cref{ass:pd:add-local} and \eqref{eq:pd:main-gap:local:init}.

\begin{theorem}
    \label{cor:bt:Ebounds}
    Let \cref{ass:eit} hold and $\PpredictConstr_{k}$ be finite dimensional.
    Define $E_k$ by \eqref{eq:eit:Ek}, and suppose $\this\optx \in \PpredictConstr_{k} \isect B(\this\primalpredict, \delta)$, $\this{\hat x} \in B(\this\primalpredict, \delta)$, and $x \in B(\this\optx,{\delta})$ for a $\delta > 0$, and
    \begin{gather}
        \label{eq:bt:Ebounds:ak-cond}
        A_k \ge \max \{c_1,c_2 \} \Id
    \shortintertext{for}
        \nonumber
        \begin{aligned}
            c_1 & \defeq
            4\theta
            +
            S''_{\max}
            (
            4\norm{\SolOp_k(\this\optx ) - b_k}_{R^{N_1N_2}}
            +
            (1+\sqrt{2} + 2S_{\text{max}}'')\delta^2
            )
            \quad\text{and}
            \\
            c_2 & \defeq 2\theta + 2 S''_{\max}\left(\left(\frac{\delta}{8} + S'_{\max}\right)\delta + \norm{S_k(\this\optx ) - b_k}_{R^{N_1N_2}}\right).
        \end{aligned}
    \end{gather}
    Then \cref{ass:pd:main-local}\,\ref{item:pd:main-forwardstep-local0} and \ref{item:pd:main-forwardstep-local}~hold with $\EkGrowth=4\theta$, $\EkGrowthMono = 3\theta$,$\ErrMono_k = \Err_k= 0$, and
    \[
        \EkLoss = 2\EkLossMono = {S''_{\max}}
        \left(
            S'_{\max}\delta
            + 3\norm{S_k(\this\optx ) - b_k}_{R^{N_1N_2}}
            + \frac{\delta^2S_{\text{max}}''}{2}
        \right)
        +
        \left(2 + \sqrt{2}\right)(S'_{\max})^2.
    \]
\end{theorem}

\begin{proof}
    Since $S_k$ is twice differentiable, the bounds \eqref{eq:dSkbounds} on $\norm{S'}$ and $\norm{S''}$ guarantee \cref{ass:bt:difbounds}.
    By the preceding discussion, \eqref{eq:bt:Ebounds:ak-cond} guarantees
    \[
        \norm{S_k'(\this\optx)(x - \this\optx)}_{Z_k}^2 \ge
        \max\{c_1, c_2\}
        \norm{x - \this\optx}_{X_k}^2
    \]
    Now \cref{cor:dataterm:final} with $\epsilon = 2\tilde \epsilon = 2\theta$ and $\beta = 1 - 2^{-1/2}$ shows for $D = \frac{1}{2}\EkLoss = \EkLossMono$ that
    \begin{gather*}
        \iprod{\grad E_k(\this\primalpredict)}{x - \this\optx}_{X_k}  \ge E_k(x) - E_k(\this\optx) + 2\theta \norm{x - \this\optx}_{X_k}^2 - D\norm{x-\this\primalpredict}_{X_k}^2
    \shortintertext{and}
        \iprod{\grad E (\this\primalpredict) - \grad E (\this{\hat x})}{x - \this{\hat x}}_{X_k}  \ge  3\theta \norm{x - \this\optx}^2_{X_k} - D\norm{x-\this\primalpredict}^2_{X_k}.
        \qedhere
    \end{gather*}
\end{proof}

Since we enforce $x \in [x_m, x_M]$, a.e., the previous theorem can also be used to prove the global \cref{ass:pd:main-global} by taking $\delta$ large enough.

\section{Numerical experiments}
\label{sec:numerical}

We numerically assess \cref{alg:pd:alg} in dynamic EIT imaging of a solid object moving in a fluid assumed to follow the incompressible constant-density transport equation.
Our evaluation extends to scenarios that challenge the constant speed and incompressibility assumptions. We evaluate several dual predictors, and compare the results against static reconstructions with the Relaxed Inexact Gauss--Newton method (RIPGN) \cite{jauhiainen2020ripgn}.
Our software implementation is available on Zenodo \cite{jauhiainen2025online-eit-codes}.

\subsection{Test scenarios}

The test scenarios take place in a disk-shaped domain denoted by $\Omega$. We use $N_1=16$ evenly placed boundary electrodes. We set the electrode potentials $U^{j,k}$ so that an electrode $j$ is set to a potential $U^{j,k}_j = 1\,\text{V}$ while all others are grounded, $U^{j,k}_i = 0\,\text{V}$ for $i\neq j$. This pattern repeats for all electrodes, leading to $N_2=16$ sets of electrode potentials. To mimic the typical EIT measurements, we exclude the currents at the excited electrode $j$ from both measurements and the forward operator, as these currents are often not measured by EIT devices. Thus, each time instance $k$ yields a total of $(N_1-1) N_2=240$ measurements.

The four experiments are as follows:
\begin{description}
    \item[Baseline]
    Features an inclusion moving at constant speed on a homogeneous background.
    Serves as a validation scenario to ensure the algorithm works as expected.
    \item[Circular Motion]
    Features an inclusion following a circular motion path, challenging the constant movement assumption.
    \item[Halting Motion]
    Features an inclusion that comes to a halt at frames 1000 and 2000, further challenging the constant movement assumption.
    \item[Disappearing Inclusions]
    Features two inclusions moving in circular path. First inclusions disappears at frame 500 and the second at frame 1000. Both reappear at frame 1500. This case challenges the incompressibility assumption.
\end{description}

The \emph{Baseline} experiment has 400 time frames, while the other three have 2000. The background conductivity is $x_\text{bg} = 1$ S and all inclusions are resistive with $x_\text{incl} = 10^{-4}$ S.

We simulate the measurement data by approximating \eqref{eq:eit:cem} with the Galerkin finite element method (FEM). We use piecewise linear basis. The simulation mesh has 5039 nodes and 9852 elements. Each simulated measurement $\EITmeas_i^{j,k}$ has added Gaussian noise with standard deviation $\text{std} = 10^{-4}\abs{\EITmeas_i^{j,k}}$, in the standard range of EIT. For the specifics of how to solve \eqref{eq:eit:cem} and its Fréchet derivative with FEM, see \cite{jauhiainen2020ripgn}.

\subsection{Numerical setup}

Recall the definition of  $E_k$ from \eqref{eq:eit:functionals}. The currents $S_k:X^k\times \R^{N_1} \to \R^{N_1N_2}$ in $E_k$ are obtained by approximating the potential functions $u^{j,k}$ in \eqref{eq:eit:cem} through FEM, using piecewise linear basis functions. The conductivity $x$ is also represented in the same basis. To prevent the `inverse crime' \cite{kaipio2000statistical}, we use a less dense mesh for the forward problem than we used for simulation, featuring 2917 nodes and 5430 mesh elements.

\subsubsection{Background processing}\label{sssec:bgprocessing}

Note that the computation of $S_k(x^k)$ and $\grad S_k(x^k)$ is highly resource-intensive. To optimise the computational speed of \cref{alg:pd:alg}, we implement the following background processing strategy: we approximate $\grad S_k(x^k)$ by $\grad S_k(\check x)$, initially $\check x =x^0$. We then approximate $S_k(x^k)$ by a first-order Taylor expansion around $\check x$. In the background, we compute $S_k(\tilde x)$ and $\grad S_k(\tilde x)$, at a linearisation point $\tilde x=x^j$. After completing these background computations, we update $\check x = \tilde x$ as well as $\tilde x=x^k$ for the current iterate, and start computing new values of the operators in the background.
In \cref{app:numerics:pde} we show that the background processing scheme satisfies \cref{ass:pd:main-local}.

\subsubsection{Predictors}\label{sssec:predictors}

We construct the primal component of the predictor $P_k$ by assuming constant velocity and incompressibility in the moving objects. Denoting by $h^k(\xi) \in \R^2$ the displacement at $\xi \in \Omega$ between frames $k$ and $k-1$, we define the primal prediction referred to as Flow, as $\primalpredict^{k+1}(\xi) = W_kx^k(\xi) \defeq x^k(\xi + h^k(\xi))$. Since $h^k$ represents the estimated displacement between the current and previous frames and is used to predict the next frame, this effectively assumes a constant velocity between frames.

We estimate displacement $h^k$ from the incompressible transport equation
\begin{equation}\label{eq:iflow:pde}
    \frac{\partial x^k}{\partial t} - \snabla x^k \cdot v = 0.
\end{equation}
Here, $v = (v_1, v_2) \in \R^2$ represents the velocity responsible for the time-dependent displacement $h$ at time $t$. Using a fixed time step of $\Delta t = 1$, we approximate $\frac{\partial x^k}{\partial t}$ as $(\thisx - x^{k-1})/\Delta t = \thisx - x^{k-1}$. This leads us to a problem akin to optical flow, where the solution for $v$ is subject to non-uniqueness due to the aperture problem \cite{beauchemin1995computation}. Similar to the well-established Horn-Schunck method \cite{horn1981determining}, we approximate the solution of \eqref{eq:iflow:pde} by
\begin{equation}\label{eq:flow:estim}
    v^k = \argmin_v~ \frac{1}{2} \adaptnorm{\frac{\partial x^k}{\partial t} - \snabla x^{k-1} \cdot v}^2 + \frac{\beta_1}{2} \left(\norm{\snabla v_1}^2 + \norm{\snabla v_2}^2\right) + \frac{\beta_2}{2}\norm{v}^2.
\end{equation}
Here, $\snabla$ denotes the spatial gradient operator, and $\beta_1$ and $\beta_2$ are positive regularisation parameters. The components $v_1^k$ and $v_2^k$ share the same piecewise linear basis as the conductivity $x^k$, and the displacement $h^k$ corresponds to $v^k$ due to the chosen time step.
To expedite the algorithm, we only update $h^k$ every fourth iteration.

For the dual variable, we consider two distinct predictors: \emph{Greedy} preserves $\iprod{\snabla \nexxt\primalpredict}{\nexxt\dualpredict} = \iprod{\snabla \thisx}{\thisy}$ element-wise, while \emph{Affine} sets $\nexxt\dualpredict \defeq \thisy + c \snabla \nexxt\primalpredict$.
Through the choice of $c > 0$, it seeks to promote sparsity in regions where we expect constant values, as discussed in \cref{ex:pd:predictors}.
To bound the prediction component of the error $e_N^\Sigma(u^{0:N-1}, \optu^{0:N})$ in \cref{cor:pd:main} for these predictors, see \cite[Lemma 3.11 and Lemma 3.13]{better-predict}.
The gradient estimate error we treat in \cref{app:numerics:pde}.
We also perform experiments with the \emph{Identity} predictor $W_k= \Id$ and $T_k=\Id$.

We evaluate the algorithm across four predictor configurations. \emph{No Prediction} is fully uninformed having identity predictors for both primal and dual variables.
\emph{Primal Only} uses incompressible flow prediction for the primal combined with uninformed identity prediction for the dual.
The remaining two configurations feature the incompressible flow prediction for the primal, accompanied by either \emph{Greedy} or \emph{Affine} dual prediction.

Note that only the Baseline experiment fully satisfies the constant velocity assumption. Additionally, the fourth test scenario violates the incompressibility assumption. However, we expect the algorithm to perform correctly, as the optimisation steps should compensate for small errors introduced in the prediction step. Furthermore, $v$ is only an approximate solution to \eqref{eq:iflow:pde}, meaning some degree of compression may still be present.

\subsubsection{Algorithm parameters}
In all experiments, we use the following parameters:%
\begin{description}
    \item[Main problem parameters]
    We set $\WOp=200 \Id$, $\alpha = 0.5$, $\conductivity_m = 10^{-5}$, and $\conductivity_M = 10^{5}$. We do not use the knowledge of the precise noise statistics or the conductivity range.
    \item[Step length parameters]
    We take constant $\tau_k \equiv \tau \defeq 0.85\norm{\WOp\grad I(\check x)(\WOp\grad I(\check x))^*}^{-1}$ and $\sigma_k \equiv \sigma \defeq 1$.
    \item[Incompressible flow parameters]
    We set $\beta_1 = 10^{-3}$ and $\beta_2=10^{-5}$.
    \item[Affine predictor parameters]
    To promote gradient sparsity in calm areas, we take $c \defeq 10\max_{\xi}\{ 0, 1- 10^{-12}\abs{h^k}^{-1} \}^2$ for $h^k$ the estimated displacement.
\end{description}
The step length parameter choices are discussed in detail in \cref{sup:step:parameters}.

\subsection{Results}

The \emph{Baseline} experiment features an inclusion moving at constant speed. \Cref{fig:case1:of:re} shows the relative objective values and iterate errors
\[
     J_{\operatorname{rel}}^k \defeq \frac{J_k(x^k)}{J_k(x^0)}
     \quad\text{and}\quad
     e_{\operatorname{rel}}^k \defeq \frac{\norm{x^k-x_{\text{true}}^k}}{\norm{x_{\text{true}}^k}}
\]
where
$x_{\text{true}}^k$ is the ground-truth for iteration $k$
with the tested predictor configurations. \emph{No Prediction} performs the worst and \emph{Affine} predictor the best. \emph{Primal Only} and \emph{Greedy} perform similarly, outperforming \emph{No Prediction} but falling behind \emph{Affine}. This is also confirmed by \cref{tab:casesall}.

\begin{table}\footnotesize
    \setlength\tabcolsep{3.5pt}
    \centering
    \caption{Average relative errors and confidence intervals (CI). The latter are calculated using Student's $t$-distribution with the relative errors for each iteration from 1 to 500 as samples.}
    \label{tab:casesall}
    \begin{tabular}{lccc}
        \hline
        \multicolumn{4}{c}{\textbf{Baseline}} \\
        &
        \multicolumn{2}{c}{Average RE} & \\
        Predictor & iter 1 & iter 50 & 95\% CI \\ \hline
        Affine&9.56&9.0725&9.0311 - 9.1139\\
        Greedy&10.7881&10.4691&10.3895 - 10.5487\\
        No predict&14.4899&14.4774&14.4204 - 14.5344\\
        Primal only&10.844&10.53&10.4523 - 10.6077\\
    \end{tabular}
    \begin{tabular}{lccc}
        \hline
        \multicolumn{4}{c}{\textbf{Circular motion}} \\
        & \multicolumn{2}{c}{Average RE} & \\
        Predictor & iter 1 & iter 50 & 95\% CI \\ \hline
         Affine&9.0218&8.8642&8.8424 - 8.886\\
         Greedy&10.5676&10.4494&10.4213 - 10.4775\\
         No predict&14.4643&14.4275&14.3768 - 14.4782\\
         Primal only&10.632&10.5153&10.4868 - 10.5438\\
    \end{tabular}
    \begin{tabular}{lccc}
        \hline
        \multicolumn{4}{c}{\textbf{Halting motion}} \\
        & \multicolumn{2}{c}{Average RE} &\\
        Predictor & iter 1 & iter 50 & 95\% CI \\ \hline
         Affine&9.1663&8.9985&8.9651 - 9.0319\\
         Greedy&10.3048&10.1686&10.1416 - 10.1956\\
         No predict&14.7439&14.7116&14.6822 - 14.741\\
         Primal only&10.2875&10.1508&10.124 - 10.1776\\
    \end{tabular}
    \begin{tabular}{lccc}
        \hline
        \multicolumn{4}{c}{\textbf{Halting motion}} \\
        & \multicolumn{2}{c}{Average RE} & \\
        Predictor & iter 1 & iter 50 & 95\% CI \\ \hline
         Affine&11.1383&10.8471&10.6356 - 11.0586\\
         Greedy&11.5483&11.2712&11.0304 - 11.512\\
         No predict&15.1044&14.8975&14.5571 - 15.2379\\
         Primal only&11.7061&11.4329&11.1949 - 11.6709\\
    \end{tabular}
\end{table}

\begin{figure}[t]
    \centering
    \pgfplotslegendfromname{leg:case1}\\
    \begin{subfigure}{0.49\textwidth}%
        \centering%
        \pgfplotstableread{img/case1/NP_s1.0_t0.85.txt}{\resNP}
        \pgfplotstableread{img/case1/PP_s1.0_t0.85.txt}{\resPP}%
        \pgfplotstableread{img/case1/PDP1_s1.0_t0.85.txt}{\resPDPone}%
        \pgfplotstableread{img/case1/PDP2_s1.0_t0.85.txt}{\resPDPtwo}%
        \pgfplotsset{ every axis/.append style = { legend columns = 4, legend to name = leg:case1} }
        \input{img/relvalue.tikz}%
        \caption{Relative objective value}%
        \label{fig:case1:of}%
    \end{subfigure}%
    \hfill%
    \begin{subfigure}{0.49\textwidth}%
        \centering%
        \pgfplotstableread{img/case1/NP_s1.0_t0.85_re.txt}{\resNP}
        \pgfplotstableread{img/case1/PP_s1.0_t0.85_re.txt}{\resPP}%
        \pgfplotstableread{img/case1/PDP1_s1.0_t0.85_re.txt}{\resPDPone}%
        \pgfplotstableread{img/case1/PDP2_s1.0_t0.85_re.txt}{\resPDPtwo}%
        \input{img/relerror.tikz}%
        \caption{Relative error to ground-truth}%
        \label{fig:case1:re}%
    \end{subfigure}%
    \caption{Iteration-wise relative objective values and iterate errors in \emph{Baseline} experiment.}
    \label{fig:case1:of:re}
 \end{figure}
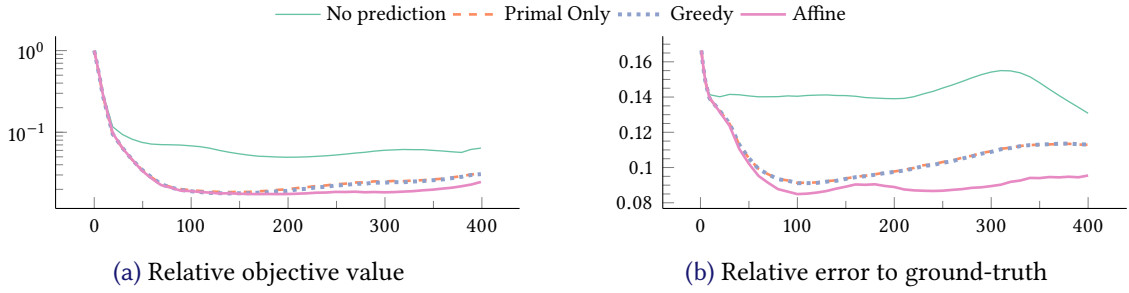

 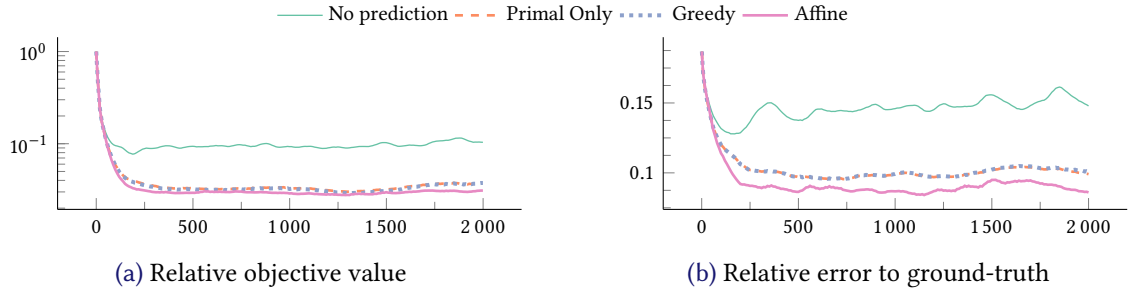
\begin{figure}[t]
    \centering
    \pgfplotslegendfromname{leg:case2}\\
    \begin{subfigure}{0.49\textwidth}%
        \centering%
        \pgfplotstableread{img/case2/NP_s1.0_t0.85.txt}{\resNP}
        \pgfplotstableread{img/case2/PP_s1.0_t0.85.txt}{\resPP}%
        \pgfplotstableread{img/case2/PDP1_s1.0_t0.85.txt}{\resPDPone}%
        \pgfplotstableread{img/case2/PDP2_s1.0_t0.85.txt}{\resPDPtwo}%
        \pgfplotsset{ every axis/.append style = { legend columns = 4, legend to name = leg:case2} }
        \input{img/relvalue.tikz}%
        \caption{Relative objective value}%
        \label{fig:case2:of}%
    \end{subfigure}%
    \hfill%
    \begin{subfigure}{0.49\textwidth}%
        \centering%
        \pgfplotstableread{img/case2/NP_s1.0_t0.85_re.txt}{\resNP}
        \pgfplotstableread{img/case2/PP_s1.0_t0.85_re.txt}{\resPP}%
        \pgfplotstableread{img/case2/PDP1_s1.0_t0.85_re.txt}{\resPDPone}%
        \pgfplotstableread{img/case2/PDP2_s1.0_t0.85_re.txt}{\resPDPtwo}%
        \input{img/relerror.tikz}%
        \caption{Relative error to ground-truth}%
        \label{fig:case2:re}%
    \end{subfigure}%
    \caption{Iteration-wise objective value and relative error in \emph{Circular Motion} experiment.}
    \label{fig:case2:of:re}
\end{figure}

Reconstructions in \cref{fig:reco:case1} confirm what we observed in \cref{fig:case1:of:re} and \cref{tab:casesall}. The incompressible Affine predictor produces slightly sharper reconstructions than other Flow predictors, with No prediction yielding the blurriest results.

\begin{figure}[t]
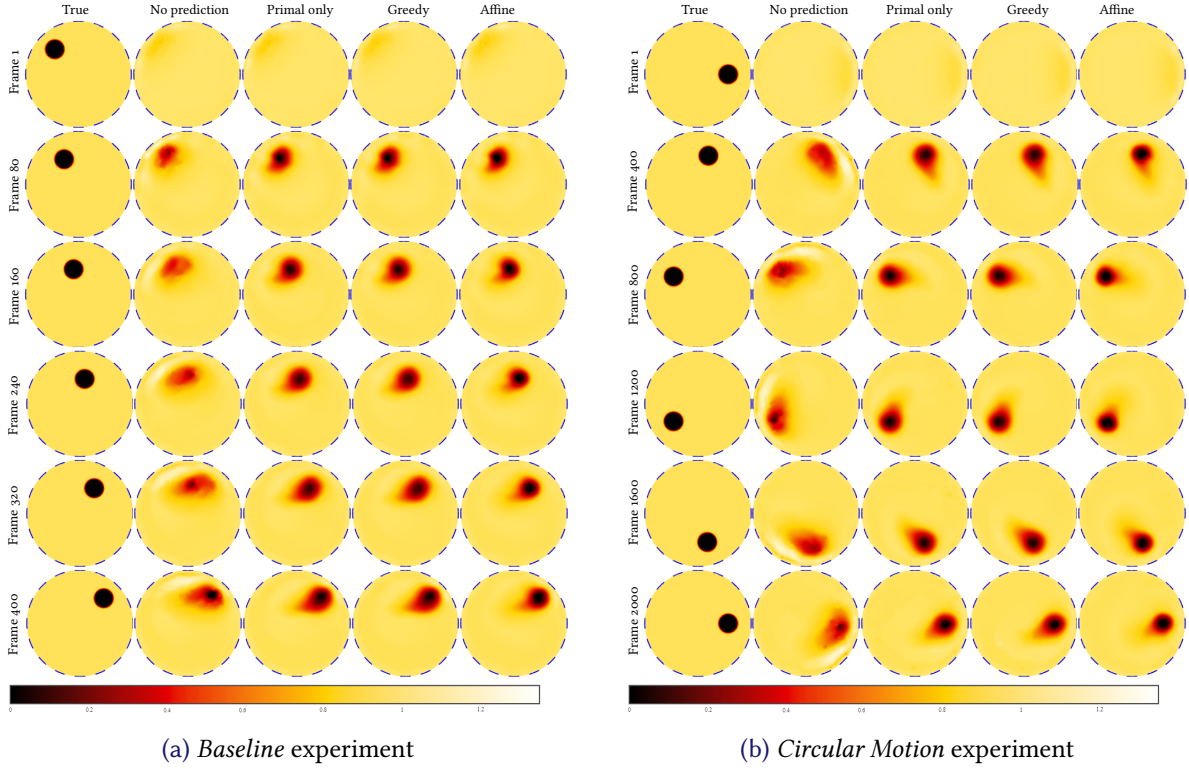

    \centering
    \subcaptionbox{\emph{Baseline} experiment\protect\label{fig:reco:case1}}
        {\centering\resizebox{0.48\textwidth}{!}{\begin{minipage}{\linewidth}{\imagearray{case1}}\end{minipage}}}%
        \hfill
    \subcaptionbox{\emph{Circular Motion} experiment\protect\label{fig:reco:case2}}
        {\centering\resizebox{0.48\textwidth}{!}{\begin{minipage}{\linewidth}{\imagearrayl{case2}}\end{minipage}}}%
    \caption{A selection of reconstructed frames in \emph{Baseline} and \emph{Circular Motion} experiments.}
    \label{fig:case12:reco}
\end{figure}

In the \emph{Circular Motion} experiment, an inclusion moves in a circular trajectory. By \cref{fig:case2:of:re,fig:reco:case2} and \cref{tab:casesall}, \emph{Affine} again provides the sharpest reconstructions and lowest objective values, while \emph{No Prediction} performs the worst.

 \begin{figure}[t]
    \centering
    \includegraphics[width=0.20\linewidth]{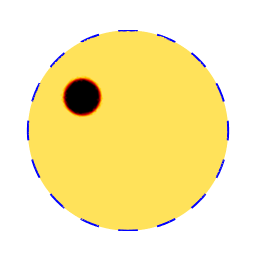}
    \includegraphics[width=0.186\linewidth, trim={0 -0.1 0 0},clip]{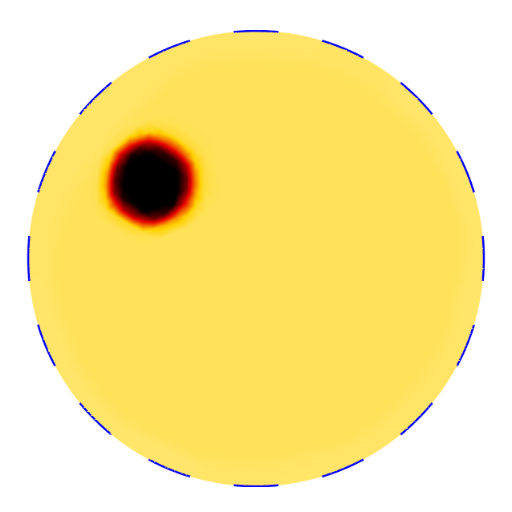}
    \includegraphics[width=0.20\linewidth]{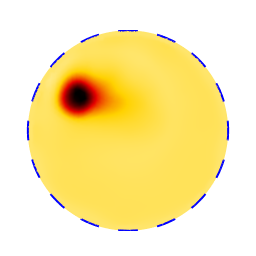}
    \caption{\emph{Circular Motion} experiment. Comparison of the reconstruction quality. Left: true target. Middle: static reconstruction with RIPGN. Right: online reconstruction with \cref{alg:pd:alg} and Affine prediction. }
    \label{fig:case2:reco:comp}
\end{figure}
\Cref{fig:case2:reco:comp} compares the true target, a static solution using RIPGN, and online reconstruction with \emph{Affine} at frame 500. The static reconstruction is sharper but significantly slower, taking $71.9$ seconds, whereas online reconstruction with \cref{alg:pd:alg} captures essential features in just 12.1 milliseconds.

 \begin{figure}[t]
    \centering
    \pgfplotslegendfromname{leg:case3}\\
    \begin{subfigure}{0.49\textwidth}%
        \centering%
        \pgfplotstableread{img/case3/NP_s1.0_t0.85.txt}{\resNP}
        \pgfplotstableread{img/case3/PP_s1.0_t0.85.txt}{\resPP}%
        \pgfplotstableread{img/case3/PDP1_s1.0_t0.85.txt}{\resPDPone}%
        \pgfplotstableread{img/case3/PDP2_s1.0_t0.85.txt}{\resPDPtwo}%
        \pgfplotsset{ every axis/.append style = { legend columns = 4, legend to name = leg:case3} }
        \input{img/relvalue.tikz}%
        \caption{Relative objective value}%
        \label{fig:case3:of}%
    \end{subfigure}%
    \hfill%
    \begin{subfigure}{0.49\textwidth}%
        \centering%
        \pgfplotstableread{img/case3/NP_s1.0_t0.85_re.txt}{\resNP}
        \pgfplotstableread{img/case3/PP_s1.0_t0.85_re.txt}{\resPP}%
        \pgfplotstableread{img/case3/PDP1_s1.0_t0.85_re.txt}{\resPDPone}%
        \pgfplotstableread{img/case3/PDP2_s1.0_t0.85_re.txt}{\resPDPtwo}%
        \input{img/relerror.tikz}%
        \caption{Relative error to ground-truth}%
        \label{fig:case3:re}%
    \end{subfigure}%
    \caption{Iteration-wise objective value and relative error in \emph{Halting Motion} experiment.}
    \label{fig:case3:of:re}
 \end{figure}
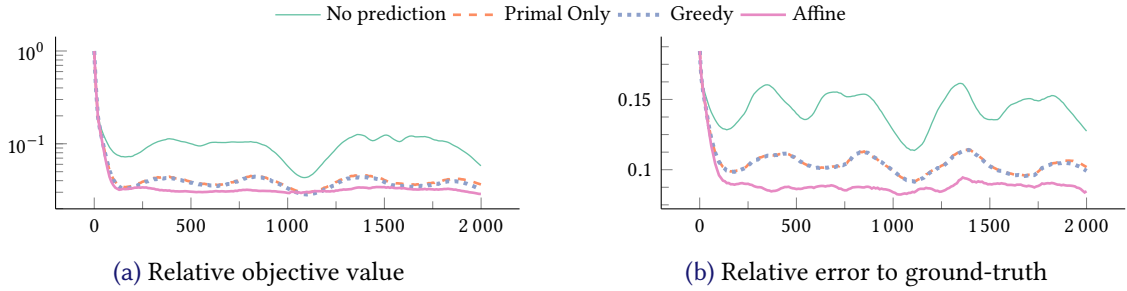

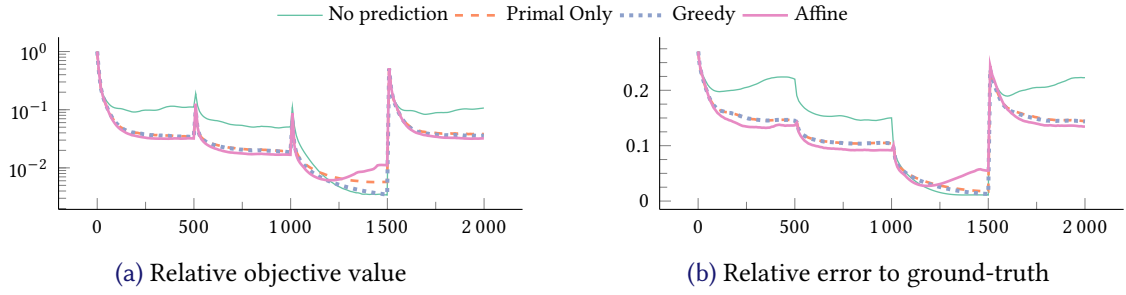
\begin{figure}[t]
    \centering
    \pgfplotslegendfromname{leg:case4}\\
    \begin{subfigure}{0.49\textwidth}%
        \centering%
        \pgfplotstableread{img/case4/NP_s1.0_t0.85.txt}{\resNP}
        \pgfplotstableread{img/case4/PP_s1.0_t0.85.txt}{\resPP}%
        \pgfplotstableread{img/case4/PDP1_s1.0_t0.85.txt}{\resPDPone}%
        \pgfplotstableread{img/case4/PDP2_s1.0_t0.85.txt}{\resPDPtwo}%
        \pgfplotsset{ every axis/.append style = { legend columns = 4, legend to name = leg:case4} }
        \input{img/relvalue.tikz}%
        \caption{Relative objective value}%
        \label{fig:case4:of}%
    \end{subfigure}%
    \hfill%
    \begin{subfigure}{0.49\textwidth}%
        \centering%
        \pgfplotstableread{img/case4/NP_s1.0_t0.85_re.txt}{\resNP}
        \pgfplotstableread{img/case4/PP_s1.0_t0.85_re.txt}{\resPP}%
        \pgfplotstableread{img/case4/PDP1_s1.0_t0.85_re.txt}{\resPDPone}%
        \pgfplotstableread{img/case4/PDP2_s1.0_t0.85_re.txt}{\resPDPtwo}%
        \input{img/relerror.tikz}%
        \caption{Relative error to ground-truth}%
        \label{fig:case4:re}%
    \end{subfigure}%
    \caption{Iteration-wise objective value and relative error for \emph{Disappearing Inclusions}.}
    \label{fig:case4:of:re}
\end{figure}

 \begin{figure}[t]
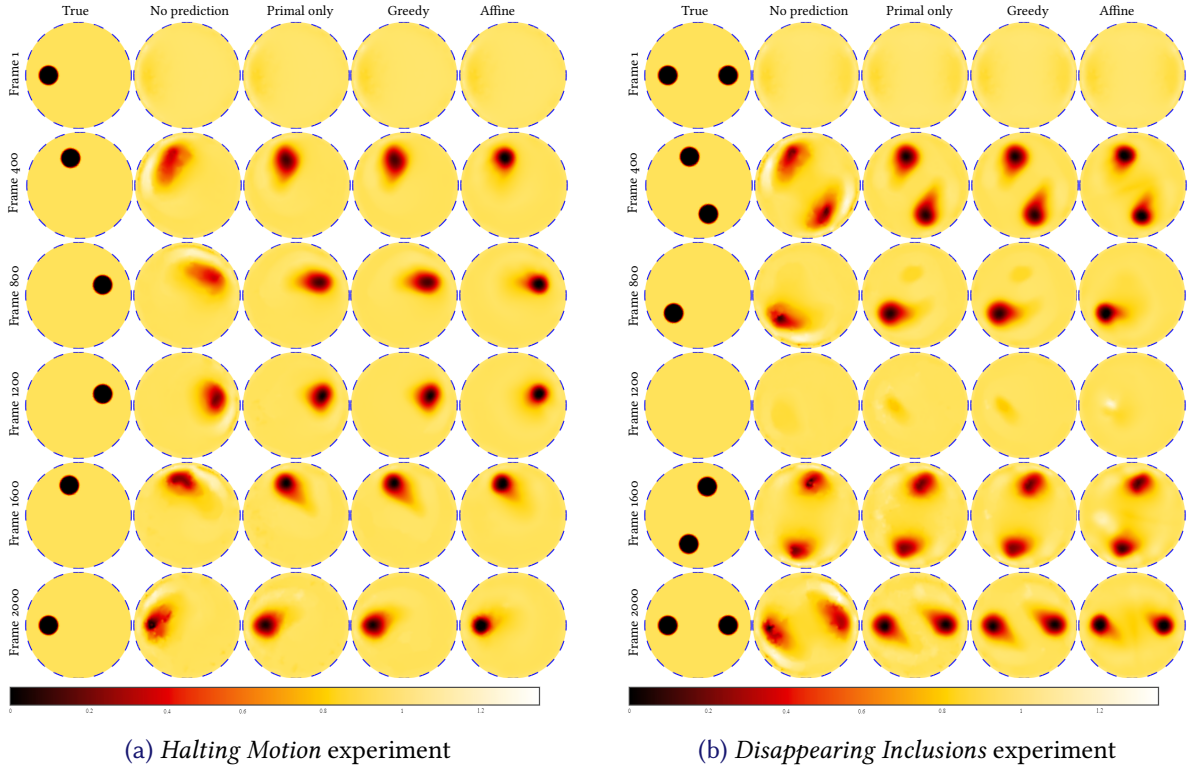

    \centering
    \subcaptionbox{\emph{Halting Motion} experiment\protect\label{fig:reco:case3}}
        {\centering\resizebox{0.48\textwidth}{!}{\begin{minipage}{\linewidth}{\imagearrayl{case3}}\end{minipage}}}%
        \hfill
    \subcaptionbox{\emph{Disappearing Inclusions} experiment\protect\label{fig:reco:case4}}
        {\centering\resizebox{0.48\textwidth}{!}{\begin{minipage}{\linewidth}{\imagearrayl{case4}}\end{minipage}}}%
    \caption{A selection of reconstructed frames in \emph{Halting Motion} and \emph{Disappearing Inclusions} experiments.}
    \label{fig:case34:reco}
\end{figure}

The third experiment involves a halting inclusion. By \cref{fig:case3:of:re,fig:reco:case3} and \cref{tab:casesall}, the reconstruction quality, the relative objective values and iterate errors align with previous experiments. Notably, as objects slow down, especially with \emph{No Prediction}, reconstructions improve. This aligns with expectations, as, aside from the noise, the data is static, where identity prediction is optimal.

The final experiment features two inclusions vanishing at different frames (500 and 1000) and reappearing at frame 1500. \Cref{tab:casesall} and \cref{fig:case4:of:re} show similar results to the previous cases, although by average relative error, the differences between \emph{Affine} and \emph{No Prediction} and \emph{Primal Only} are smaller. However, \cref{fig:case4:of:re} reveals abrupt spikes in the objective value and in the relative error when objects disappear, followed by subsequent decreases as reconstructions exhibit fewer edges, resulting in lower total variation penalties. \emph{No Prediction} dominates when both inclusions disappear since, aside from the noise, the data is completely static. By \cref{fig:reco:case4}, the reconstructions accurately capture the process of inclusions disappearing and reappearing.

\section{Conclusions}

Online optimisation offers a real-time option for solving sequential optimisation problems. While its application in dynamic inverse problems remains relatively rare, we introduced a predictive online primal-dual proximal splitting method tailored for objective functions with non-smooth and non-convex components.

We established a regret bound for this method and we comprehensively evaluated the method in dynamic EIT. Through numerical evaluations using incompressible constant flow-based predictors, we have demonstrated a substantial enhancement in reconstruction quality when compared to uninformed predictions, even in cases where the assumptions of constancy and incompressibility were violated. Remarkably, this improvement was achieved while maintaining minimal computational times, averaging just around 12 milliseconds.

It is worth noting that, in our experiments, the online reconstruction quality, albeit slightly inferior, remained competitive with significantly slower and computationally more costly static reconstructions.
We anticipate further enhancements in the computational speed of our algorithm through optimised numerical implementations, promising even more efficient and effective solutions in the future.

\input{online-eit.bbl}
\appendix

\section{The data fitting term and its properties}
\label{app:dataterm}

In this appendix, we prove \cref{cor:dataterm:final} on smoothness inequalities for
\[
    E_k(x) \defeq \frac{1}{2}\norm{S_k(x)-b_k}_{Z_k}^2,
\]
required in the proof of \cref{cor:bt:Ebounds}. We assume:

\begin{assumption}
    \label{ass:bt:difbounds}
    Let $\SolOp_k:X_k\mapsto Z_k$ be twice Fréchet differentiable between the Hilbert spaces $X_k$ and $Z_k$, and suppose $\SolOp_k'(x) \in \linear( X_k;  Z_k)$ and $\SolOp_k''(x) \in \linear( X_k \times X_k;  Z_k)$ satisfy
    \begin{align*}
        \norm{\SolOp_k'(x)(y-x)}_{Z_k} & \le S'_{\max}\norm{y-x}_{X_k}
        \quad\text{and}
        \\
        \norm{\SolOp_k''(x)(y-x,z-x)}_{Z_k} & \le S''_{\max}\norm{y-x}_{X_k}\norm{z-x}_{X_k}
        \quad\text{for all}\quad x,y,z \in B(\optx,{\delta}),
    \end{align*}
\end{assumption}

For brevity, in the proofs, we drop the time index $k$ as it bears no impact on our analysis. We also abbreviate $\SolOp'_{z}(h) \defeq \SolOp'(z)(h)$ and $R(x)=S_k(x)-b_k$.
Then $E(x)=\frac{1}{2}\norm{R(x)}^2$, keeping in mind that $S_k'=R'$ and $S_k''=R''$.
We start with an auxiliary lemma.

\begin{lemma}
    \label{lemma:dataterm:basic}
    Suppose that \cref{ass:bt:difbounds} holds for some $\optx\in {X_k}$ and $\delta > 0$, and that $x,z \in B(\optx,{\delta})$. Then, $c=z + t(x - z)$ and some $t \in [0,1]$,
    \[
        \begin{aligned}
            \iprod{\grad E_k (z)}{z - x}_{X_k}
            &
            = E_k(z) - E_k(x) + \frac{1}{2}\norm{S_k'(z)(x - z)}_{Z_k}^2 - \frac{1}{8}\norm{S_k''(c)(x -z, x -z)}_{Z_k}^2
            \\
            \MoveEqLeft[-1]
            +
            \frac{1}{2}\iprod{S_k''(c)(x -z, x -z)}{S_k(x) - b_k}_{Z_k}.
        \end{aligned}
    \]
\end{lemma}

\begin{proof}
    Since $S_k$ and therefore $R$ is twice differentiable, Taylor expansion gives for $c \defeq x + t(z - x)$ and some $t \in [0,1]$ that
    \begin{equation}
        \label{eq:dataterm:basic:0}
        \ResOp(x) = \ResOp(z) + \ResOp'_z(x - z) + \frac{1}{2}\ResOp''_c(x -z, x -z).
    \end{equation}
    Expanding $E(x) - E(z)$ and using \eqref{eq:dataterm:basic:0} and $\iprod{\grad E(z)}{x-z} = \iprod{\ResOp'_z(x-z)}{\ResOp(z)}$ then yields
    \begin{multline}
        \label{eq:dataterm:basic:1}
        E(x) - E(z)
        =
        \frac{1}{2}\left(\norm{\ResOp(x)}^2 - \norm{\ResOp(z)}^2\right) = \frac{1}{2}\iprod{\ResOp(x)-\ResOp(z)}{\ResOp(x)+\ResOp(z)}
        \\
        =
        \frac{1}{2}\iprod{\ResOp(z) + \ResOp'_z(x - z) + \frac{1}{2}\ResOp''_c(x -z, x -z)-\ResOp(z)}{\ResOp(x)+\ResOp(z)}
        =
        \frac{1}{2}\iprod{\grad E (z)}{x - z}
        + A,
    \end{multline}
    where an application of \eqref{eq:dataterm:basic:0} establishes
    \begin{equation}
        \label{eq:dataterm:basic:2}
        \begin{aligned}[t]
            A
            &
            \defeq
            \frac{1}{2}\Bigl(\Bigiprod{\ResOp'_z(x - z) + \frac{1}{2}\ResOp''_c(x -z, x -z)}{\ResOp(x)} + \frac{1}{2}\Bigiprod{\ResOp''_c(x -z, x -z)}{\ResOp(z)}\Bigr)
            \\
            &
            = \frac{1}{2}\Bigiprod{\ResOp'_z(x - z) + \frac{1}{2}\ResOp''_c(x -z, x -z)}{\ResOp(z) + \ResOp'_z(x - z) + \frac{1}{2}\ResOp''_c(x -z, x -z)}
            \\
            \MoveEqLeft[-1]
            + \frac{1}{4}\iprod{\ResOp''_c(x -z, x -z)}{\ResOp(z)}
            = \frac{1}{2}\iprod{\grad E (z)}{x - z} + \frac{1}{2}\norm{\ResOp'_z(x - z)}^2 + B.
        \end{aligned}
    \end{equation}
    Here again, an application of \eqref{eq:dataterm:basic:0}, gives
    \begin{equation}
        \label{eq:dataterm:basic:3}
        \begin{aligned}[t]
            B
            &
            \defeq \frac{1}{8}\norm{\ResOp''_c(x -z, x -z)}^2
            + \frac{1}{2}\iprod{\ResOp''_c(x -z, x -z)}{\ResOp(z) + \ResOp'_z(x - z)}
            \\
            &
            =
            - \frac{1}{8}\norm{\ResOp''_c(x -z, x -z)}^2 + \frac{1}{2}\iprod{\ResOp''_c(x -z, x -z)}{\ResOp(x)}. \\
        \end{aligned}
    \end{equation}
    By combining \cref{eq:dataterm:basic:1,eq:dataterm:basic:2,eq:dataterm:basic:3}, we obtain the claim.
\end{proof}

With the help of \cref{lemma:dataterm:basic}, we next derive a lower bound for $\iprod{\grad E_k(z)}{x - \optx}_{X_k}$.

\begin{lemma}
    \label{lemma:dataterm:norms-only}
    Suppose that \cref{ass:bt:difbounds} holds for some $\optx\in {X_k}$ and $\delta > 0$.
    Then
    $$
        \begin{aligned}
            \iprod{\grad E_k (z)}{x - \optx}_{X_k}  &\ge E_k(x) - E_k(\optx)
            + \frac{1 - \beta}{2}\norm{S_k'(z)(x-\optx)}^2_{Z_k} - \frac{1}{2\beta}\norm{S_k'(z)(x-z)}^2_{Z_k}
            \\[-1ex]
            \MoveEqLeft[-1]
            -
            \frac{A}{2}\norm{x-\optx}^2_{X_k}
            -
            \frac{B}{2}\norm{x-z}^2_{X_k}
        \end{aligned}
    $$
    for all $x,z \in B(\optx,{\delta})$, where for any $\beta > 0$,
    \begin{align}
        \nonumber
        A
        &
        \defeq S''_{\max}\left(
            2\norm{S_k(\optx)-b_k}_{Z_k} + \delta^2S''_{\max}/2
        \right)
        \quad\text{and}\quad
        \\
        \label{eq:dataterm:B}
        B
        &
        \defeq S''_{\max}\left(
            S'_{\max}\delta + 3\norm{S_k(\optx)-b_k}_{Z_k} + \delta^2S''_{\max}/2
        \right).
    \end{align}
\end{lemma}

\begin{proof}
    Since $x,z, \optx \in B(\optx,{\delta})$, applying \cref{lemma:dataterm:basic} twice establishes for some  $t_1, t_2\in [0,1]$ and $b \defeq z + t_1(x - z)$ and $c \defeq z + t_2(\optx - z)$ the identity
    \begin{multline}
        \label{eq:dataterm:norms-only:0}
        \iprod{\grad E (z)}{x - \optx}  = \iprod{\grad E (z)}{x - z} + \iprod{\grad E (z)}{z - \optx}
        \\
        = E(x) - E(\optx)
        + \frac{1}{2}\norm{\ResOp'_z(\optx - z)}^2 - \frac{1}{8}\norm{\ResOp''_c(\optx -z, \optx -z)}^2 + \frac{1}{2}\iprod{\ResOp''_c(\optx -z, \optx -z)}{\ResOp(\optx)}
        \\
        - \frac{1}{2}\norm{\ResOp'_z(x - z)}^2 + \frac{1}{8}\norm{\ResOp''_b(x -z, x -z)}^2 - \frac{1}{2}\iprod{\ResOp''_b(x -z, x -z)}{\ResOp(x)}.
    \end{multline}
    Next, we estimate the $\optx - z$ terms with similar $\optx - x$ and $x -z$ terms. First, let us inspect the term $\norm{\ResOp'_z(\optx - z)}^2$. By Young's inequality and the linearity of $\ResOp'$,
    \begin{equation}
        \label{eq:dataterm:norms-only:1}
            \frac{1}{2}\norm{\ResOp'_z(\optx - z)}^2
            =
            \frac{1}{2}\norm{\ResOp'_z(\optx - x)  - \ResOp'_z(z - x)}^2
            \ge \frac{1 - \beta}{2}\norm{\ResOp'_z(\optx - x)}^2 + \frac{1 - \beta^{-1}}{2}\norm{\ResOp'_z(z - x)}^2.
    \end{equation}
    Next, we inspect the term $\norm{\ResOp''_c(\optx -z, \optx -z)}^2$. $\ResOp$ is bilinear and bounded by $R''_{\text{max}} =S''_{\max}$ by \cref{ass:bt:difbounds}, thus Young's inequality, and $z\in B(\optx,{\delta})$ gives
    \begin{equation}
        \label{eq:dataterm:norms-only:2}
        \begin{aligned}[t]
            \norm{\ResOp''_c(\optx -z, \optx -z)}^2
            &
            \le
            2\norm{\ResOp''_c(\optx - x, \optx -z)}^2
            +
            2\norm{\ResOp''_c(x -z, \optx -z)}^2
            \\
            &
            \le
            2(\ResOp''_{\max})^2\norm{\optx - x}^2\delta^2
            +
            2(\ResOp''_{\max})^2\norm{x -z}^2\delta^2.
        \end{aligned}
    \end{equation}
    Finally, let us inspect the term $\iprod{\ResOp''_c(\optx -z, \optx -z)}{\ResOp(\optx)} - \iprod{\ResOp''_b(x -z, x -z)}{\ResOp(x)}$.
    Notice that by mean value theorem
    $
        \ResOp(x) = \ResOp(\optx) + \ResOp'_a(x-\optx)
    $
    for $a = \optx + t_3(x-\optx)$ with some $t_3 \in [0,1]$.
    Using this, bilinearity and symmetricity of $\ResOp''$, boundedness of $\ResOp'$ and $\ResOp''$, Cauchy-Schwartz and Young's inequalities, and $x\in B(\optx,\delta)$ we obtain
    \begin{equation}
        \label{eq:dataterm:norms-only:3}
        \begin{aligned}[t]
            &\iprod{\ResOp''_c(\optx -z, \optx -z)}{\ResOp(\optx)} - \iprod{\ResOp''_b(x -z, x -z)}{\ResOp(x)}
            \\
            &
            =  \iprod{\ResOp''_c(\optx -x, \optx -x)}{\ResOp(\optx)} - \iprod{\ResOp''_b(x -z, x -z)}{\ResOp'_a(x - \optx)}
            \\
            \MoveEqLeft[-1]
            + \iprod{\ResOp''_c(x - z, x - z) -\ResOp''_b(x -z, x -z)}{\ResOp(\optx)}
            + 2\iprod{\ResOp''_c(\optx -x, x - z)}{\ResOp(\optx)}
            \\
            &
            \ge
            -
            S''_{\max}\norm{\ResOp(\optx)}\norm{\optx -x}^2
            -
            S''_{\max}S'_{\max}\delta\norm{x -z}^2
            - 2S''_{\max}\norm{\ResOp(\optx)}\norm{x-z}^2
            \\
            \MoveEqLeft[-1]
            - 2S''_{\max}\norm{\ResOp(\optx)}\norm{\optx -x}\norm{ x - z}
            \\
            &
            \ge
            -
            S''_{\max}
            \left(
                2\norm{\ResOp(\optx)}\norm{\optx -x}^2
                +
                \left(
                3\norm{\ResOp(\optx)}
                +
                S'_{\max}{\delta}
                \right)\norm{x -z}^2
            \right).
        \end{aligned}
    \end{equation}
    Finally, applying the above estimates \cref{eq:dataterm:norms-only:1,eq:dataterm:norms-only:2,eq:dataterm:norms-only:3} to \cref{eq:dataterm:norms-only:0} shows that
    \begin{multline*}
        \iprod{\grad E (z)}{x - \optx} \ge E(x) - E(\optx)
        + \frac{1 - \beta}{2}\norm{\ResOp'_z(\optx - x)}^2 - \frac{1}{2\beta}\norm{\ResOp'_z(z - x)}^2
        -
        S''_{\max}\norm{\ResOp(\optx)}\norm{\optx -x}^2
        \\
        - \frac{S''_{\max}(S'_{\max}{\delta} + 3\norm{\ResOp(\optx)})}{2}\norm{x-z}^2
        -
        \frac{\delta^2(\ResOp''_{\max})^2}{4}\left(
            \norm{\optx - x}^2 + \norm{x -z}^2
        \right).
        \qedhere
    \end{multline*}
\end{proof}

The next corollary transfers $S_k'(z)$ to $S_k'(\optx)$ in one of the terms.

\begin{corollary}
    \label{cor:dataterm:initial}
    Suppose \cref{ass:bt:difbounds} holds for some $\optx\in {X_k}$ and $\delta > 0$.
    Then
    $$
        \begin{aligned}
            \iprod{\grad E_k (z)}{x - \optx}_{X_k}  &\ge E_k(x) - E_k(\optx)
            + \frac{(1 - \beta)^2}{2}\norm{S_k'(\optx)(x - \optx)}^2_{Z_k}
             - \frac{1}{2\beta}\norm{S_k'(z)(x-z)}^2_{Z_k}
            \\[-1ex]
            \MoveEqLeft[-1]
            -
            \frac{C}{2}\norm{x - \optx}^2_{X_k}
            -
            \frac{B}{2}\norm{x-z}^2_{X_k},
        \end{aligned}
    $$
    for all $x,z \in B(\optx,{\delta})$, where, for any $1 >\beta > 0$, $B$ is given in \eqref{eq:dataterm:B}, and
    \begin{equation}
        \label{eq:dataterm:C}
        C
        \defeq
        A+ S''_{\max}\inv\beta(\beta-1)^2\delta^2
        =
        S''_{\max}
        \left(
        2\norm{S_k(\optx) - b_k}_{Z_k}
        +
        \delta^2S''_{\max}/2
        +
        \inv\beta(\beta-1)^2\delta^2
        \right).
    \end{equation}
\end{corollary}

\begin{proof}
    With fixed $h$, $\ResOp'_{z}(h)$ is continuously differentiable by \cref{ass:bt:difbounds}.
    Thus, by the main value theorem, for $a=\optx+ t_1(x - \optx)$ with some $t_1 \in [0,1]$, we have
    \begin{align}
        \label{eq:dataterm:initial:0}
        \ResOp'_z(\optx - x)
        &
        = \ResOp'(z)(\optx - x) = \ResOp'(\optx)(\optx - x) + \ResOp''(a)(\optx - x,z - \optx)
        \\
        \nonumber
        &
        = \ResOp'_\optx(\optx - x) + \ResOp''_a(\optx - x,z - \optx).
    \end{align}
    Since $\beta > 0$, \eqref{eq:dataterm:initial:0}, $z\in B(\optx,\delta)$, Young's inequality and the boundedness of $\ResOp''$ (due to boundedness of $\SolOp''$ in \cref{ass:bt:difbounds}) show that
    \begin{equation*}
        \begin{aligned}[t]
            \norm{\ResOp_z(\optx - x)}^2
            &
            = \norm{\ResOp'_{\optx}(\optx - x)}^2 + \norm{\ResOp''_a(\optx - x,z - \optx) }^2 + 2\iprod{\ResOp'_{\optx}(z - \optx)}{\ResOp''_a(\optx - x,z - \optx) }
            \\
            &
            \ge (1 - \beta)\norm{\ResOp'_{\optx}(\optx - x)}^2 - (\beta^{-1} - 1)\norm{\ResOp''_a(\optx - x,z - \optx) }^2
            \\
            &
            \ge (1 - \beta)\norm{\ResOp'_{\optx}(\optx - x)}^2 - (\beta^{-1} - 1)S''_{\max}\delta^2\norm{\optx - x}^2.
        \end{aligned}
    \end{equation*}
    Since $(1-\beta)(\beta^{-1} - 1) = (\beta-1)^2/\beta$, using this and \cref{lemma:dataterm:norms-only} yields the claim.
\end{proof}

Finally, we state conditions that guarantee the growth estimates of \cref{ass:pd:main-global,ass:pd:main-local} for $E$.
Essentially, $S'(\optx)^*S'(\optx)$ has to be sufficiently elliptic.

\begin{corollary}
    \label{cor:dataterm:final}
    Suppose that \cref{ass:bt:difbounds} holds for some $\optx\in {X_k}$ and $\delta > 0$.
    Let $x,z \in B(\optx,{\delta})$, and suppose some $1 > \beta > 0$ and some $\epsilon \in \R$ that
    \begin{gather}
        \label{eq:dataterm:final:cond:i}
        \norm{S_k'(\optx)(x - \optx)}_{Z_k}^2 \ge
        \frac{
            C
            +
            2\epsilon
        }{(1-\beta)^2}
        \norm{x - \optx}_{X_k}^2
    \intertext{for $C$ defined in \eqref{eq:dataterm:C}. Then, for $B$ defined in \eqref{eq:dataterm:B} and $D \defeq \frac{1}{2}B + \frac{1}{2\beta} (S'_{\max})^2$,}
        \label{eq:dataterm:final:claim:i}
        \iprod{\grad E_k(z)}{x - \optx}_{X_k}  \ge E_k(x) - E_k(\optx) + \epsilon \norm{x - \optx}_{X_k}^2 - D\norm{x-z}_{X_k}^2
    \end{gather}

    Additionally, if for some $\tilde \epsilon \in \R$ also
    \begin{gather}
        \label{eq:dataterm:final:cond:ii}
        \norm{S'_k(\optx)(x - \optx)}_{Z_k}^2
        \ge
        2\left(\tilde\epsilon + S''_{\max}\left(\Bigl(\frac{\delta}{8} + S'_{\max}\Bigr)\delta + \norm{S_k(\optx) - b_k}_{Z_k}\right)\right)\norm{x -\optx}_{X_k}^2,
    \shortintertext{then}
        \label{eq:dataterm:final:claim:ii}
        \iprod{\grad E (z) - \grad E (\optx)}{x - \optx}_{X_k}  \ge  (\epsilon + \tilde \epsilon) \norm{x - \optx}^2_{X_k} - D\norm{x-z}^2_{X_k}.
    \end{gather}
\end{corollary}
\begin{proof}
    The boundedness of $\ResOp'$, \cref{cor:dataterm:initial}, and \eqref{eq:dataterm:final:cond:i} yield
    \[
        \begin{aligned}
            E(\optx) - E(x) + \iprod{\grad E (z)}{x - \optx}
            &
            \ge
            \frac{(1-\beta)^2}{2}\norm{\ResOp_{\optx}(x - \optx)}^2
            - \frac{1}{2\beta}\norm{\ResOp'(z)(x-z)}^2
            \\
            \MoveEqLeft[-1]
            -
            \frac{C}{2}
            \norm{x-\optx}^2
            -
            \frac{B}{2}
            \norm{x-z}^2
            \\
            &
            \ge
            \epsilon\norm{x - \optx}^2
            - \frac{1}{2\beta}(S'_{\max})^2\norm{x-z}^2
            - \Bigl(D-\frac{1}{2\beta}(S'_{\max})^2\Bigr)\norm{x-z}^2
            \\
            &
            = \epsilon\norm{x - \optx}^2 - D\norm{x-z}^2.
        \end{aligned}
    \]
    This gives \eqref{eq:dataterm:final:claim:i}.
    Subtracting $\iprod{\grad E (\optx)}{x - \optx}$ from both sides, we further obtain
    \begin{equation}
        \label{eq:dataterm:final:0}
        \iprod{\grad E (z) - \grad E (\optx)}{x - \optx}
        \ge
        E(x) - E(\optx) + \epsilon \norm{\optx - x}^2 - D\norm{x-z}^2
        -
        \iprod{\grad E (\optx)}{x - \optx}.
    \end{equation}
    Recall that by the mean value theorem
    $
        \ResOp(x) = \ResOp(\optx) + \ResOp'_a(x-\optx)
    $.
    Thus
    \[
        \norm{\ResOp(x)} = \norm{\ResOp(x) - \ResOp(\optx) + \ResOp(\optx)}
        \le
        \norm{\ResOp'_a(x-\optx)} + \norm{\ResOp(\optx)}
        \le
        S'_{\max}\norm{x-\optx} + \norm{\ResOp(\optx)}.
    \]
    This, \cref{lemma:dataterm:basic} with $z=\optx$,
    $x\in B(\optx,\delta)$, and \eqref{eq:dataterm:final:cond:ii} give
    \begin{multline*}
        E(x)-E(\optx)-\iprod{\grad E (\optx)}{x - \optx}
        = \frac{1}{2}\norm{\ResOp'_\optx(x - \optx)}^2 - \frac{1}{8}\norm{\ResOp''_a(x -\optx, x -\optx)}^2
        + \iprod{\ResOp''_a(x -\optx, x -\optx)}{\ResOp(x)}
        \\
        \ge
        \frac{1}{2}\norm{\ResOp'_\optx(x - \optx)}^2 - \frac{S''_{\max}}{8}\norm{x -\optx}^4 - S''_{\max}\norm{\ResOp(x)}\norm{ x -\optx}^2
        \\
        \ge
        \frac{1}{2}\norm{\ResOp'_\optx(x - \optx)}^2 - S''_{\max}\left(\frac{\delta^2}{8} + S'_{\max}\delta + \norm{\ResOp(\optx)}\right)\norm{x -\optx}^2
        \ge
        \tilde \epsilon\norm{x - \optx}^2.
    \end{multline*}
    Plugging this into \eqref{eq:dataterm:final:0} yields \eqref{eq:dataterm:final:claim:ii}
\end{proof}

\input{online-eit-sup}

\end{document}

%% file: symbols.tex
\newcommand{\term}{\emph}

\newcommand{\field}[1]{\mathbb{#1}}
\newcommand{\N}{\mathbb{N}}

\newcommand{\R}{\field{R}}
\newcommand{\extR}{\overline \R}

\newcommand{\norm}[1]{\|#1\|}
\newcommand{\adaptnorm}[1]{\left\|#1\right\|}

\newcommand{\abs}[1]{|#1|}
\newcommand{\adaptabs}[1]{\left|#1\right|}

\newcommand{\inv}[1]{#1^{-1}}
\newcommand{\grad}{\nabla}
\newcommand{\freevar}{\,\boldsymbol\cdot\,}

\newcommand{\Union}\bigcup
\newcommand{\Isect}\bigcap
\newcommand{\union}\cup
\newcommand{\isect}\cap
\newcommand{\bigunion}\bigcup
\newcommand{\bigisect}\bigcap

\newcommand{\defeq}{:=}

\newcommand{\downto}{\searrow}
\newcommand{\upto}{\nearrow}
\newcommand{\subdiff}{\partial}

\DeclareMathOperator*{\argmin}{arg\,min}

\DeclareMathOperator{\diag}{diag}

\def \uminusSym{\setbox0=\hbox{$\cup$}\rlap{\hbox 
        to\wd0{\hss\raise0.5ex\hbox{$\scriptscriptstyle{-}$}\hss}}\box0}
    
\newcommand{\iprod}[2]{\langle #1,#2\rangle}

\newcommand{\Bigiprod}[2]{\Bigl\langle #1,#2\Bigr\rangle}
\def\llangle{\langle\kern-3pt\langle}
\def\rrangle{\rangle\kern-3pt\rangle}

\def \weaktostarSym{\setbox0=\hbox{$\rightharpoonup$}\rlap{\hbox 
        to\wd0{\hss\raise1ex\hbox{$\scriptscriptstyle{*\,}$}\hss}}\box0}
    
\def\linear{\mathbb{L}}

\def\extR{\overline \R}

\def\opt#1{\bar #1}
\def\realopt#1{\hat #1}
\def\this#1{#1^k}
\def\nexxt#1{#1^{k+1}}
\def\overnext#1{\bar #1^{k+1}}

\def\optu{{\opt{u}}}
\def\optx{{\opt{x}}}

\def\opty{{\opt{y}}}

\def\realoptu{{\realopt{u}}}

\def\realoptx{{\realopt{x}}}

\def\realopty{{\realopt{y}}}

\def\nextx{\nexxt{x}}

\def\nexty{\nexxt{y}}

\def\thisu{\this{u}}
\def\thisx{\this{x}}

\def\thisy{\this{y}}

\def\overnextu{\overnext{u}}

\def\tauTest{\phi}

\def\sigmaTest{\psi}

\def\GenGap{\mathcal{G}}

\newcommand{\Precond}{M}

\def\prev#1{#1^{k-1}}
\def\prevu{\prev{u}}

\DeclareMathOperator{\prox}{prox}

\def\d{\,d}

\newcommand{\fakenorm}[1]{\llbracket #1 \rrbracket}

\def\infconv{\mathop{\Box}}

\let\phi=\varphi
\let\epsilon=\varepsilon

\def\ProductSpace{U}
\def\PpredictConstr{\mathcal{X}}
\def\DpredictConstr{\mathcal{Y}}
\def\PDpredictConstr{\mathcal{U}}

\def\Id{\mathop{\mathrm{Id}}}

\DeclareMathOperator{\dynregret}{dynreg}

\def\EITmeas{\mathscr{I}}

\def\dif{\,\text{d}}
\def\BoundMeas{s}

\def\Lmeas{\xi}
\def\conductivity{x}

\def\Hs{\mathcal{H}}

\newcommand{\EkGrowthMono}[1][k]{{\hat\gamma}_{E,#1}}
\newcommand{\EkLossMono}[1][k]{{\hat\lambda}_{E,#1}}
\newcommand{\EkLossGlobal}[1][k]{{\bar\lambda}_{E,#1}}
\newcommand{\EkGrowthGlobal}[1][k]{{\bar\gamma}_{E,#1}}
\newcommand{\EkLoss}[1][k]{\lambda_{E,#1}}
\newcommand{\EkGrowth}[1][k]{\gamma_{E,#1}}
\def\Err{e}
\def\ErrMono{\hat e}
\def\ErrGlobal{\bar e}

\def\GrowthOpMono{\hat\Gamma}
\def\LossOpMono{\hat\Omega}
\def\GrowthOp{\Gamma}
\def\LossOp{\Omega}
\def\GrowthOpGlobal{\bar\Gamma}
\def\LossOpGlobal{\bar\Omega}

\def\criticalprox{\hat r}

\def\primalpredict{\breve x}
\def\dualpredict{\breve y}
\def\pdpredict{\breve u}

\def\YoungCoeffIII{\kappa_k}

\def\SolOp{S}
\def\ResOp{R}
\def\WOp{\Sigma^{-1/2}}

\def\snabla{\grad_\xi}
\def\ifempty#1{\def\temp{#1}\ifx\temp\empty}
\newcommand{\CS}[1]{w_{#1}}
\newcommand{\dCS}[2]{w'_{#1,\ifempty{#2}\freevar\else#2\fi}}
\newcommand{\ddCS}[3]{w''_{#1,\ifempty{#2}\freevar\else#2\fi,\ifempty{#3}\freevar\else#3\fi}}

\newcommand{\CSf}[1]{v_{#1}}
\newcommand{\dCSf}[2]{v'_{#1,\ifempty{#2}\freevar\else#2\fi}}
\newcommand{\ddCSf}[3]{v''_{#1,\ifempty{#2}\freevar\else#2\fi,\ifempty{#3}\freevar\else#3\fi}}

\newcommand{\CSb}[1]{{V_{#1}}}

\def\Perturbr{\Delta}
\def\EkLipCoeffCoeff{\kappa_k}

\def\YoungCoeff{\xi}

\def\pNormCoeff{\theta_k}
\def\deltaCoeff{\theta}

\def\pxrealoptxbound{\delta}
\def\xrealoptxbound{\delta}

\def\xoptxbound{\delta}
\def\pxoptxbound{\delta}

\def\estgrad{\widetilde{\grad E}}

%% file: symbols-texonly.tex
\renewrobustcmd{\downto}{{{\mathchoice%
            {\rotatebox[origin=c]{-20}{$\to$}}
            {\rotatebox[origin=c]{-20}{$\to$}}
            {\rotatebox[origin=c]{-20}{\scalebox{0.75}{$\to$}}}
            {\rotatebox[origin=c]{-20}{\scalebox{0.6}{$\to$}}}
}}}

\renewrobustcmd{\upto}{{{\mathchoice%
            {\rotatebox[origin=c]{20}{$\to$}}
            {\rotatebox[origin=c]{20}{$\to$}}
            {\rotatebox[origin=c]{20}{\scalebox{0.75}{$\to$}}}
            {\rotatebox[origin=c]{20}{\scalebox{0.6}{$\to$}}}
}}}

%% file: online-eit-macros.tex
\newcommand{\imagearray}[1]{
        \hspace{1.55cm} True \hspace{1.5cm} No prediction\hspace{0.95cm} Primal only \hspace{1.3cm} Greedy \hspace{1.2cm} Affine\raggedright
            \\
            \centering
            \rotatebox{90}{\hspace{0.9cm}Frame 1}\hspace{0.15cm}\includegraphics[width=0.19\linewidth, trim={25pt 25pt 25pt 25pt},clip]{img/#1/true/conductivity_1}%
            \includegraphics[width=0.19\linewidth, trim={25pt 25pt 25pt 25pt},clip]{img/#1/NP_s1.0_t0.85/reco_conductivity_1}%
            \includegraphics[width=0.19\linewidth, trim={25pt 25pt 25pt 25pt},clip]{img/#1/PP_s1.0_t0.85/reco_conductivity_1}%
            \includegraphics[width=0.19\linewidth, trim={25pt 25pt 25pt 25pt},clip]{img/#1/PDP1_s1.0_t0.85/reco_conductivity_1}%
            \includegraphics[width=0.19\linewidth, trim={25pt 25pt 25pt 25pt},clip]{img/#1/PDP2_s1.0_t0.85/reco_conductivity_1}%
            \\%
            \rotatebox{90}{\hspace{0.7cm}Frame 80}\hspace{0.15cm}\includegraphics[width=0.19\linewidth, trim={25pt 25pt 25pt 25pt},clip]{img/#1/true/conductivity_2}%
            \includegraphics[width=0.19\linewidth, trim={25pt 25pt 25pt 25pt},clip]{img/#1/NP_s1.0_t0.85/reco_conductivity_2}%
            \includegraphics[width=0.19\linewidth, trim={25pt 25pt 25pt 25pt},clip]{img/#1/PP_s1.0_t0.85/reco_conductivity_2}%
            \includegraphics[width=0.19\linewidth, trim={25pt 25pt 25pt 25pt},clip]{img/#1/PDP1_s1.0_t0.85/reco_conductivity_2}%
            \includegraphics[width=0.19\linewidth, trim={25pt 25pt 25pt 25pt},clip]{img/#1/PDP2_s1.0_t0.85/reco_conductivity_2}%
            \\%
            \rotatebox{90}{\hspace{0.6cm}Frame 160}\hspace{0.15cm}\includegraphics[width=0.19\linewidth, trim={25pt 25pt 25pt 25pt},clip]{img/#1/true/conductivity_3}%
            \includegraphics[width=0.19\linewidth, trim={25pt 25pt 25pt 25pt},clip]{img/#1/NP_s1.0_t0.85/reco_conductivity_3}%
            \includegraphics[width=0.19\linewidth, trim={25pt 25pt 25pt 25pt},clip]{img/#1/PP_s1.0_t0.85/reco_conductivity_3}%
            \includegraphics[width=0.19\linewidth, trim={25pt 25pt 25pt 25pt},clip]{img/#1/PDP1_s1.0_t0.85/reco_conductivity_3}%
            \includegraphics[width=0.19\linewidth, trim={25pt 25pt 25pt 25pt},clip]{img/#1/PDP2_s1.0_t0.85/reco_conductivity_3}%
            \\%
            \rotatebox{90}{\hspace{0.6cm}Frame 240}\hspace{0.15cm}\includegraphics[width=0.19\linewidth, trim={25pt 25pt 25pt 25pt},clip]{img/#1/true/conductivity_4}%
            \includegraphics[width=0.19\linewidth, trim={25pt 25pt 25pt 25pt},clip]{img/#1/NP_s1.0_t0.85/reco_conductivity_4}%
            \includegraphics[width=0.19\linewidth, trim={25pt 25pt 25pt 25pt},clip]{img/#1/PP_s1.0_t0.85/reco_conductivity_4}%
            \includegraphics[width=0.19\linewidth, trim={25pt 25pt 25pt 25pt},clip]{img/#1/PDP1_s1.0_t0.85/reco_conductivity_4}%
            \includegraphics[width=0.19\linewidth, trim={25pt 25pt 25pt 25pt},clip]{img/#1/PDP2_s1.0_t0.85/reco_conductivity_4}%
            \\%
            \rotatebox{90}{\hspace{0.6cm}Frame 320}\hspace{0.15cm}\includegraphics[width=0.19\linewidth, trim={25pt 25pt 25pt 25pt},clip]{img/#1/true/conductivity_5}%
            \includegraphics[width=0.19\linewidth, trim={25pt 25pt 25pt 25pt},clip]{img/#1/NP_s1.0_t0.85/reco_conductivity_5}%
            \includegraphics[width=0.19\linewidth, trim={25pt 25pt 25pt 25pt},clip]{img/#1/PP_s1.0_t0.85/reco_conductivity_5}%
            \includegraphics[width=0.19\linewidth, trim={25pt 25pt 25pt 25pt},clip]{img/#1/PDP1_s1.0_t0.85/reco_conductivity_5}%
            \includegraphics[width=0.19\linewidth, trim={25pt 25pt 25pt 25pt},clip]{img/#1/PDP2_s1.0_t0.85/reco_conductivity_5}%
            \\%
            \rotatebox{90}{\hspace{0.6cm}Frame 400}\hspace{0.15cm}\includegraphics[width=0.19\linewidth, trim={25pt 25pt 25pt 25pt},clip]{img/#1/true/conductivity_6}%
            \includegraphics[width=0.19\linewidth, trim={25pt 25pt 25pt 25pt},clip]{img/#1/NP_s1.0_t0.85/reco_conductivity_6}%
            \includegraphics[width=0.19\linewidth, trim={25pt 25pt 25pt 25pt},clip]{img/#1/PP_s1.0_t0.85/reco_conductivity_6}%
            \includegraphics[width=0.19\linewidth, trim={25pt 25pt 25pt 25pt},clip]{img/#1/PDP1_s1.0_t0.85/reco_conductivity_6}%
            \includegraphics[width=0.19\linewidth, trim={25pt 25pt 25pt 25pt},clip]{img/#1/PDP2_s1.0_t0.85/reco_conductivity_6}%
            \\%
            \includegraphics[width=0.95\linewidth]{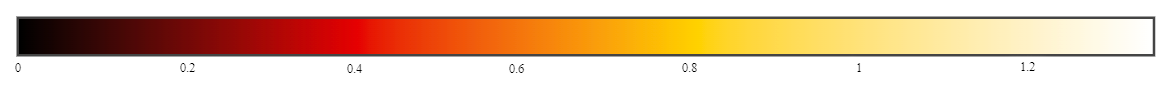}%
}

\newcommand{\imagearrayl}[1]{
    \hspace{1.55cm} True \hspace{1.5cm} No prediction\hspace{0.95cm} Primal only \hspace{1.3cm} Greedy \hspace{1.2cm} Affine\raggedright
        \\
        \centering
        \rotatebox{90}{\hspace{0.9cm}Frame 1}\hspace{0.15cm}\includegraphics[width=0.19\linewidth, trim={25pt 25pt 25pt 25pt},clip]{img/#1/true/conductivity_1}%
        \includegraphics[width=0.19\linewidth, trim={25pt 25pt 25pt 25pt},clip]{img/#1/NP_s1.0_t0.85/reco_conductivity_1}%
        \includegraphics[width=0.19\linewidth, trim={25pt 25pt 25pt 25pt},clip]{img/#1/PP_s1.0_t0.85/reco_conductivity_1}%
        \includegraphics[width=0.19\linewidth, trim={25pt 25pt 25pt 25pt},clip]{img/#1/PDP1_s1.0_t0.85/reco_conductivity_1}%
        \includegraphics[width=0.19\linewidth, trim={25pt 25pt 25pt 25pt},clip]{img/#1/PDP2_s1.0_t0.85/reco_conductivity_1}%
        \\%
        \rotatebox{90}{\hspace{0.7cm}Frame 400}\hspace{0.15cm}\includegraphics[width=0.19\linewidth, trim={25pt 25pt 25pt 25pt},clip]{img/#1/true/conductivity_2}%
        \includegraphics[width=0.19\linewidth, trim={25pt 25pt 25pt 25pt},clip]{img/#1/NP_s1.0_t0.85/reco_conductivity_2}%
        \includegraphics[width=0.19\linewidth, trim={25pt 25pt 25pt 25pt},clip]{img/#1/PP_s1.0_t0.85/reco_conductivity_2}%
        \includegraphics[width=0.19\linewidth, trim={25pt 25pt 25pt 25pt},clip]{img/#1/PDP1_s1.0_t0.85/reco_conductivity_2}%
        \includegraphics[width=0.19\linewidth, trim={25pt 25pt 25pt 25pt},clip]{img/#1/PDP2_s1.0_t0.85/reco_conductivity_2}%
        \\%
        \rotatebox{90}{\hspace{0.6cm}Frame 800}\hspace{0.15cm}\includegraphics[width=0.19\linewidth, trim={25pt 25pt 25pt 25pt},clip]{img/#1/true/conductivity_3}%
        \includegraphics[width=0.19\linewidth, trim={25pt 25pt 25pt 25pt},clip]{img/#1/NP_s1.0_t0.85/reco_conductivity_3}%
        \includegraphics[width=0.19\linewidth, trim={25pt 25pt 25pt 25pt},clip]{img/#1/PP_s1.0_t0.85/reco_conductivity_3}%
        \includegraphics[width=0.19\linewidth, trim={25pt 25pt 25pt 25pt},clip]{img/#1/PDP1_s1.0_t0.85/reco_conductivity_3}%
        \includegraphics[width=0.19\linewidth, trim={25pt 25pt 25pt 25pt},clip]{img/#1/PDP2_s1.0_t0.85/reco_conductivity_3}%
        \\%
        \rotatebox{90}{\hspace{0.6cm}Frame 1200}\hspace{0.15cm}\includegraphics[width=0.19\linewidth, trim={25pt 25pt 25pt 25pt},clip]{img/#1/true/conductivity_4}%
        \includegraphics[width=0.19\linewidth, trim={25pt 25pt 25pt 25pt},clip]{img/#1/NP_s1.0_t0.85/reco_conductivity_4}%
        \includegraphics[width=0.19\linewidth, trim={25pt 25pt 25pt 25pt},clip]{img/#1/PP_s1.0_t0.85/reco_conductivity_4}%
        \includegraphics[width=0.19\linewidth, trim={25pt 25pt 25pt 25pt},clip]{img/#1/PDP1_s1.0_t0.85/reco_conductivity_4}%
        \includegraphics[width=0.19\linewidth, trim={25pt 25pt 25pt 25pt},clip]{img/#1/PDP2_s1.0_t0.85/reco_conductivity_4}%
        \\%
        \rotatebox{90}{\hspace{0.6cm}Frame 1600}\hspace{0.15cm}\includegraphics[width=0.19\linewidth, trim={25pt 25pt 25pt 25pt},clip]{img/#1/true/conductivity_5}%
        \includegraphics[width=0.19\linewidth, trim={25pt 25pt 25pt 25pt},clip]{img/#1/NP_s1.0_t0.85/reco_conductivity_5}%
        \includegraphics[width=0.19\linewidth, trim={25pt 25pt 25pt 25pt},clip]{img/#1/PP_s1.0_t0.85/reco_conductivity_5}%
        \includegraphics[width=0.19\linewidth, trim={25pt 25pt 25pt 25pt},clip]{img/#1/PDP1_s1.0_t0.85/reco_conductivity_5}%
        \includegraphics[width=0.19\linewidth, trim={25pt 25pt 25pt 25pt},clip]{img/#1/PDP2_s1.0_t0.85/reco_conductivity_5}%
        \\%
        \rotatebox{90}{\hspace{0.6cm}Frame 2000}\hspace{0.15cm}\includegraphics[width=0.19\linewidth, trim={25pt 25pt 25pt 25pt},clip]{img/#1/true/conductivity_6}%
        \includegraphics[width=0.19\linewidth, trim={25pt 25pt 25pt 25pt},clip]{img/#1/NP_s1.0_t0.85/reco_conductivity_6}%
        \includegraphics[width=0.19\linewidth, trim={25pt 25pt 25pt 25pt},clip]{img/#1/PP_s1.0_t0.85/reco_conductivity_6}%
        \includegraphics[width=0.19\linewidth, trim={25pt 25pt 25pt 25pt},clip]{img/#1/PDP1_s1.0_t0.85/reco_conductivity_6}%
        \includegraphics[width=0.19\linewidth, trim={25pt 25pt 25pt 25pt},clip]{img/#1/PDP2_s1.0_t0.85/reco_conductivity_6}%
        \\%
        \includegraphics[width=0.95\linewidth]{img/colorbar_vert}%
}

 \pgfplotsset{
    zoomed/.code n args={3}{
        \pgfmathsetmacro{\YAmin}{#1}
        \pgfmathsetmacro{\YAzoom}{#2}
        \pgfmathsetmacro{\YAmax}{#3}
        \pgfplotsset{
            y coord trafo/.code={%
                \pgfmathifthenelse{##1 >= \YAzoom}
                                  {((abs(##1-\YAzoom)/(\YAmax-\YAzoom)))^(1/3)}
                                  {-(abs((##1-\YAzoom)/(\YAmin-\YAzoom)))^(1/3)}

            },
            y coord inv trafo/.code={%
                \pgfmathifthenelse{##1 >= 0}
                                  {(##1^3)*(\YAmax-\YAzoom)+\YAzoom}
                                  {(-##1^3)*(\YAmin-\YAzoom)+\YAzoom}
            },
            extra y ticks = {\YAzoom},
            y axis line style= { draw=none },
            yticklabel style = {xshift=-2ex},
            ylabel style = {yshift=2ex}, 
            ymax = \YAmax,
            ymin = \YAmin,
        }
    }
}

%% file: img/relvalue.tikz
\begin{tikzpicture}
    \begin{axis}[%
        width=\linewidth,
        height=0.5\linewidth,
        ymode = log,
        scaled x ticks=false,
        x tick label style={/pgf/number format/fixed, /pgf/number format/set thousands separator={\,}},
        xminorticks=true,
        minor x tick num=1,
        yminorticks=true,
        minor y tick num=3,
        axis x line*=bottom,
        axis y line*=left,
        legend style={legend pos=north east,inner sep=0pt,outer sep=0pt,legend cell align=left,align=left,draw=none,fill=none,font=\scriptsize},
        ]
        \addplot [color=Set2-A, line width=0.5pt] table [x=Time, y=Jn] {\resNP};
        \addlegendentry{No prediction}
        \addplot [color=Set2-B, dashed, line width=1pt] table [x=Time, y=Jn] {\resPP};
        \addlegendentry{Primal Only}
        \addplot [color=Set2-C, dotted, line width=1.5pt] table [x=Time, y=Jn] {\resPDPone};
        \addlegendentry{Greedy}
        \addplot [color=Set2-D, line width=1pt] table [x=Time, y=Jn] {\resPDPtwo};
        \addlegendentry{Affine}
    \end{axis}
\end{tikzpicture}

%% file: img/relerror.tikz
\begin{tikzpicture}
    \begin{axis}[%
        width=\linewidth,
        height=0.5\linewidth,
        scaled x ticks=false,
        x tick label style={/pgf/number format/fixed, /pgf/number format/set thousands separator={\,}},
        y tick label style={/pgf/number format/fixed, /pgf/number format/set thousands separator={\,}},
        xminorticks=true,
        minor x tick num=1,
        yminorticks=true,
        minor y tick num=3,
        axis x line*=bottom,
        axis y line*=left,
        legend style={legend pos=north east,inner sep=0pt,outer sep=0pt,legend cell align=left,align=left,draw=none,fill=none,font=\scriptsize},
        ignore legend,
        ]
        \addplot [color=Set2-A, line width=0.5pt] table [y=value] {\resNP};
        \addlegendentry{No Prediction}
        \addplot [color=Set2-B, dashed, line width=1pt] table [y=value] {\resPP};
        \addlegendentry{Primal Only}
        \addplot [color=Set2-C, dotted, line width=1.5pt] table [y=value] {\resPDPone};
        \addlegendentry{Greedy}
        \addplot [color=Set2-D, line width=1pt] table [y=value] {\resPDPtwo};
        \addlegendentry{Affine}
    \end{axis}
\end{tikzpicture}

%% file: online-eit-sup.tex
\section{Implementation details}
\label{app:numerics}

We now elaborate on the background gradient approximation scheme and the step length parameters choices used in the numerical experiments.

\subsection{Background solution of PDEs}
\label{app:numerics:pde}

Formally, the gradient of the approximation scheme of \cref{sssec:bgprocessing} takes
\[
    \estgrad(\thisx) \defeq \grad S_j(x^j)^*\left(S_j(x^j) - b_k\right),
\]
for some past time index $0 \le j \le k$.
As we keep the excitation potentials fixed at every frame, $U^k = U^j$ for any $j,k \ge 0$, hence $I^k(x) = I^j(x)$ and $\nabla \tilde S_k(x) = \Sigma^{-1/2} \grad \tilde I^k(x) = \Sigma^{-1/2}  \grad I^j(x) = \nabla \tilde S_j(x)$ for all $x$, meaning that $S_k = S_j$.

\begin{lemma}\label{lemma:gradient:error}
    Suppose $0 < x_m < x_M < \infty$, $\PpredictConstr_{k}$ is finite dimensional, and that $\Omega \subset \R^d$ is a Lipschitz domain. Further, define $E_k$ by \cref{eq:eit:Ek}, and suppose that $S_k = S_j$ and $Y_k = Y_j$ for any $k,j \in  \N$, and that $E_k$, $x^k$, and $\bar x^k$ satisfy the assumptions of \cref{cor:bt:Ebounds}.
    Then $E_k(x)$ and $\estgrad(\thisx)$ satisfy \cref{ass:pd:main-local} with error terms
    \[
        \bar e_k = e_k = \hat e_k = (S'_{\max})^2(S_{\max} + \norm{b_k}_{Y_k})\delta\norm{x^j - \thisx}_{X_k}.
    \]
\end{lemma}

\begin{proof}
    Since $S_k$ is differentiable by \cref{lemma:eit:frechet-first}, so is $E_k$.
    Let $\check x = x^j$ for some $j \le k$. Since $S_k = S_j$ and $Y_k = Y_k$,
    \[
        \begin{aligned}
        \iprod{\nabla E_j(x^j)}{\thisx - \optx^k}_{X_j}
        &=
        \iprod{S_j'(x^j)(\thisx - \optx^k)}{S_j(x^j) - b_k}_{Y_j}
        \\
        &
        = \iprod{S_k'(x^j)(\thisx - \optx^k)}{S_k(x^j) - b_k}_{Y_k}
        = \iprod{\nabla E_k(x^j)}{\thisx - \optx^k}_{X_k}.
        \end{aligned}
    \]
    Thus for a $z \in \bar B(\thisx,\norm{x^j - \thisx})$, by the mean value theorem,
    \[
        \iprod{\grad E_k(x^j)}{\thisx - \optx^k}_{X_k}
        =
        \iprod{\grad E_k(\this\primalpredict)}{\thisx - \optx^k}_{X_k}
        + \iprod{\grad E_k'(z)(x^j - \thisx)}{\thisx - \optx^k}_{X_k}.
    \]
    Since
    $\thisx \in B(\this \optx,\delta)$,
    \[
        \begin{aligned}
            \iprod{\grad E_k'(z)(x^j - \thisx)}{\thisx - \optx^k}_{X_k}
            &
            = \iprod{S_k''(z)(\thisx - \optx^k)(x^j - \thisx)}{S_k(z) - b_k}_{Y_k}
            \\
            \MoveEqLeft[-1]
            + \iprod{S_k'(z)(\thisx - \optx^k)}{S_k'(z)(x^j - \thisx)}_{Y_k}
            \\
            \MoveEqLeft[11]
            \le
            S''_{\max}(S_{\max} + \norm{b_k}_{Y_k})\norm{\thisx - \optx^k}_{X_k} \norm{x^j - \thisx}_{X_k}
            +
            (S'_{\max})^2\norm{\thisx - \optx^k}_{X_k} \norm{x^j - \thisx}_{X_k}
            \\
            \MoveEqLeft[11]
            \le
            S''_{\max}(S_{\max} + \norm{b_k}_{Y_k})\delta \norm{x^j - \thisx}_{X_k}
            +
            (S'_{\max})^2\delta\norm{x^j - \thisx}_{X_k}
            =: \Err_k
        \end{aligned}
    \]
    Now given that $S_k$ satisfies the assumptions of \cref{cor:bt:Ebounds}, we have
    \begin{multline*}
        E_k(\thisx) - E_k(\optx^k) + \theta \norm{\thisx - \optx^k}_X^2 - D\norm{\thisx-\this\primalpredict}_X^2
        \le
        \iprod{\grad E_k(\this\primalpredict)}{\thisx - \optx^k}_{X_k}
        \\
        =
        \iprod{\grad E_k(\this\primalpredict)}{\thisx - \optx^k}_{X_k}
        + \iprod{\grad E_k'(z)(x^j - \thisx)}{\thisx - \optx^k}_{X_k}
        \le
        \iprod{\grad E_k(\this\primalpredict)}{\thisx - \optx^k}_{X_k} + \Err_k.
    \end{multline*}
    Thus \cref{ass:pd:main-local} holds as claimed.
\end{proof}

The error term $\Err_k$ (and likewise $\bar \Err_k$ and $\hat \Err_k$) depends on the measurements $b_k$ and on the linearisation lag through the term ${\thisx - \optx^k}$. The exact values of $S'_{\max}$ and $S''_{\max}$ are given in \cref{ssec:tpi} and it is easily verified from \eqref{eq:eit:boundedness} and \eqref{eq:eit:Ek} that $S_{\max} = N_2 \norm{\WOp}_2 \max_{j,k} (C_2\norm{U^{j,k}}_{2} + \norm{\mathcal{I}^{j,k}}_2)$.

\subsection{Step length parameter choice}
\label{sup:step:parameters}

The reasoning behind the choice of $\tau$ and $\sigma$ is as follows: We set $\kappa_k \equiv \kappa \defeq 0.15$ and assume that $\frac{1}{2}\EkLoss$, $\EkLossGlobal$, and $\EkLossMono$ are all bounded by $L_{\grad E}$, where $L_{\grad E}$ satisfies $L_{\grad E} \leq \frac{1}{2} \lVert\WOp\grad I(\check x)(\WOp\grad I(\check x))^*\rVert$. This norm is computationally efficient, with a manageable size of $240 \times 240$. Under these assumptions, the first term in \eqref{eq:pd:primaltestcond-positivity-local} becomes
\[
    0.85 >
    \max \left\{
            \EkLoss,
            -2(\gamma_k + \EkGrowthGlobal),
            \frac{2}{1-\EkLipCoeffCoeff}\EkLossGlobal,
            \frac{2}{1-\EkLipCoeffCoeff}\EkLossMono
    \right\}
\]
From numerical experiments, we determined that $\norm{K_k} \approx 16 \cdot 10^{-6}$ and that, in all cases, $2L_{\grad E} \le 10^{4}$. Consequently, the choice of $\tau$ and $\sigma$ ensures that $0.85 \cdot 10^4 \cdot 16 \cdot 10^{-6} < 0.15$, satisfying the condition in \eqref{eq:pd:primaltestcond-positivity-local}.

The approximation of $\EkLoss$, $\EkLossGlobal$, and $\EkLossMono$ is heuristic, derived from the convex static case where all these terms would equal $L_{\grad E}$. Numerically, we observed that the algorithm typically diverged when $\kappa < 0.15$, indicating the reasonability of the approximation.